\documentclass[11pt,letterpaper]{article}

\usepackage[utf8]{inputenc}
\usepackage{comment}
\usepackage{microtype}
\usepackage{graphicx}
\usepackage{subfig}
\usepackage{booktabs} % for professional tables
\usepackage{scalerel}
\usepackage[margin=1in]{geometry}
\usepackage[style=alphabetic, maxbibnames=15, maxcitenames=15, natbib=true,
  maxalphanames=10, backend=biber, sorting=nty, backref=true]{biblatex}
\DefineBibliographyStrings{english}{backrefpage = {page},backrefpages = {pages},}

\usepackage{amssymb}
\usepackage{amsthm}
\usepackage{amsmath}
\usepackage{mathtools}

\usepackage{nicefrac}

\usepackage{algorithmic}
\usepackage{algorithm}
\usepackage[shortlabels]{enumitem}

\usepackage{url}
\usepackage{float}

\usepackage{pifont}
\usepackage{hhline}
\usepackage{multirow}

\usepackage{graphicx,wrapfig}

\def\grad{\nabla}

\def\bb{\mathbf{b}}

\def\bd{\mathbf{d}}

\def\bs{\mathbf{s}}

\def\bv{\mathbf{v}}
\def\bw{\mathbf{w}}
\def\bx{\mathbf{x}}  %{\mbox{\boldmath $\lambda$}}
\def\by{\mathbf{y}}

\def\bA{\mathbf{A}}

\def\bC{\mathbf{C}}

\def\bH{\mathbf{H}}

\def\bJ{\mathbf{J}}

\def\bL{\mathbf{L}}
\def\bM{\mathbf{M}}

\def\bY{\mathbf{Y}}
\def\bX{\mathbf{X}}
\def\bV{\mathbf{V}}

\def\cD{\mathcal{D}}

\def\cG{\mathcal{G}}

\def\cO{\mathcal{O}}
\def\cP{\mathcal{P}}

\def\cR{\mathcal{R}}

\def\cX{\mathcal{X}}

\def\mE{\mathbb{E}}

\def\smskip{\smallskip}

\def\texitem#1{\par\smskip\noindent\hangindent 25pt
               \hbox to 25pt {\hss #1 ~}\ignorespaces}

\def\abs#1{\left|#1\right|}
\def\norm#1{\left\|#1\right\|}

\newcommand{\BEAS}{\begin{eqnarray*}}
\newcommand{\EEAS}{\end{eqnarray*}}
\newcommand{\BEA}{\begin{eqnarray}}
\newcommand{\EEA}{\end{eqnarray}}
\newcommand{\BEQ}{\begin{eqnarray}}
\newcommand{\EEQ}{\end{eqnarray}}
\newcommand{\BIT}{\begin{itemize}}
\newcommand{\EIT}{\end{itemize}}
\newcommand{\BNUM}{\begin{enumerate}}
\newcommand{\ENUM}{\end{enumerate}}

\newcommand{\BA}{\begin{array}}
\newcommand{\EA}{\end{array}}

\newcommand{\reals}{\mathbb{R}}

\DeclareMathOperator*{\argmin}{\arg\!\min}
\DeclareMathOperator*{\argmax}{\arg\!\max}
\def\blue#1{\textcolor{blue}{#1}}

\newif\ifpagenumbering
\pagenumberingtrue

\pagenumberingfalse

\newcommand{\na}[1]{\textcolor{black}{#1}}
\newcommand{\ey}[1]{\textcolor{black}{#1}}

\newtheorem{assumption}{Assumption}
\newtheorem{definition}{Definition}
\newtheorem{theorem}{Theorem}
\newtheorem{lemma}[theorem]{Lemma}

\newtheorem{corollary}[theorem]{Corollary}
\theoremstyle{remark}
\newtheorem{remark}{Remark}

\numberwithin{assumption}{section}
\numberwithin{definition}{section}
\numberwithin{theorem}{section}
\numberwithin{remark}{section}

\usepackage{hyperref}
\hypersetup{colorlinks,citecolor=blue,linktocpage,breaklinks=true}

\begin{document}

\title{An Inexact Conditional Gradient Method for
Constrained Bilevel Optimization
\vspace{5mm}
}

\author{Nazanin Abolfazli\thanks{Department of Systems and Industrial Engineering, The University of Arizona, Tucson, AZ, USA \qquad\{nazaninabolfazli@arizona.edu, erfany@arizona.edu\}}
\quad Ruichen Jiang\thanks{Department of Electrical and Computer Engineering, The University of Texas at Austin, Austin, TX, USA \qquad\{rjiang@utexas.edu, mokhtari@austin.utexas.edu\}} 
\quad Aryan Mokhtari$^\dagger$
\quad Erfan Yazdandoost Hamedani$^*$
\vspace{5mm}
}

%\date{}

\maketitle
\begin{abstract}

Bilevel optimization is an important class of optimization problems where one optimization problem is nested within another. %This framework is widely used in machine learning problems, including meta-learning, data hyper-cleaning, and matrix completion with denoising. 
%In this paper, we focus on a bilevel optimization problem with a strongly convex lower-level problem and a smooth upper-level objective function over a compact and convex constraint set. 
While various methods have emerged to address unconstrained general bilevel optimization problems, there has been a noticeable gap in research when it comes to methods tailored for the constrained scenario. The few methods that do accommodate constrained problems, often exhibit slow convergence rates or demand a high computational cost per iteration.
To tackle this issue, our paper introduces a novel single-loop projection-free method employing a nested approximation technique. This innovative approach not only boasts an improved per-iteration complexity compared to existing methods but also achieves optimal convergence rate guarantees that match the best-known complexity of projection-free algorithms for solving convex constrained single-level optimization problems. In particular, when the hyper-objective function corresponding to the bilevel problem is convex, our method requires $\tilde{\cO}(\epsilon^{-1})$ iterations to find an $\epsilon$-optimal solution.  
Moreover,  when the hyper-objective function is non-convex, our method's complexity for finding an $\epsilon$-stationary point is  $\cO(\epsilon^{-2})$.  
To showcase the effectiveness of our approach, we present a series of numerical experiments that highlight its superior performance relative to state-of-the-art methods.
\end{abstract}

\section{Introduction}\label{sec:intro}

Many learning and inference problems take a \textit{hierarchical} form, where one
optimization problem is nested within another. Bilevel optimization is often used to model problems of this kind with two levels of hierarchy. In this paper, we consider the bilevel optimization problem of the form
\begin{equation}\label{eq:bilevel}
  \min_{\bx \in \mathcal{X}}~ \ell(\bx) := f(\bx,\by^*(\bx))\quad \hbox{s.t.}~  \by^*(\bx) \in \argmin_{\by\in \reals^{m}}~g(\bx,\by),
\end{equation}
where $n, m \geq 1$ are integers; $\mathcal{X}\subset \reals^n$ is a compact and convex set with diameter $D_{\cX}$, i.e., $\|\bx-\by\| \leq D_{\cX}$ for all $\bx,\by \in \cX$. Further,  $f : \mathcal{X} \times \reals^{m} \to \reals$ and
$g : \mathcal{X} \times \reals^{m} \to \reals$ are continuously differentiable functions with respect to $\bx$ and $\by$ on an open set containing $\cX$ and $\reals^m$, respectively.
Problem \eqref{eq:bilevel} involves two optimization problems following a two-level structure. The outer objective $f(\bx,\by^*(\bx))$ depends on $\bx$ both directly and also indirectly through $\by^*(\bx)$, which is a solution of the lower-level problem of minimizing another function $g$ parameterized by $\bx$. Throughout the paper, we assume that $g(\bx, \by)$ is strongly convex in $\by$, and hence $\by^*(\bx)$ is uniquely well-defined for any $\bx \in \cX$.
The application of \eqref{eq:bilevel} arises in a number of machine learning problems, such as meta-learning \citet{rajeswaran2019meta}, continual learning \citet{borsos2020coresets}, hyper-parameter optimization~\cite{franceschi2018bilevel,pedregosa2016hyperparameter}, and data hyper-cleaning \citet{shaban2019truncated}.

Numerous methods have emerged to tackle the general bilevel optimization problems described in \eqref{eq:bilevel}. For instance, previous works like \cite{hansen1992new,shi2005extended,moore2010bilevel} utilized the optimality conditions of the lower-level problem to transform the bilevel problem into a single-level constrained problem. However, this approach faces two significant challenges: $(i)$ The reduced problem becomes excessively constrained when dealing with large-scale inner problems; $(ii)$ Unless the lower-level function $g$ possesses a specific structure, such as a quadratic form, the optimality condition of the lower-level problem introduces nonconvexity into the feasible set of the reduced problem.

Recently, more efficient gradient-based methods for bilevel optimization have emerged, broadly categorized as the approximate implicit differentiation (AID) approach \cite{pedregosa2016hyperparameter,gould2016differentiating,domke2012generic,liao2018reviving,ghadimi2018approximation,lorraine2020optimizing} and the iterative differentiation (ITD) approach \cite{shaban2019truncated,maclaurin2015gradient,Franceschi_ICML18,grazzi2020iteration}. However, except for a few recent efforts {{that we discuss}}, most studies have primarily focused on analyzing asymptotic convergence, leaving room for the development of novel algorithms with guaranteed convergence rates.
%the finite-time analysis, which characterizes how fast an algorithm converges. 
%{\color{green}{not well-written}} {\color{red}{For the ITD-based strategy,\citep{ghadimi2018approximation} performed the finite-time analysis. Authors in\citep{grazzi2020iteration} gave the iteration complexity for the hypergradient computation using ITD and AID, but they did not provide the finite-time convergence for all of the algorithms. The finite-time (i.e.,
%non-asymptotic) convergence analysis for bilevel optimization has been recently studied in several works \citep{ghadimi2018approximation,hong2020two}}}

Moreover, in most previous studies, it is assumed that $\cX = \reals^n$, simplifying the optimization problem to an unconstrained one. However, several practical applications, such as meta-learning \citet{franceschi2018bilevel}, personalized federated learning \citet{fallah2020personalized}, and coreset selection \citet{borsos2020coresets}, require $\cX$ to be a strict subset of $\reals^n$. 
When dealing with such constraint sets, a common approach is to employ projection-based methods like projected gradient methods. These techniques, while widely used, entail solving nonlinear projection problems on the constraint set, which may not always be computationally feasible.
The limitations of projection-based techniques has led to the development of projection-free algorithms, such as Frank Wolfe-based methods \citet{frank1956algorithm}. Unlike projection methods, which deal with non-linear projections like $\ell_1$-norm or nuclear norm ball constraints, Frank Wolfe-based techniques involve solving linear minimization problems over $\cX$, offering lower computational costs.

In the context of bilevel optimization problems, many studies have delved into constrained scenarios. However, most existing methods primarily rely on projection-based algorithms, with limited exploration of projection-free alternatives. Unfortunately, these methods often exhibit slow convergence rates or impose high computational burdens per iteration. Notably, the rapid convergence rates observed in methods like \citet{ghadimi2018approximation} are achieved by utilizing the Hessian inverse of the lower-level function, which comes at a steep price, imposing a worst-case computational cost of $\cO(m^3)$ and limiting its applicability. To address this issue, an approximation technique for the Hessian inverse was introduced in \citet{ghadimi2018approximation} and subsequently used in studies such as \cite{hong2020two,akhtar2022projection}. This approximation technique introduces a vanishing bias as the number of inner steps (matrix-vector products) increases, and its computational cost scales with the condition number ($\kappa_g$) of the lower-level problem, leading to a per-iteration complexity of $\cO(\kappa_g m^2\log(K))$.

\textbf{Contributions}.
%{\color{red}{we have already said this stuff in the above paragraph}} In this paper, we consider a class of bilevel optimization problems  with a strongly convex lower-level problem and a smooth upper-level objective function over a compact and convex constraint set. This extends the literature, which has primarily focused on unconstrained problems. We propose a novel single-loop projection-free method that overcomes the limitations of existing approaches by offering improved per-iteration complexity and convergence guarantees. 
To address these challenges, we introduce a novel inexact projection-free method that attains optimal convergence rates in the studied scenarios. Remarkably, our approach demands just two matrix-vector products per iteration, resulting in a per-iteration complexity of $\cO(m^2)$. 
Our main idea is to simultaneously track the trajectories of the lower-level optimal solution as well as
{the solution to a time-varying quadratic optimization problem}.
% \red{a parametric quadratic programming containing Hessian inverse information using a nested approximation}. 
These estimators are calculated using one step of a gradient-type update and are used to estimate the hyper-gradient for a Frank Wolfe-type update. 
%At each iteration, we perform one step update to keep track of the hyper-gradient information and then combine it with a Frank Wolfe-type updates.
% This leads to a scheme that only requires two matrix-vector products at each iteration. 
Furthermore, existing methods work under the assumption that $\grad_y f(x,\cdot)$ is uniformly bounded which may not hold in many applications. To address this limitation, we also analyze our proposed method without the gradient boundedness assumption.
% \naz{A handful of studies, such as \citep{dagreou2022framework, li2022fully}, have leveraged this concept; however, by adopting this approach, our proposed method only requires two matrix-vector products at each iteration which presents a substantial improvement compared to the existing methods in the considered setting.} 
Our theoretical guarantees for the proposed Inexact Bilevel Conditional Gradient (IBCG) method are as follows:
\begin{itemize}
  \item When the {hyper}-objective function $\ell(x)$ is convex, and $\grad_yf(\bx,\cdot)$ is uniformly bounded for any $\bx\in\cX$, our IBCG method finds an $\epsilon$-optimal solution after $\tilde{\cO}(\kappa_g^4\epsilon^{-1})$ iterations.  If we relax the assumption of gradient boundedness to assume Lipschitz continuity, its complexity becomes $\tilde{\cO}(\kappa_g^5\epsilon^{-1})$.
%  When the upper-level function $f$ is convex and $\grad_yf(\bx,\cdot)$ is uniformly bounded for any $\bx\in\cX$, we show that our IBCG method requires $\tilde{\cO}(\kappa^4\epsilon^{-1})$ iterations to find an $\epsilon$-optimal solution. This result changes to $\tilde{\cO}(\kappa^5\epsilon^{-1})$ when the gradient boundedness assumption is relaxed and we assume the gradient is Lipschitz.
  \item When $\ell(x)$ is non-convex and $\grad_yf(\bx,\cdot)$ is bounded, IBCG requires $\cO(\kappa_g^4\epsilon^{-2})$ iterations to find an $\epsilon$-stationary point. This result changes to ${\cO}(\kappa_g^5\epsilon^{-2})$ when the gradient boundedness assumption is replaced by gradient Lipschitzness.
\end{itemize}

These results match the best-known complexity of projection-free algorithms for solving convex constrained single-level optimization problems \cite{lacoste2016convergence,mokhtari2018escape}. %{\color{red}{add reference}}
%Before proceeding, we provide some practical examples and explain the challenges of such kinds of problems.

% Projection-free algorithms for single-level stochastic optimization
% algorithms are well-known and state-of-the-art algorithms
% achieve a sample complexity of $\cO(\epsilon^{-2})$.

%In terms of potential and challenges, bilevel optimization problems are more flexible in their formulation and often outperform their single-level counterparts. This is particularly true for sparse bilevel problems with many applications, such as image reconstruction \citep{huang2021biadam}, neural network architecture search, and sparse MAML~\citep{stackelberg1952theory,bennett2008bilevel}, that have $\ell_1$ or nuclear norm constraints. Dealing with set constraints $\cX$, which are commonly solved through projection operations that might be computationally expensive (e.g. in nuclear norm constraints), is one of the main challenges in bilevel optimization. However, from the perspective of practical application, the novel projection-free algorithm put forward in this study provides a very effective method for dealing with set constraints in bilevel optimization.

\begin{table*}[t!]\scriptsize
\renewcommand{\arraystretch}{1.3}
\centering
    \caption{Summary of results for bilevel optimization with a strongly convex lower-level function. The abbreviations ``C'', ``NC'', ``PO", ``LMO" stand for ``convex'', ``non-convex'', ``projection oracle", and ``linear minimization oracle" respectively and $\kappa_g\triangleq L_g/\mu_g$.  
    %For the algorithms with $\cO(m^2 \log(K))$ per-iteration dependence, the Hessian inverse can be computed via multiple rounds of matrix-vector products.
$^{\dagger}$ We use $\emph{poly}(\kappa_g)$ as explicit dependency on $\kappa_g$ is not reported.
%\naz{$^{\P}$ Measured using distance to a fixed point with the Moreau proximal map $\hat{\bx}(z):= \argmin_{\bx \in \cX}\{\ell(\bx)+ (\rho/2)\| \bx - z \|^2\}$.}
$^*$These works primarily examined convergence rates in the stochastic setting (no deterministic result), so the results presented here are for stochastic cases.
}
\label{tab:bilevel}
     \resizebox{\textwidth}{!}{%
\begin{tabular}{ccccccc}
\hline
 & \multirow{2}{*}{\textbf{Reference}} & \multirow{2}{*}{\textbf{Oracle}} & \textbf{Function $\blue{\ell}$} & {\textbf{Assumption on}} & \multirow{2}{*}{\textbf{Overall Complexity}} &\multirow{2}{*}{\textbf{Convergence metric}} \\  
 \cline{4-4} &    &    & NC/C  & $\grad_y f(\bx,\cdot)$  &  &  \\ \hline
\multicolumn{1}{c|}{\multirow{6}{*}{\textbf{Unconstrained}}}
 % & STABLE                              & -----                                    & NC             & \multicolumn{2}{c}{$\cO(m^3)$}                                              & \multicolumn{2}{c}{$\cO(\epsilon^{-2})$} 
 % \\ \cline{2-8} 
% \multicolumn{1}{c|}{} 
& SUSTAIN$^*$ \cite{khanduri2021near}                            & -----                                    & NC      &bounded       & $\tilde{\cO}( \emph{poly}(\kappa_g) m^2 \epsilon^{-1.5})^{\dagger}$                          & $\mE\| \nabla \ell(\bx_{k^*})\|^2$                         \\ \cline{2-7} 
\multicolumn{1}{c|}{}                                                    & FSLA$^*$ \cite{li2022fully}                               & -----                                    & NC      &bounded       & $ \cO( \emph{poly}(\kappa_g) m^2 \epsilon^{-2})$                                             & $\mE\| \nabla \ell(\bx_{k^*})\|^2$                          \\ \cline{2-7} 
% \multicolumn{1}{c|}{}    & F$^{3}$SA \citep{kwon2023fully}                               & -----                                    & NC    &bounded         & $\tilde{\cO}(\emph{poly}(\kappa_g) m \epsilon^{-1.5})$                                              & $\| \nabla \ell(\bx_{k^*})\|^2$                        \\ \cline{2-7} 
\multicolumn{1}{c|}{}    & AID-BiO \cite{ji2021bilevel}                     & -----                                    & NC      &bounded       & $\cO((\kappa_g^{3.5} m^2+m\kappa_g^4) \epsilon^{-1})$       & $\| \nabla \ell(\bx_{k^*})\|^2$                    \\
\cline{2-7} 
\multicolumn{1}{c|}{}    & F$^{3}$SA \cite{kwon2023fully}                               & -----                                    & NC    &bounded         & $\tilde{\cO}(\emph{poly}(\kappa_g) m \epsilon^{-1.5})$                                              & $\| \nabla \ell(\bx_{k^*})\|^2$                        \\
\cline{2-7} 
\multicolumn{1}{c|}{}                                                    & AmIGO \cite{arbel2021amortized}                     & -----                                    & NC      &bounded      & $\cO((\kappa_g^{4} m^2+m\kappa_g^4)\epsilon^{-1})$ & $\| \nabla \ell(\bx_{k^*})\|^2$                                    \\  \cline{2-7} 
\multicolumn{1}{c|}{}                                                    & RAHGD \cite{yang2023accelerating}                     & -----                                    & NC      &bounded      & $\cO(\kappa_g^{3.25}m^2\epsilon^{-0.875})$ & $\| \nabla \ell(\bx_{k^*})\|^2$                                      \\ \hline
\multicolumn{1}{c|}{\multirow{8}{*}{\textbf{Constrained}}}   &\multirow{2}{*}{ABA \cite{ghadimi2018approximation}}                & \multirow{2}{*}{PO}                     & NC      & \multirow{2}{*}{bounded}       & $\cO(\kappa_g^{4.5} m^3 \epsilon^{-1}+\kappa_g^5m\epsilon^{-1.25})$                           &  $\| \nabla \ell(\bx_{k^*})\|^2$                    \\
\multicolumn{1}{c|}{}                                                      &                                  &                                          & C              &        & $\cO(\kappa_g^{4.5} m^3 \epsilon^{-0.5}+\kappa_g^5m\epsilon^{-0.75})$ &
                  $\ell(\bx_K)-\ell(\bx^*)$              \\ \cline{2-7}
\multicolumn{1}{c|}{} 
& \multirow{2}{*}{TTSA$^*$ \cite{hong2020two}}                                & \multirow{2}{*}{PO}                                      & NC      & \multirow{2}{*}{bounded}       & $\tilde{\cO}(\emph{poly}(\kappa_g) m^2 \epsilon^{-2.5})$                                            &$\mathbb{E}{\| \bx_{k^*}-\hbox{prox}_{\rho\ell}(\bx_{k^*}) \|^2 }$                         \\
\multicolumn{1}{c|}{}                                                    &                                     &                                          & C              &                                         & $\tilde{\cO}(\emph{poly}(\kappa_g) m^2 \epsilon^{-4})$       &$\mathbb{E}{[\ell(\bx_K)- \ell(\bx^*)]}$  \\  \cline{2-7} 
\multicolumn{1}{c|}{}                                                     & \multirow{2}{*}{SBFW$^*$ \cite{akhtar2022projection}}                & \multirow{2}{*}{LMO}                     & NC      & \multirow{2}{*}{bounded}       & $\tilde{\cO}(\emph{poly}(\kappa_g) m^2 \epsilon^{-4})$  & $\mathbb{E}{[\cG(\bx_{k^*})]}$                 \\
\multicolumn{1}{c|}{}                                                      &                                  &                                          & C              &        & $\tilde{\cO}(\emph{poly}(\kappa_g) m^2 \epsilon^{-3})$       &$\mathbb{E}{[\ell(\bx_K)- \ell(\bx^*)]}$    \\ \cline{2-7} 
\multicolumn{1}{c|}{}                                                   & \multicolumn{1}{c|}{\multirow{4}{*}{\textbf{Ours}}}      & \multirow{2}{*}{LMO}                                      & NC      & \multirow{2}{*}{Lip. cont.}       & {{$\cO(\kappa_g^5 m^2 \epsilon^{-2})$}}                                            & $\mathcal{G}(\bx_{k^*})$                         \\
\multicolumn{1}{c|}{}                                                    &    \multicolumn{1}{c|}{}        &                                          & C                    &                                          & $ \tilde{\cO}(
\kappa_g^5 m^2 \epsilon^{-1})$                        & $\ell(\bx_K)- \ell(\bx^*)$ \\ 
\cline{3-7} 
\multicolumn{1}{c|}{}                                                   &   \multicolumn{1}{c|}{}  & \multirow{2}{*}{LMO}                                      & NC      & \multirow{2}{*}{bounded}       & {{$\cO(\kappa_g^4 m^2 \epsilon^{-2})$}}                                            & $\mathcal{G}(\bx_{k^*})$                         \\
\multicolumn{1}{c|}{}                                                    &    \multicolumn{1}{c|}{}        &                                          & C                    &                                          & $ \tilde{\cO}(
\kappa_g^4 m^2 \epsilon^{-1})$                        & $\ell(\bx_K)- \ell(\bx^*)$ \\ \hline
\end{tabular}
}
\end{table*}

\textbf{Related work.} 
% Various algorithms, including constraint-based methods \citep{shi2005extended,moore2010bilevel} and gradient-based methods, have been developed for solving bilevel optimization problems.
% In recent years, gradient-based methods for dealing with problem~\eqref{eq:bilevel} have become increasingly popular including implicit differentiation~\citep{domke2012generic,pedregosa2016hyperparameter,gould2016differentiating,ghadimi2018approximation,ji2021bilevel} and iterative differentiation~\citep{Franceschi_ICML18,maclaurin2015gradient}.
In this section, we review related work on bilevel optimization with non-asymptotic guarantees; see Table~\ref{tab:bilevel} for a summary of these results. As mentioned earlier, the majority of the existing works consider unconstrained bilevel problems, i.e., $\mathcal{X} = \mathbb{R}^n$.  
In particular, \cite{khanduri2021near,li2022fully}
focused on the stochastic setting and proved an overall complexity of $\tilde{\mathcal{O}}(m^2 \epsilon^{-1.5})$ and ${\mathcal{O}}(m^2\epsilon^{-2})$, respectively. In the deterministic setting, \citet{yang2021provably} analyzed the convergence of bilevel algorithms via AID and proved a complexity of $\cO((\kappa_g^{3.5} m^2+m\kappa_g^4) \epsilon^{-1})$, where $\kappa_g$ denotes the condition number of the lower-level objective. By incorporating acceleration techniques, a recent work by \citet{yang2023accelerating} further improved the complexity to $\cO(\kappa_g^{3.25}m^2\epsilon^{-0.875})$. Moreover, to avoid expensive Hessian computation, % all the works above require access to the Hessian of the lower-level objective. To address this issue, 
\citet{kwon2023fully}  proposed a fully first-order method with a complexity of $\tilde{\cO}(m \epsilon^{-1.5})$. 

In comparison, there are relatively few works on constrained bilevel optimization problems, which is the considered setting of this paper. 
% On this basis,   
\citet{ghadimi2018approximation} presented an Accelerated Bilevel Approximation ({ABA}) method consisting of two iterative loops. In the nonconvex setting, they demonstrated that the method achieves an overall complexity of $\cO(\kappa_g^{4.5} m^3 \epsilon^{-1})$ and $\cO(\kappa_g^5 m \epsilon^{-1.25})$ in terms of the upper-level and lower-level objective values, respectively. In convex setting, they further shaved a factor of $\cO(\epsilon^{-0.5})$ from the complexities. 
%For the sake of simplicity,  authors assume that set $\cX = \reals^n$. 
% Although their convergence guarantees in both convex and non-convex settings are superior to other similar works,
However, their computational complexity is high since they must compute the Hessian inverse matrix at each iteration, incurring a per-iteration cost of~$\cO(m^3)$. 

Efforts have been made to design efficient single-loop methods aimed at lowering the per-iteration expenses. Notably, similar to our proposed IBCG method, approaches in \cite{dagreou2022framework, li2022fully} demand only two matrix-vector products per iteration. However, these works have exclusively addressed unconstrained bilevel problems.
%when the upper level function is non-convex. 
% Built upon the work of \citep{ghadimi2018approximation}, a class of double-loop approximation algorithms was proposed in \citep{yang2021provably} to iteratively
% approximate the stochastic gradient of the outer objective and
% obtained a sample complexity of $\cO(\epsilon^{-2})$ in order to achieve the $\epsilon-$stationary point.
% The double loop structure of these approaches made them impractical for large-scale problems; therefore to address this issue, various single-loop methods, involving
% simultaneous updates of lower and upper optimization variables,
% have been developed \citep{hong2020two,akhtar2022projection,kwon2023fully,chen2022single,yang2021provably}. All the mentioned methods have been developed for the unconstrained bilevel optimization problem unless the algorithm proposed in \citep{hong2020two}. 
% Built upon the work of \citep{ghadimi2018approximation},
Built upon the work of \citet{ghadimi2018approximation}, a Two-Timescale Stochastic Approximation (TTSA) algorithm has been proposed by \citet{hong2020two} for constrained bilevel problems in the stochastic setting, which is shown to achieve a complexity of $\tilde{\cO}(m^2\epsilon^{-2.5})$ and $\tilde{\cO}(m^2\epsilon^{-4})$ in nonconvex and convex settings, respectively. 
% Results for the convex case are also included in Table~\ref{tab:bilevel}. 
Concurrently, \citet{shen2023penalty} introduced a penalty-based bilevel gradient descent (PBGD) algorithm, categorizing it as a projection-based method. In their work, they focused on scenarios where the lower-level function satisfies the Polyak-Lojasiewicz condition.

It should be noted that the above methods require a projection onto the set $\cX$ at every iteration.  In contrast, our 
proposed method is projection-free and only requires access to a linear solver, which
is suitable for settings where projection is computationally costly; e.g., when $\cX$ is a nuclear-norm ball. 
A closely related study is the work by \citet{akhtar2022projection}, where they presenrted a projection-free algorithm named {SBFW} for \textit{stochastic} bilevel optimization problems. Their findings demonstrate that their method attains a complexity of $\cO(m^2\epsilon^{-4})$ for nonconvex scenarios and $\cO(m^2\epsilon^{-3})$ for convex settings, respectively.

% \rj{Along another line of research, there have been efforts in designing efficient single-loop methods in order to reduce the per-iteration cost. %\citep{khanduri2021near,li2022fully,kwon2023fully,hong2020two,akhtar2022projection,dagreou2022framework, li2022fully}.
% In particular, similar to our proposed IBCG, the methods in \citep{dagreou2022framework, li2022fully} only require two matrix-vector products per iteration,  but these works only considered unconstrained bilevel problems.}
%We note that most of the work studied bilevel optimization problems in the stochastic setting and did not provide any results for the deterministic setting motivating us to focus on this category of bilevel problems to improve the convergence and computational complexity over existing methods.
Lastly, concurrent papers like \cite{liu2020generic,sow2022constrained,chen2023bilevel} explored scenarios where the lower-level problem can possess multiple minima. Given the increased complexity of this generalized setting, their theoretical results are comparatively weaker, offering either asymptotic  guarantees or slower convergence rates.

\section{Preliminaries}\label{sec:pre}
\subsection{Motivating Examples}\label{subsec:examples}
The bilevel optimization formulation in \eqref{eq:bilevel} finds applications in various ML problems, including matrix completion \citet{yokota2017simultaneous}, meta-learning \citet{rajeswaran2019meta}, data hyper-cleaning \citet{shaban2019truncated}, hyper-parameter optimization \citet{franceschi2018bilevel}, and more. Next, we delve into two specific examples.

%{\color{green}{we have to shorten this section. }}

\textbf{Matrix Completion with Denoising}: Consider the matrix completion problem, where the objective is to recover missing items from noisy observations of a subset of the matrix's entries. Typically, in noiseless scenarios, the data matrix can be represented as a low-rank matrix, justifying the use of the nuclear norm constraint. However, in applications such as image processing and collaborative filtering, noisy observations are common, and relying solely on the nuclear norm constraint can lead to suboptimal results \cite{mcrae2021low,yokota2017simultaneous}.
To address this issue, one approach to incorporate denoising into the matrix completion problem is by formulating it as a bilevel optimization problem \citet{akhtar2022projection}, expressed as:
%%%
%The bilevel matrix completion with a denoising problem can be expressed mathematically as
\begin{align}\label{ex:Matrix_comp}
    &\min_{\|\bX \|_{*}\leq \alpha}~\frac{1}{\abs{\Omega_1}}\sum_{(i,j) \in \Omega_{1}}(\bX_{i,j} - \bY_{i,j})^2\\
    &\;\, \hbox{s.t.}\quad  \bY\in\argmin_{\bV}~\Big\{ \frac{1}{\abs{\Omega_2}}\sum_{(i,j) \in \Omega_{2}}(\bV_{i,j} - \bM_{i,j})^2  + \lambda_1 \cR(\bV) +\lambda_2\|\bX - \bV\|_F^2 \Big\} \nonumber,
\end{align}
%{\color{red}{sth is not correct}}, \blue{I think we should set $\Omega_1 = \Omega_2 = \Omega$}
where $\bM \in \reals^{n\times m}$ is the given incomplete noisy matrix. Moreover, $\Omega_1\subseteq\Omega$ and $\Omega_2\subseteq \Omega$ where $\Omega$ is the set of observable entries, and
% \blue{where $\Omega_1$ and $\Omega_2$ represent the set of available entries in upper and lower-level respectively}. 
$\cR(\bV)$ is a regularization term to induce sparsity, e.g., $\ell_1$-norm or pseudo-Huber loss, 
%$\|\bV\|_1 = \Sigma_{i,j} \abs{V_{i,j}}$ is the sum-absolute-value $\ell_1$ norm, 
$\lambda_1$ and $\lambda_2$ are regularization parameters. 
%, and $\Omega_1$ and $\Omega_2$ represents the set of available entries at outer and inner level respectively. 
%Note that the regularization over the discrepancy between $\bX$ and denoised matrix $\bY$ results in bilevel formulation~\eqref{ex:Matrix_comp}. A similar technique in deterministic settings is utilized in various other applications in machine learning and signal processing problems. 
The presence of the nuclear norm constraint poses a significant challenge in \eqref{ex:Matrix_comp}. This constraint renders the problem computationally demanding, often making projection-based algorithms impractical. Consequently, there is a compelling need to develop and employ projection-free methods to overcome these computational limitations. %reference bezanam

\textbf{Multi-Task Learning}:
Multi-Task Learning (MTL) is a machine learning paradigm that leverages insights from multiple related tasks to enhance overall performance and generalization. In essence, when faced with $T$ learning tasks, whether all or a subset are related, MTL seeks to jointly learn these tasks. This joint learning process aims to improve each task's model by tapping into the knowledge contained within all or some of the tasks. In the supervised setting, research has explored various approaches, including the Multi-task Feature Learning (MTFL) method \citet{zhang2014regularization}.
%Multi-Task Learning (MTL) represents a paradigm within machine learning that focuses on using valuable insights gained from multiple related tasks to enhance the overall performance and generalization capabilities across all these tasks. In other words, given $T$ learning tasks where all the tasks or a subset of them are related, MTL aims to learn the $T$ tasks together to improve the learning of a model for each task by using the knowledge contained in all or some of the tasks. In the supervised setting, various approaches including Multi-task Feature Learning (MTFL) method \cite{zhang2014regularization}, have been studied. 
In MTFL, the goal is to minimize the overall loss function while learning the feature covariance matrix $\Omega\in\reals^{T\times T}$ for all the tasks. Considering a model parameter $W=[w_1,\hdots,w_T]\in\reals^{d\times T}$ where $w_i\in\reals^d$ denotes the model parameter for task $i$, \citet{zhang2014regularization} formulated this problem as $\min_{W,\Omega}\{\sum_{i=1}^T L_i(w_i)+\lambda\text{Tr}(W\Omega^{-1}W) \mid \Omega \succeq 0, \text{Tr}(\Omega) \leq c\}$ for some $c>0$ where $L_i(\cdot;\cD_i)$ denotes the loss function of task $i$ given dataset $\cD_i$ and $\text{Tr}(\cdot)$ denotes the trace of a square matrix. This problem can also be written as a bilevel optimization problem. To do so, consider $\cD_i^{tr}$ and $\cD_i^{val}$ as training and validation datasets for task $i$, respectively. In this problem, the upper-level problem seeks to identify the best covariance matrix by minimizing the loss function over the validation dataset, while the lower-level problem aims to find the best model for the training set given a covariance matrix $\Omega$. Specifically, for $\lambda_1,\lambda_2>0$, the MTLF problem can be written as 
\begin{align}\label{ex:MTL_bi}
  &\min_{\substack{\Omega \succeq 0 \\ \text{Tr}(\Omega) \leq c}} \: \sum_{i=1}^{T} L_i(w_i^*(\Omega);\cD_i^{val})  \\
  &\;\, \hbox{s.t.}\quad W^*(\Omega) \in \argmin_{W} \bigg\{\sum_{i=1}^{T} L_i(w_i;\cD_i^{tr}) + \lambda_1 \left\| W \right\|^2_F + \lambda_2 \text{Tr}(W \Omega^{-1} W^T)\bigg\},\nonumber
\end{align}
which is an instance of the constrained bilevel problem~\eqref{eq:bilevel}.

\subsection{Assumptions and Definitions}

In this subsection, we discuss the definitions and assumptions required throughout the paper. We begin by discussing the assumptions on the upper-level and lower-level objective functions, respectively.
%{\bf Notations.} Let $\|\cdot\|$ be $\ell_2$ norm for vectors and \ey{spectral} norm for matrices.  
%\naz{We assume  $\mathcal{X}\subset\reals^{n}$ is convex and compact with diameter $D_{\cX}$, i.e., $\|\bx-\by\| \leq D_{\cX}$ for all $\bx,\by \in \cX$}. 
% For a multivariate function $f(\bx,\by)$, the notation $\nabla_x f(\bx,\by)$ (resp. $\nabla_y f(\bx,\by)$) refers to the partial gradient taken with respect to (w.r.t) $\bx$ (resp. $\by$). For some $\mu>0$ a function $f(\bx,\by)$ is said to be $\mu$-strongly convex in $\bx$ if $f(\bx,\by)- \frac{\mu}{2}\|\bx\|^2$ is convex in $\bx$. For some $L>0$, the map $\cA : \reals^n \to \reals^m$ is said to be $L$-Lipschitz continuous if  $\| \cA(\bx) - \cA(\by)\| \leq L \| \bx - \by\|$ for any $\bx, \by \in \reals^n$. Now, we discuss the assumptions on problem~\eqref{eq:bilevel} to specify the problem class of interest.
\begin{assumption}\label{assum:upper}
%\ey{$f(\cdot,\by)$ and $f(\bx,\cdot)$ are continuously differentiable for any $\bx\in\reals^n$ and $\by\in\reals^m$} satisfying the following conditions:
% \begin{enumerate}[(i)]
    %\item \ey{$\nabla_x f(\bx,\by)$} and \ey{$\nabla_y f(\bx,\by)$} are Lipschitz continuous w.r.t $(\bx,\by) \in \mathcal{X}\times \reals^{m}$ and with constants $L_{x}^{f}\geq 0$ and $L_{yy}^{f}\geq 0$, respectively. 
    %\item 
    $\nabla_x f(\bx,\by)$ and $\nabla_y f(\bx,\by)$ are Lipschitz continuous w.r.t $(\bx,\by) \in \mathcal{X}\times \reals^{m}$ such that for any $ \bx,\bar\bx\in\cX$ and $\by,\bar\by\in\reals^m$ we have:
    \begin{enumerate}[(i)]
        \item $\!\!\!\|\nabla_{\!x} f( \bx, \by) - \nabla_{\!x} f(\Bar{\bx}, \bar\by) \| \leq  L_{xx}^f\| \bx - \Bar{\bx}\| + L_{xy}^f\|\bar\by-\by\|$, 
        \item $\!\!\!\|\nabla_{\!y} f( \bx, \by) - \nabla_{\!}y f(\Bar{\bx}, \bar\by) \| \leq  L_{yx}^f\| \bx - \Bar{\bx}\| + L_{yy}^f\|\by-\bar\by\|$.
    \end{enumerate}
        %\rj{Is this a corollary of (i)?}
%        \item For any \ey{$(\bx,\by) \in \mathcal{X}\times \reals^{m}$, we have $\|\nabla_y f(\bx,\by)\| \leq C_{y}^{f}$}, for some $C_{y}^{f}\geq 0$.
%    \end{enumerate}
\end{assumption}
%Next, we state the assumptions required for the lower-level function. 

\begin{assumption}\label{assum:lower} 
$g(\bx,\by)$ satisfies the following conditions:
    \begin{enumerate}[(i)]
        \item $\forall\bx \in \mathcal{X}$, $g(\bx,\cdot)$ is twice continuously differentiable. Moreover, $\grad_y g(\cdot,\cdot)$ is continuously differentiable.
        %\item \ey{For any given $\by \in \mathcal{X}$, $g(\cdot,\by)$ is continuously differentiable.}
        \item $\forall\bx\in \cX$, $\nabla_y g(\bx,\cdot)$ is Lipschitz continuous with constant $L_g \geq 0 $. Moreover, $\forall \by\in \reals^m$, $\nabla_y g(\cdot,\by)$ is Lipschitz continuous with constant $C^g_{yx} \geq 0 $.
        \item $\forall\bx \in \mathcal{X}$, $g(\bx,\cdot)$ is $\mu_g$-strongly convex.
        \item $\forall\bx \in \mathcal{X}$, $\nabla^2_{yx} g(\bx,\by)\in\reals^{n\times m}$ and $\nabla^2_{yy} g(\bx,\by)$ are Lipschitz continuous with respect to $(\bx,\by) \in \mathcal{X} \times \reals^{m}$, and with constant $L_{yx}^{g} \geq 0 $ and  $L_{yy}^{g} \geq 0 $, respectively.
        % \item For any $(\bx,\by) \in \mathcal{X}\times \reals^{m}$, we have $\| \nabla_{xy} g(\bx,\by)\|\leq C_{xy}^{g}$ for some $C_{xy}^{g}>0$.
    \end{enumerate}
\end{assumption}
\begin{remark}\label{re:bounded-norm}
Considering Assumption~\ref{assum:lower}-(ii), %if $\nabla_y g(x,y)$ is Lipschitz continuous w.r.t $(x,y) \in \mathcal{X}\times \reals^{m}$, then 
we can conclude that $\| \nabla^2_{yx} g(\bx,\by)\|$ is bounded above by constant $C_{yx}^{g}\geq 0$ for any $(\bx,\by) \in \mathcal{X}\times \reals^{m}$.
\end{remark}

Next, we present essential properties of the bilevel problem in \eqref{eq:bilevel} based on the aforementioned assumptions. Firstly, through standard analysis under Assumption \ref{assum:lower}, it is evident that the optimal solution trajectory of the lower-level problem denoted as $\by^*(\bx)$, is Lipschitz continuous, as previously discussed in \cite{ghadimi2018approximation}.

Secondly, when aiming to develop a method with a guaranteed convergence, one key requirement is demonstrating the Lipschitz continuity of the gradient of the single-level objective function. In bilevel optimization literature, this often necessitates assuming the boundedness of $\grad_y f(\bx,\by)$ for any $\bx\in \cX$ and $\by\in\reals^m$, as seen in \cite{ghadimi2018approximation,yang2021provably,hong2020two,akhtar2022projection}. However, in contrast, our paper establishes that this condition is solely required for the gradient map $\grad_y f(\bx,\by)$ when restricted to the optimal trajectory of the lower-level problem. Specifically, we demonstrate that it suffices to prove the boundedness of $\grad_y f(\bx,\by^*(\bx))$ for any $\bx\in \cX$, a result derivable from the boundedness of constraint set $\cX$.

Lastly, employing the above findings, we can assert that the gradient of the single-level objective function denoted as $\grad\ell(\bx)$, is also Lipschitz continuous. This outcome constitutes a fundamental element in the subsequent convergence analysis of our proposed method in the following section. The detailed formal statements of these results are presented in the subsequent lemma.
 
%Next, we state some important properties related to the bilevel problem in \eqref{eq:bilevel} based on the assumptions above. First, a standard analysis reveals that given Assumption \ref{assum:lower}, the optimal solution trajectory of the lower-level problem, i.e., $\by^*(\bx)$, is Lipschitz continuous as mentioned in \cite{ghadimi2018approximation}. Second, one of the required properties to develop a method with a convergence guarantee is to show the Lipschitz continuity of  the gradient of the single-level objective function. In the literature of bilevel optimization, to show this result it is often required to assume boundedness of $\grad_y f(\bx,\by)$ for any $\bx\in \cX$ and $\by\in\reals^m$, e.g., see \cite{ghadimi2018approximation,ji2020provably,hong2020two,akhtar2022projection}. In contrast, in this paper, we show that this condition is only required for the gradient map $\grad_y f(\bx,\by)$ when restricted to the optimal trajectory of the lower-level problem. In particular, we demonstrate that it is sufficient to show the boundedness of $\grad_y f(\bx,\by^*(\bx))$ for any $\bx\in \cX$ which can be proved using the boundedness of constraint set $\cX$. Third and lastly, Using the above results we can show that the gradient of the single-level objective function, i.e., $\grad\ell(\bx)$, is Lipschitz continuous. This result is one of the main building blocks of the convergence analysis of our proposed method in the next section. The aforementioned results are formally stated in the following Lemma.

\begin{lemma}\label{lem:v_b}
Suppose Assumptions \ref{assum:upper} and \ref{assum:lower} hold. Then for any $\bx,\Bar{\bx} \in \mathcal{X}$,  the following results hold.\\
(I) $\|\by^*(\bx) - \by^*(\Bar{\bx})\| \leq \bL_{\by}\| \bx - \Bar{\bx}\|$,  where $\bL_{\by} \triangleq \frac{C_{yx}^g}{\mu_g}$.\\
(II) $\norm{\grad_y f(\bx,\by^*(\bx))}\leq C_y^f$, 
where $C_y^f \triangleq \Bigl(L_{yx}^f+\frac{L_{yy}^f {C_{yx}^{g}}}{\mu_g}\Bigr)D_\cX+ \|\grad_y f({\bx^*},\by^*({\bx^*}))\|$. 
% where $C_y^f \triangleq L_{yx}^f+ \frac{L_{yy}^f C_{yx}^g }{\mu_g}$.  
\\
%\item[(II)] $\|\bv(\bx ) -\bv(\Bar{\bx})\| \leq \bC_{\bv}\|\bx -\Bar{\bx}\|$,
(III) $\|\nabla \ell(\bx ) - \nabla \ell(\Bar{\bx})\| \leq \bL_{\ell}\|\bx -\Bar{\bx}\|$, where $\bL_{\ell} \triangleq L_{xx}^f +{L_{xy}^f} \bL_{\by} +C_{yx}^g \bC_{\bv} +\frac{C_{y}^{f}}{\mu_g} L_{yx}^g(1+\bL_{\by})$. 
%$\|\bv(\bx ) -\bv(\Bar{\bx})\| \leq \bC_{\bv}\|\bx -\Bar{\bx}\|, \quad \|\by^*(\bx) - \by^*(\Bar{\bx})\| \leq \bL_{\by}\| \bx - \Bar{\bx}\|,$\\
%where 
%$\bC_{\bv} \triangleq \frac{L_{yx}^{f} + L_{yy}^{f}\bL_{\by}}{\mu_g} + \frac{C_{y}^{f}L_{yy}^{g}}{\mu_g^2}(1+\bL_{\by})$, 
%$\bL_{\by} \triangleq \frac{\ey{C^g_{yx}}}{\mu_g}$, $\bL_{\ell} \triangleq L_{xx}^f +{L_{xy}^f} \bL_{\by} +C_{yx}^g \bC_{\bv} +\frac{C_{y}^{f}}{\mu_g} L_{yx}^g(1+\bL_{\by})$, and $\bC_{\bv} \triangleq \frac{L_{yx}^{f} + L_{yy}^{f}\bL_{\by}}{\mu_g} + \frac{C_{y}^{f}L_{yy}^{g}}{\mu_g^2}(1+\bL_{\by})$.
\end{lemma}

To measure the quality of the solution at each iteration in the nonconvex setting, we use the standard Frank-Wolfe gap function associated with the hyper-function defined below.
 \begin{definition}\label{def:FW_gap} 
 The Frank-Wolfe gap of  $\ell$ over the set $\mathcal{X}$ is 
  \begin{equation}\label{eq:FW_gap}
    \mathcal{G}(\bx)\triangleq \max_{\bs\in \mathcal{X}}\{\langle \grad \ell(\bx),\bx-\bs\rangle \}.
  \end{equation}
 \end{definition} 
 When the hyper-function $\ell$ is non-convex, we use its Frank-Wolfe gap as our criteria for convergence which is a standard performance metric for 
constrained non-convex settings; see, e.g., \cite{zhang2020one, reddi2016stochastic}.
 In the convex setting, we simply use the suboptimality gap function, i.e., $\ell(\bx)-\ell(\bx^*)$, for capturing the suboptimality of the solution.

\section{Proposed Method}\label{sec:proposed-method}

As previously discussed in Section~\ref{sec:intro}, the bilevel problem presented in \eqref{eq:bilevel} can be conceptualized as a single-level minimization problem $\min_{\bx\in\cX}\ell(\bx)$. 
%This requirement is essential for evaluating both the objective function and its gradient. 
%In particular, 
Given Assumptions~\ref{assum:upper}~and~\ref{assum:lower}, it can be shown that the gradient of the hyper-objective $\ell(\cdot)$ can be expressed as follows: 
\begin{align}\label{eq:v_def}
    \nabla \ell(\bx)&= \nabla_{x}f(\bx,\by^*(\bx))+ \bJ \by^*(\bx)\nabla_{y}f(\bx,\by^*(\bx)),
\end{align}
where $\bJ\by^*(\bx)=- \nabla^2_{yx}g(\bx,\by^*(\bx)) [\nabla^2_{yy}g(\bx,\by^*(\bx))]^{-1}$; see \cite{ghadimi2018approximation}. Implementing first-order methods for solving the single-level minimization problem poses a significant challenge due to the costly operation of matrix inversion and the requirement of finding the exact optimal solution for the lower-level problem. We further elaborate on these issues in the upcoming section and will discuss how our proposed method overcomes them.
\subsection{Main Algorithm}
As discussed above, there are major limitations in a naive implementation of a first-order method for solving \eqref{eq:bilevel}. 
%which makes the method in \eqref{alg:FW-direct} impractical. 
To propose a practical conditional gradient-based method, we revisit the problem's structure. In particular, the gradient of the single-level problem in \eqref{eq:v_def} can be rewritten as follows 
\begin{subequations}
\begin{align}
&\grad\ell(\bx) =\nabla_{x}f(\bx,\by^*(\bx)) - \nabla^2_{yx}g(\bx,\by^*(\bx)) \bv(\bx),\label{eq:ell-v}\\
& \text{where}\ \ \bv (\bx) \triangleq [\nabla^2_{\!yy}g(\bx,\by^*(\bx))]^{-1} \nabla_{\!y}f(\bx,\by^*(\bx)).
\end{align}
\end{subequations}
In this formulation, the effect of Hessian inversion is presented in a separate term $\bv(\bx)$ which can be viewed as the solution of the following \emph{parametric} quadratic problem
\begin{align}\label{eq:obj-v}
    &\bv(\bx) =  \\
    &\argmin_{\bv}\  \frac{1}{2}\bv^\top \nabla^2_{yy}g(\bx,\by^*(\bx))\bv - \nabla_{y}f(\bx,\by^*(\bx))^\top \bv\nonumber.
\end{align}
Our main idea is to provide \emph{nested} approximations for the true gradient in \eqref{eq:v_def} by estimating trajectories of $\by^*(\bx)$ and $\bv(\bx)$. To ensure convergence, we carefully control the algorithm's progress in terms of variable $\bx$ and limit the error introduced by these approximations. 
More specifically, at each iteration $k\geq 0$, given an iterate $\bx_k$ and an approximated solution of the lower-level problem $\by_k$ we first consider an approximated solution $\tilde \bv(\bx_k)$ of \eqref{eq:obj-v} by replacing $\by^*(\bx_k)$ with its currently available approximation, i.e., $\by_k$, which leads to the following quadratic programming
\begin{equation*}
    \tilde \bv(\bx_k) \triangleq \argmin_{\bv}\  \frac{1}{2}\bv^\top \nabla^2_{yy}g(\bx_k,\by_k)\bv - \nabla_{y}f(\bx_k,\by_k)^\top \bv.
\end{equation*}
Then $\tilde \bv(\bx_k)$ is approximated with an iterate $\bw_{k+1}$ obtained by taking one step of gradient descent with respect to the above objective function as follows,
\begin{equation*}
    \bw_{k+1}\gets \bw_k-\eta_k\left(\grad^2_{yy}g(\bx_k,\by_k)\bw_k-\grad_y f(\bx_k,\by_k)\right),
\end{equation*}
for some step-size $\eta_k\geq 0$. 
This generates an increasingly accurate sequence  $\{\bw_k\}_{k\geq 0}$ that tracks the sequence $\{\bv(\bx_k)\}_{k\geq 0}$. 
%take one step of gradient descent with respect to the objective function in \eqref{eq:obj-v} to generate an increasingly accurate estimate of $\bv(\bx_k)$ denoted by $\bw_{k+1}$. 
Next, given approximated solutions $\by_k$ and $\bw_{k+1}$ for $\by^*(\bx_k)$ and $\bv(\bx_k)$, respectively, we can construct a direction to estimate the hyper-gradient $\grad\ell(\bx_k)$ in \eqref{eq:ell-v}. To this end, we construct a direction $F_k=\nabla_{x}f(\bx_k, \by_k)- \nabla^2_{yx}g(\bx_k, \by_k)\bw_{k+1}$, which determines the next iteration $\bx_{k+1}$ using a Frank-Wolfe type update, i.e., 
%by solving a linear minimization problem over the constraint set $\cX$. 
%as follows:
\begin{align*}
    &\bs_k\gets \argmin_{\bs\in\cX}\langle{F_k,\bs}\rangle,\quad \bx_{k+1}\gets (1-\gamma_k)\bx_k+\gamma_k\bs_k.
\end{align*}
for some step-size $\gamma_k\in [0,1]$. 
Finally, having an updated decision variable $\bx_{k+1}$ we estimate the lower-level optimal solution $\by^*(\bx_{k+1})$ by performing another gradient descent step with respect to the lower-level function $g(\bx_k,\by_k)$ with step-size $\alpha>0$ to generate a new iterate $\by_{k+1}$ as follows:
\begin{equation*}
    \by_{k+1}\gets \by_k-\alpha\grad_y g(\bx_{k+1},\by_k). 
\end{equation*}
Our proposed inexact bilevel conditional gradient (IBCG) method is summarized in Algorithm~\ref{alg:In-BiCoG}.

\begin{algorithm}[t!]
\caption{Inexact Bilevel Conditional Gradient (IBCG) Method}\label{alg:In-BiCoG}
\begin{algorithmic}[1]
\STATE \textbf{Input}: $\{\gamma_k, \eta_k\}_k\subseteq\reals_+$, $\alpha>0$, $\bx_0\in \mathcal{X}$, $\by_0 \in \reals^{m}$
\STATE \textbf{Initialization}: ${\bw^0\gets \by^0}$
\FOR{$k = 0,\dots,K-1$}
\STATE $\bw_{k+1} {\gets}(I - \eta_k \nabla^2_{yy}g(\bx_k, \by_k) ) \bw_{k} + \eta_k \nabla_{y}f(\bx_k, \by_k)$
\STATE {$F_k\gets \nabla_{x}f(\bx_k, \by_k)- \nabla^2_{yx}g(\bx_k, \by_k)\bw_{k+1}$}
\STATE Compute $\bs_k \gets \argmin_{\bs \in \mathcal{X}}~\langle{F_k,\bs}\rangle $ 
\STATE $ \bx_{k+1} \gets (1-\gamma_k)\bx_k+\gamma_k \bs_k$\\
\STATE $\by_{k+1}\gets\by_k-\alpha\grad_y g(\bx_{k+1},\by_k)$
\ENDFOR
\end{algorithmic}
\end{algorithm}

 \section{Convergence Analysis}

To establish the convergence guarantee of the proposed IBCG method, we first present the following lemma that quantifies the error between the approximated direction $F_k$ and the exact hyper-gradient $\grad \ell(\bx_k)$. This involves providing upper bounds on the errors induced by our nested approximation technique discussed above, i.e., $\norm{\bw_{k+1}-\bv(\bx_k)}$ and $\norm{\by_{k+1}-\by^*(\bx_{k+1})}$, as well as Lemma \ref{lem:v_b}.

 \begin{lemma}\label{lem:ell-es}
 Suppose Assumptions~\ref{assum:upper}-\ref{assum:lower} hold and let $\beta\triangleq (L_g-\mu_g)/(L_g+\mu_g)$ and $\bC_\bv\triangleq \frac{L_{yx}^{f} + L_{yy}^{f}\bL_{\by}}{\mu_g} + \frac{C_{y}^{f}L_{yy}^{g}}{\mu_g^2}(1+\bL_{\by})$.
 Moreover, let $\{\bx_k,\by_k,\bw_k\}_{k\geq 0}$ be the sequence generated by Algorithm \ref{alg:In-BiCoG} with step-sizes $\gamma_k=\gamma \in (0,1]$, $\eta_k=\eta< \frac{1-\beta}{\mu_g}$, and $\alpha=2/(\mu_g+L_g)$. Then, for any $k\geq 0$
%By utilizing Lemma~\ref{lem:nu_est_K} we have:
     \begin{align}
         \| \nabla \ell (\bx_k) - F_k \| &\leq \bC_2 \big( \beta^k D_0^y  + \frac{\gamma \beta \bL_{\by}}{1-\beta}D_{\cX} \big)+ C_{yx}^{g} \Big(\rho^{k+1} \|{\bw_0}  - \bv(\bx_0) \|  + \frac{\gamma \rho  \bC_{\bv}}{1-\rho} D_{\cX}  \nonumber \\
         & \quad + \frac{\eta \bC_1 }{\rho - \beta} \rho^{k+2}D_0^y  +\frac{\gamma \beta \bC_1 \bL_{\by}}{(1-\rho) (1- \beta)}D_{\cX}\Big),
     \end{align}
     where $\rho\triangleq 1-\eta\mu_g$, $\bC_1 \triangleq L_{yy}^{g} \frac{ C_{y}^{f}}{\mu_g} +  L_{yy}^{f}$,  $\bC_2 \triangleq  \na{L_{xy}^{f}} + L_{yx}^{g} \frac{C_{y}^{f}}{\mu_g}$, and $D_0^y\triangleq \norm{\by_0-\by^*(\bx_0)}$. 
 \end{lemma}
 
 %\begin{proof}
% The proof is relegated to section \ref{proof:lemma-grad-est} of the Appendix.
% \end{proof}

% {\color{red}{Aryan: we need to explain the above result and its importance}}

Lemma \ref{lem:ell-es} provides an upper bound on the error of the approximated gradient direction $F_k$. This bound encompasses two types of terms: those that decrease exponentially fast and others that are influenced by the parameter $\gamma$. 
The selection of the appropriate value for $\gamma$ is crucial because if it is too large, it can introduce significant errors in the algorithm's direction, and if it is too small, it can impede progress in the iterations. Therefore, it is essential to choose an appropriate $\gamma$ based on the algorithm's overall progress. By utilizing the insights from Lemma \ref{lem:ell-es}, we can set a limit on the gap function and ensure a guaranteed rate of convergence through the appropriate choice of $\gamma$.

Next, we present our convergence guarantees for the IBCG method. We start by considering the scenario where the hyper-objective function  $\ell(\cdot)$ exhibits convexity.
Before presenting our results, we highlight two generic examples that yield a convex hyper-objective function $\ell$: (i) When $f$ is jointly convex, and $y^*(\cdot)$ is an affine map. (ii) In cases where we encounter a min-max optimization problem of the form $\min_{\mathbf{x} \in \mathcal{X}} \max_{\mathbf{y} \in \mathbb{R}^m} f(\mathbf{x},\mathbf{y})$, and $f$ demonstrates convexity in $\bf x$ and strong concavity in $\bf y$. This problem can be reformulated into a bilevel optimization problem by defining $g({\bf x},{\bf y})=-f({\bf x},{\bf y})$, which results in a convex hyper-objective function $\ell(\mathbf{x}) = \max_{\mathbf{y}} f(\mathbf{x},\mathbf{y})$.

\begin{theorem}[Convex bilevel]\label{thm:convex-upper-bound}
Suppose Assumptions~\ref{assum:upper} and \ref{assum:lower} hold. Let $\{\bx_k\}_{k=0}^{K-1}$ be the sequence generated by Algorithm~\ref{alg:In-BiCoG} with step-sizes specified as in Lemma \ref{lem:ell-es}. If $\ell(\bx)$ is convex, then for all $K\geq 1$, 
\begin{align}\label{eq:subopt-convex}
   \ell(\bx_K) - \ell(\bx^*) &\leq (1- \gamma )^K( \ell(\bx_0) - \ell(\bx^* )) + \sum_{k=0}^{K-1}(1-\gamma)^{K-k}\cR_k(\gamma),
\end{align}
where
$
     %&
     \cR_k(\gamma)\triangleq 
     %\nonumber\\  &
     \gamma \bC_2 \beta^k D_0^y D_{\cX} + \frac{\gamma^2 \bC_2 D_{\cX}^2       \bL_{\by} \beta}{1-\beta} +   C_{yx}^{g}[ \gamma D_{\cX}  \rho^{k+1}%\nonumber \\
         %& \quad 
         \|{\bw_0} - \bv(\bx_0) \|  + \frac{\gamma^2 D_{\cX}^2      \rho \bC_{\bv}}{1-\rho} 
    + \frac{\gamma D_{\cX} D_0^y \bC_1\eta \rho^{k+2}}{\rho - \beta} %\nonumber \\
        % & \quad 
         + \frac{\gamma^2 D_{\cX}^2 \bL_{\by} \bC_1 \beta \eta}{(1-\beta)(1-\rho)}] + \frac{1}{2} \bL_{\ell} \gamma^2 D_{\cX}^2$.
% \end{align}
\end{theorem}

Theorem~\ref{thm:convex-upper-bound} establishes that in the convex setting, the suboptimality of the function $\ell$ is bounded above by an upper bound composed of two components. The first component exhibits a linear convergence rate and decreases exponentially, while the second component stems from errors in nested approximations and can be alleviated by reducing the step-size $\gamma$. Therefore, by properly choosing the step-size $\gamma$, we can attain a guaranteed convergence rate, as outlined in the following Corollary.

\begin{corollary}\label{cr:convex_upper-bound}
Let $\{\bx_k\}_{k=0}^{K-1}$ be the sequence generated by Algorithm \ref{alg:In-BiCoG} with step-size $\gamma_k = \gamma =  \frac{\log( K)}{K}$. Under the premises of Theorem~\ref{thm:convex-upper-bound} we have that  $\ell(\bx_K) - \ell(\bx^*) \leq \epsilon$ after $\cO(\kappa_g^5 \epsilon^{-1}\log(\epsilon^{-1}))$ iterations. Furthermore, assuming that $\grad_y f(\bx,\cdot)$ is uniformly bounded for any $\bx\in\cX$, we have that $\ell(\bx_k)-\ell(\bx^*) \leq \epsilon$ after $\cO(\kappa_g^4 \epsilon^{-1}\log(\epsilon^{-1}))$ iterations.
\end{corollary}

By leveraging the result of Theorem~\ref{thm:convex-upper-bound}, the above Corollary characterizes the complexity of IBCG in the convex setting. Specifically, when we set $\gamma=\log(K)/K$, the total iteration complexity of IBCG to find an $\epsilon$-optimal solution is $\cO(\kappa_g^5 \epsilon^{-1}\log(\epsilon^{-1}))$, where $\kappa_g$ represents the condition number of the lower-level problem. Additionally, if the gradient of the upper-level is bounded, the required complexity is further reduced to $\cO(\kappa_g^4 \epsilon^{-1}\log(\epsilon^{-1}))$.
Now we turn to the case where $\ell(\cdot)$ is nonconvex.
\begin{theorem}[Non-convex bilevel]\label{thm:nonconvex-upper-bound}
Suppose Assumptions~\ref{assum:upper} and \ref{assum:lower} hold. Let $\{\bx_k\}_{k=0}^{K-1}$ be the iterates generated by Algorithm \ref{alg:In-BiCoG} with step-sizes specified as in Lemma \ref{lem:ell-es}.
Then, 
\begin{align}
    \mathcal{G}_{k^*} &\leq  \frac{\ell(\bx_0) -\ell(\bx^*)}{K \gamma}+ \frac{\gamma \bC_2 D_{\cX} \bL_{\by} \beta}{1-\beta}+ \frac{\gamma D_{\cX}^2 \rho \bC_{\bv} C_{yx}^{g} \rho}{1-\rho} +\frac{\gamma D_{\cX}^2 C_{yx}^{g} \bL_{\by} \bC_1 \beta \eta}{(1-\beta)(1-\rho)} + \frac{1}{2} \bL_{\ell}\gamma D_{\cX}^2  + \frac{\bC_2  D_0^y D_{\cX} \beta}{K(1- \beta)} \nonumber \\
    & \quad + \frac{D_{\cX}  C_{yx}^{g} \rho \|{\bw_0} - \bv(\bx_0) \|}{K(1- \rho)}+ \frac{ D_{\cX} D_0^y C_{yx}^{g} \bC_1 \eta \rho^2}{K(\rho-\beta)(1-\rho)} \nonumber,
\end{align}
where $\mathcal{G}_{k^*}$ is defined as $\mathcal{G}_{k^*} \triangleq  \min_ {0\leq k \leq K-1}~\mathcal{G}(\bx_k)$.
\end{theorem}

Theorem~\ref{thm:nonconvex-upper-bound} establishes
an upper bound on the  Frank-Wolfe gap for the iterates
generated by IBCG. It shows that the Frank-Wolfe gap vanishes when the step-size $\gamma$ is properly selected. Specifically, setting $\gamma = \cO(1/\sqrt{K})$  as outlined in the next corollary leads to a convergence rate of $\cO(1/\sqrt{K})$. 
%Interestingly, these results match the state-of-the-art methods for projection-free single-level optimization problems.
\begin{corollary}\label{cr:nonconvex_upper-bound}
%Considering the result of Theorem~\ref{thm:nonconvex-upper-bound}, 
Let $\{\bx_k\}_{k=0}^{K-1}$ be the sequence generated by Algorithm \ref{alg:In-BiCoG} with step-size $\gamma_k = \gamma =  \kappa_g^{-2.5}K^{-0.5}$, then there exists $k^*\in\{0,1,\dots,K-1\}$ such that $\mathcal{G}_{k^*} \leq \epsilon$ after $\cO(\kappa_g^5 \epsilon^{-2})$ iterations. Furthermore, if $\grad_y f(\bx,\cdot)$ is uniformly bounded for any $\bx\in\cX$, selecting $\gamma_k = \gamma =  \kappa_g^{-2}K^{-0.5}$ implies that $\mathcal{G}_{k^*} \leq \epsilon$ after $\cO(\kappa_g^4 \epsilon^{-2})$ iterations.
\end{corollary}

\begin{remark}
Our proposed method's efficiency stands out due to its requirement of just two matrix-vector multiplications. Furthermore, our results establish the state-of-the-art bounds for the considered setting.
In the convex setting, our complexity result is near-optimal among projection-free methods for single-level optimization problems, as it is known that the worst-case complexity of such methods is $\cO(1/\epsilon)$ \cite{jaggi2013revisiting, lan2013complexity}.
In the nonconvex setting, our complexity result matches the best-known bound of $\cO(1/\epsilon^2)$ within the family of projection-free methods for single-level optimization problems \cite{jaggi2013revisiting, lan2013complexity}. This underscores the efficiency and effectiveness of our approach in this specific context.
\end{remark}

\begin{remark}
One of the key distinctions between our IBCG method and SBFW lies in how we approximate the Hessian matrix inversion within the hyper-gradient equation in \eqref{eq:v_def}. In our approach, we compute $\tilde{\mathbf{\bv}}(\bx_k)$ by computing a single step of gradient descent on the quadratic function presented in \eqref{eq:obj-v}. This leads to a single-loop algorithm that demands only two matrix-vector products per iteration. Considering the general scenario where $\grad_y f(\bx,\cdot)$ may not be bounded, the error of approximating $\bv(\bx_k)$ can be bounded as $\norm{\bw_k-\bv(\bx_k)}\leq \cO((\frac{\kappa_g-1}{\kappa_g+1})^k+{\kappa_g^{2.5}}/{\sqrt{k}})$ by selecting $\gamma$ as in Corollary \ref{cr:nonconvex_upper-bound} -- see Lemma~\ref{lem:nu_est_K} in Appendix for details. In contrast, SBFW constructs a biased stochastic estimate of the Hessian inverse matrix, namely $\bH_k$, such that $\norm{\bH_k\grad_yf(\bx_k,\by_k)-\bv(\bx_k)}\leq \cO(\kappa_g^2/\sqrt{k})$ at the cost of computing $\cO(\kappa_g\log(k))$ matrix-vector products per iteration. Consequently, our analysis is fundamentally different from those in prior work and presents a novel contribution within the realm of projection-free algorithms.
%\green{As an illustration of this distinction, consider Lemma~\ref{lem:nu_est_K} in our paper. This lemma represents a novel contribution within the realm of projection-free algorithms, where we provide an upper bound on the error of approximating ${\bf v}({\bf x}_k)$ by ${\bf w}_{k+1}$}.
\end{remark}

\section{Numerical Experiments}\label{sec:numeric}

\begin{figure*}[ht!]
     \centering
     \includegraphics[scale=0.32]{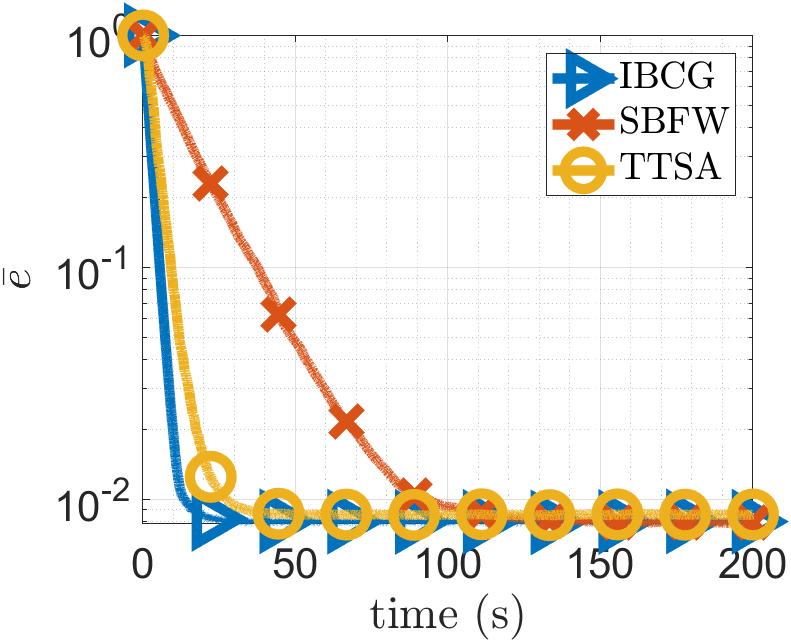}\qquad 
     \includegraphics[scale=0.32]{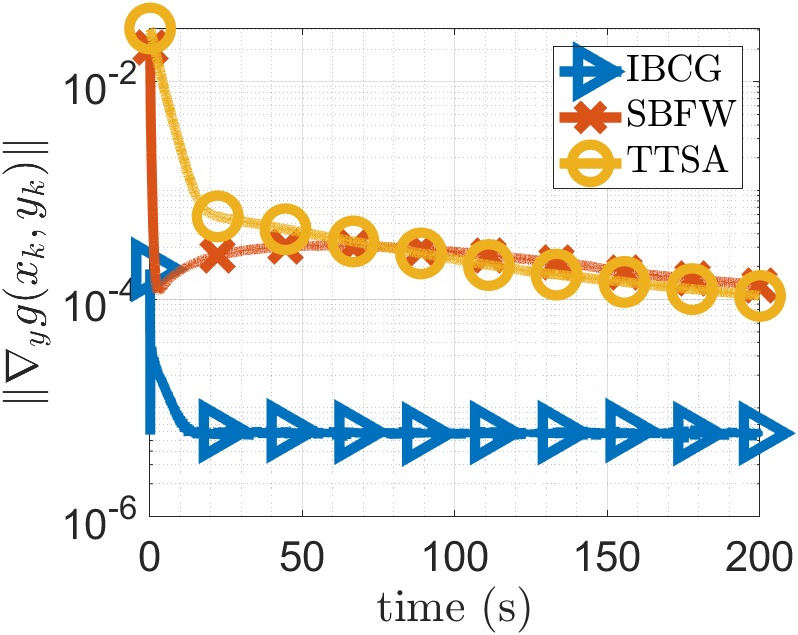}\qquad
     \includegraphics[scale=0.32]{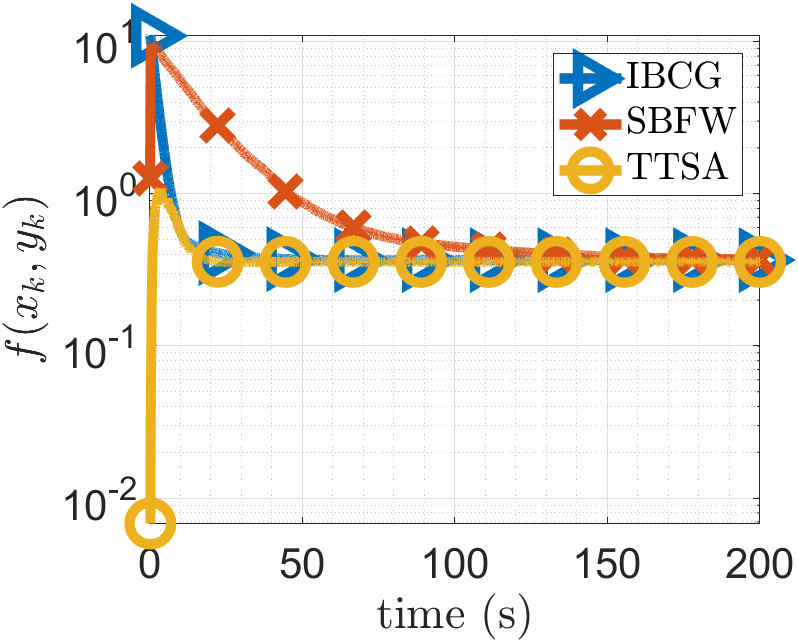}
     \vspace{-2mm}
     \caption{Performance of IBCG  vs SBFW and  TTSA on problem \eqref{ex:Matrix_comp} for synthetic dataset. Plots from left to right:  normalized error $(\Bar{e})$, 
      $\|\nabla_y g(\bx_k,\by_k)\|$, 
     and $f(\bx_k,\by_k)$ over time.}
     \label{fig:matrixcom250}
 \end{figure*}

 \begin{figure*}[ht!]
     \centering
     \includegraphics[scale=0.32]{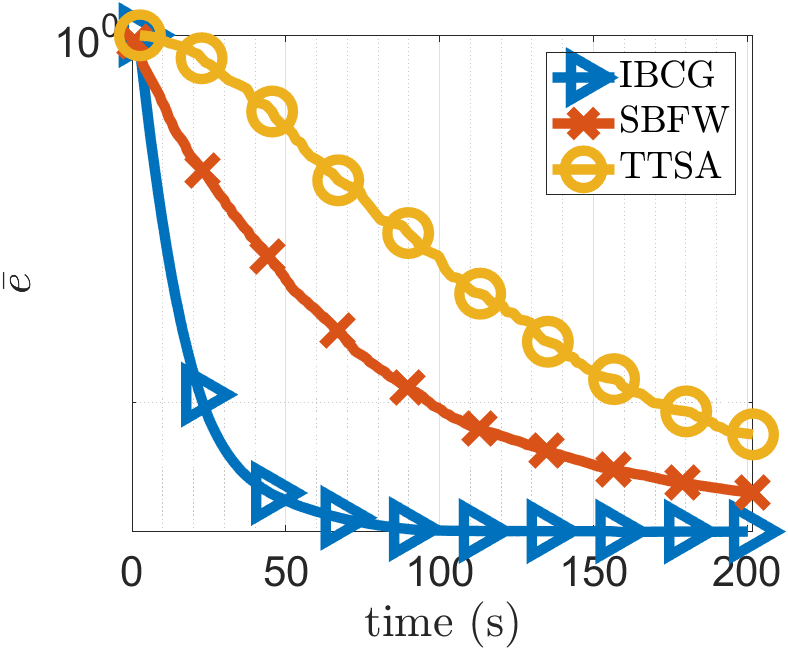}
     \qquad
     \includegraphics[scale=0.32]{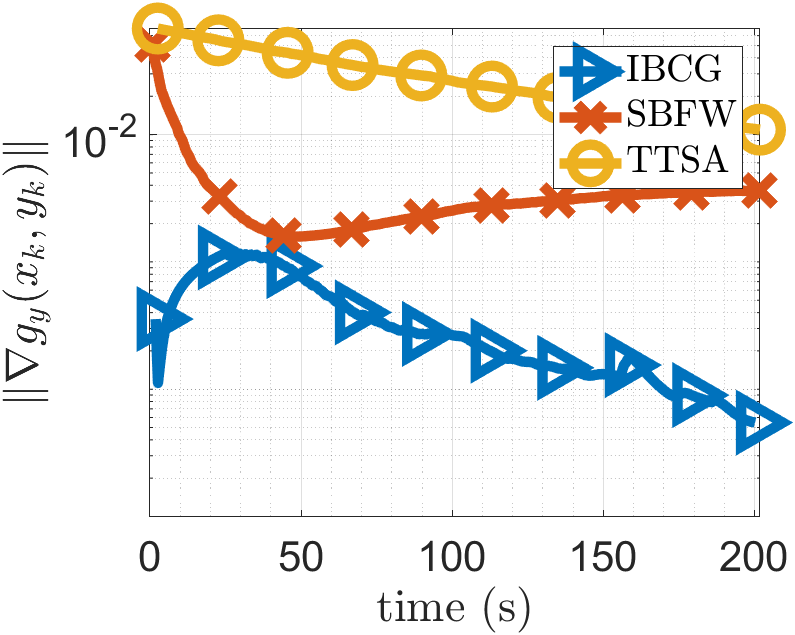}
     \qquad
     \includegraphics[scale=0.32]{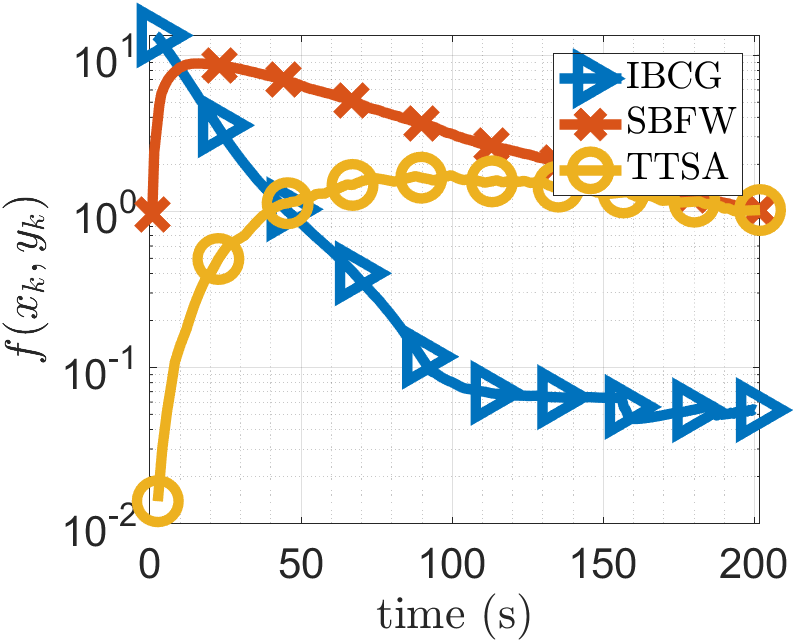}
     \vspace{-2mm}
     \caption{Performance of IBCG (blue) vs SBFW (red) and  TTSA (yellow) on problem \eqref{ex:Matrix_comp} for the MovieLens dataset. Plots from left to right: normalized error $(\Bar{e})$, 
      $\|\nabla_y g(\bx_k,\by_k)\|$, 
     and $f(\bx_k,\by_k)$ over time.}
     \label{fig:matrixcom_real}
     \vspace{-3mm}
 \end{figure*}
 
 In this section, we test our method for solving different bilevel optimization problems and compare it with methods proposed in \cite{hong2020two,akhtar2022projection}. All the experiments are performed with Intel(R) Core(TM) i5-10210U CPU @ 1.60GHz. Additional numerical experiments are presented in Appendix \ref{sec:additional_example}.

\subsection{Matrix Completion with Denoising}\label{num: matrix_comp}

In this section, we evaluate the performance of our IBCG algorithm in solving the matrix completion with denoising  problem outlined in \eqref{ex:Matrix_comp}. We conduct experiments using both synthetic and real datasets to assess its effectiveness.

\noindent \textbf{Synthetic dataset.}
The experimental setup we adopt is aligned with the methodology used in \citet{mokhtari2020stochastic}. In particular, we create an observation matrix $M = \hat{X}+E$. In this setting $\hat{X} = WW^{T}$  where $W \in \reals^{n \times r}$ containing normally distributed independent entries, and $E = \hat{n}(L+L^{T})$ is a noise matrix where $L \in \reals^{n \times n}$ containing normally distributed independent entries and $\hat{n} \in (0,1)$ is the noise factor. During the simulation process, we set $n= 250$, $r = 10$, and $\alpha = \| \hat{X}\|_{*}$. Additionally, we establish the set of observed entries $\Omega$ by randomly sampling entries of  $M$ with a probability of $0.8$. Initially, we set $\hat{n}=0.5$ and employ the IBCG algorithm to solve the problem described in~\eqref{ex:Matrix_comp}. We compare our method with 
TTSA~\citet{hong2020two} and SBFW~\citet{akhtar2022projection}. 
We set $ \lambda_1 = \lambda_2 = 0.05$, and set the maximum number of iteration as $10^4$. It should be noted that we consider pseudo-Huber loss defined by $\cR_{\delta}(\bV) = \sum_{i,j}\delta^2 (\sqrt{1+(\bV_{ij}/ \delta)^2}-1)$ as a regularization term to induce sparsity and set $\delta = 0.9$. The performance is analyzed based on the normalized
error, defined as $ \Bar{e} := \sum _{(i,j) \in \Omega} (X_{i,j} - \hat{X}_{i,j})^2 / \sum _{(i,j) \in \Omega}  (\hat{X}_{i,j})^2$, 
% \begin{align}
%     \Bar{e} = \frac{\sum _{(i,j) \in \Omega} (X_{i,j} - \hat{X}_{i,j})^2} {\sum _{(i,j) \in \Omega}  (\hat{X}_{i,j})^2}
% \end{align}
where $X$ is the matrix generated by the algorithm.
Note that, despite being a projection-free method, the SBFW algorithm exhibits a slower theoretical convergence rate. In comparison, our proposed method surpasses TTSA by attaining lower values of $\|\grad_y g(\bx_k,\by_k)\|$ and demonstrating slightly better performance in terms of the normalized error values, as shown in Fig.~\ref{fig:matrixcom250}.
This gain is due to the projection-free nature of our method, allowing for fast convergence without the need for complex projections at each iteration.

\noindent \textbf{Real dataset.}
To assess IBCG's scalability, we conducted experiments using the MovieLens datasets, containing large matrices of user-generated movie ratings (ranging from 1 to 5). We employed the MovieLens 100k dataset, encompassing $10^5$ ratings from 1000 individuals and 1700 movies. This dataset is represented by the observation matrix $M\in\reals^{1000 \times 1700}$.
Fig.~\ref{fig:matrixcom_real} illustrates the performance of considered methods. Our proposed method exhibits faster convergence in terms of normalized error ($\bar{e}$), lower-level optimality $\|\grad_y g(\bx_k,\by_k)\|$, and upper-level objective function value $f(\bx_k,\by_k)$ compared to other methods.
To underscore the practical significance of the projection-free bilevel approach, we further conducted experiments on a more extensive dataset, with the results available in Appendix~\ref{appen: matrix_comp}.

%\vspace{-2mm}
{\subsection{Multi-Task Learning}}

In this section, we compare our IBCG method with other benchmarks for solving the Multi-Task Learning problem defined in \eqref{ex:MTL_bi}. We use various datasets, including \texttt{bodyfat}, \texttt{housing}, \texttt{mg}, \texttt{mpg}, and \texttt{space-ga} from the LibSVM library \citet{CC01a}, treating each as the dataset corresponding to a distinct task.
For each dataset $\cD_i=(\bA_i,\bb_i)$, where $\bA_i\in \reals^{n_i\times d}$ and $\bb_i\in\reals^{n_i}$, we define its corresponding loss function as $L_i(w_i;\cD_i)=\frac{1}{n_i}\norm{\bA_iw_i-\bb_i}^2$. We assign $60\%$ of the data points of each task as the training set $(\bA_i^{\mathrm{tr}},\bb_i^{\mathrm{tr}})$, $20\%$ as the validation set $(\bA_i^{\mathrm{val}},\bb_i^{\mathrm{val}})$ and the rest as the test set $(\bA_i^{\mathrm{test}},\bb_i^{\mathrm{test}})$.
To evaluate our method's efficacy, we compare it with a state-of-the-art projection-free method for constrained bilevel optimization problems namely SBFW, and excluding TTSA due to its frequent projections onto the constraint set, incurring substantial computational overhead -- for the details of the experiment see Appendix~\ref{appen:MTL}. Fig.~\ref{fig:MTL_itr} illustrates the performance of the two methods, showing our proposed method achieving faster convergence compared to the slower convergence rate of SBFW.
To emphasize our method's practical significance over SBFW, we also conducted this experiment by fixing the running time for both methods and observe that IBCG algorithm demonstrates rapid initial improvement as shown in Fig.~\ref{fig:MTL_time}. 
%and remains relatively stable throughout the time span. 
%These observations could be critical when selecting an optimization algorithm for time-sensitive applications or when computational resources are limited. Additional experiment and analysis 
Furthermore, the comparison of test accuracy for the trained model with the algorithms and training datasets individually can be found in Appendix~\ref{appen:MTL}.

\section{Conclusion}
%\vspace{-1mm}
In this paper, we focused on the constrained general bilevel optimization problem that appears in a wide range of applications. We introduced a novel single-loop projection-free method based on nested approximation techniques, which provides optimal convergence rate guarantees, matching the best-known complexity of projection-free algorithms for solving convex constrained single-level optimization problems.
Specifically, we demonstrated that our method requires approximately $\tilde{\cO}(\epsilon^{-1})$ iterations to find an $\epsilon$-optimal solution when the hyper-objective function $\ell(x)$ is convex. For nonconvex $\ell(x)$, it takes approximately $\cO(\epsilon^{-2})$ iterations to find an $\epsilon$-stationary point. Our numerical results further underscored the superior performance of our IBCG algorithm compared to existing methods.

\begin{figure}
     %\vspace{-1mm}
     \centering
     \includegraphics[scale=0.3]{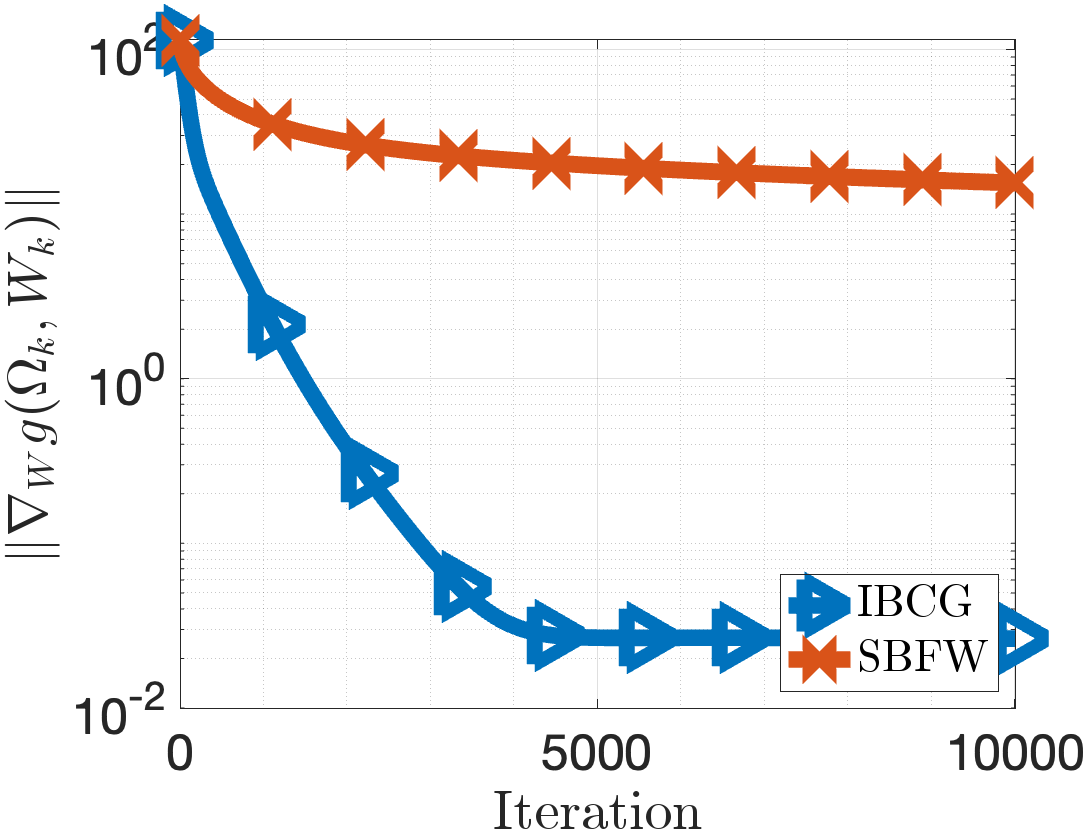}
     \includegraphics[scale=0.3]{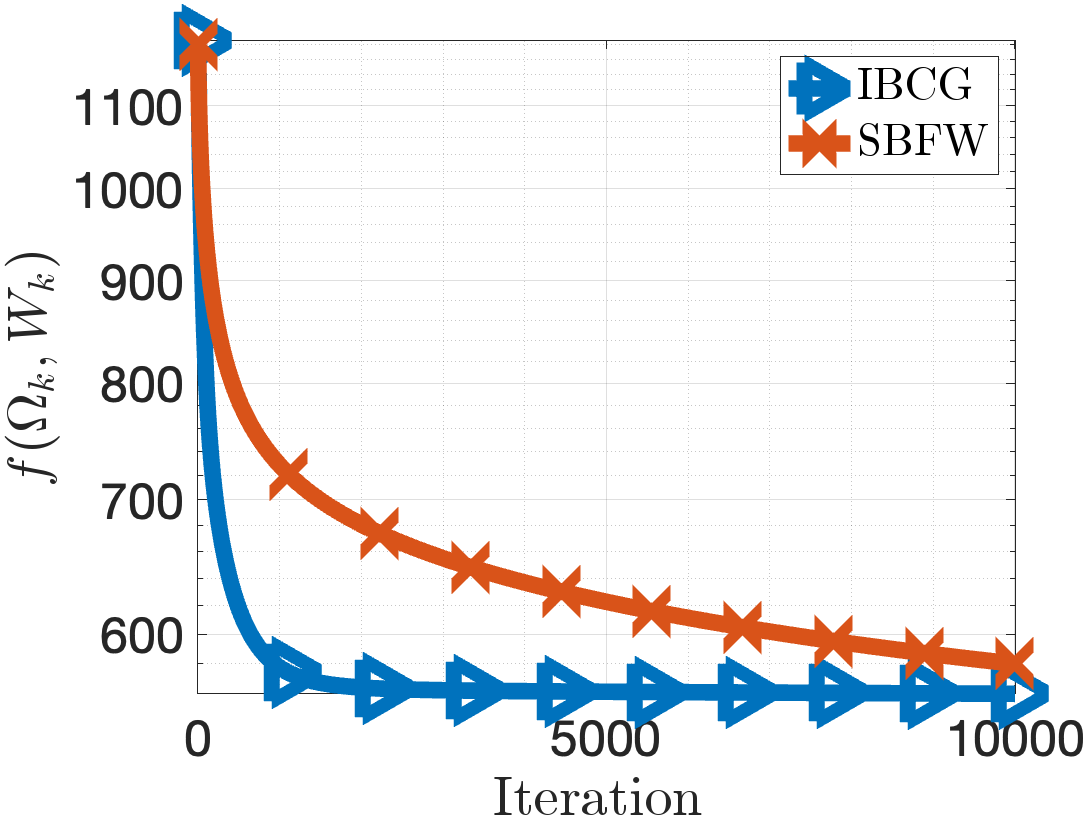}
     \vspace{-2mm}
     \caption{Performance of IBCG  vs SBFW on problem \eqref{ex:MTL_bi} for real dataset1 . Plots from left to right:
      $\|\nabla_W g(\Omega_k,W_k)\|$, 
     and $f(\Omega_k,W_k)$ in terms of number of iterations.}
     \label{fig:MTL_itr}
 \end{figure}
\begin{figure}
    \vspace{-4mm}
     \centering
     \includegraphics[scale=0.3]{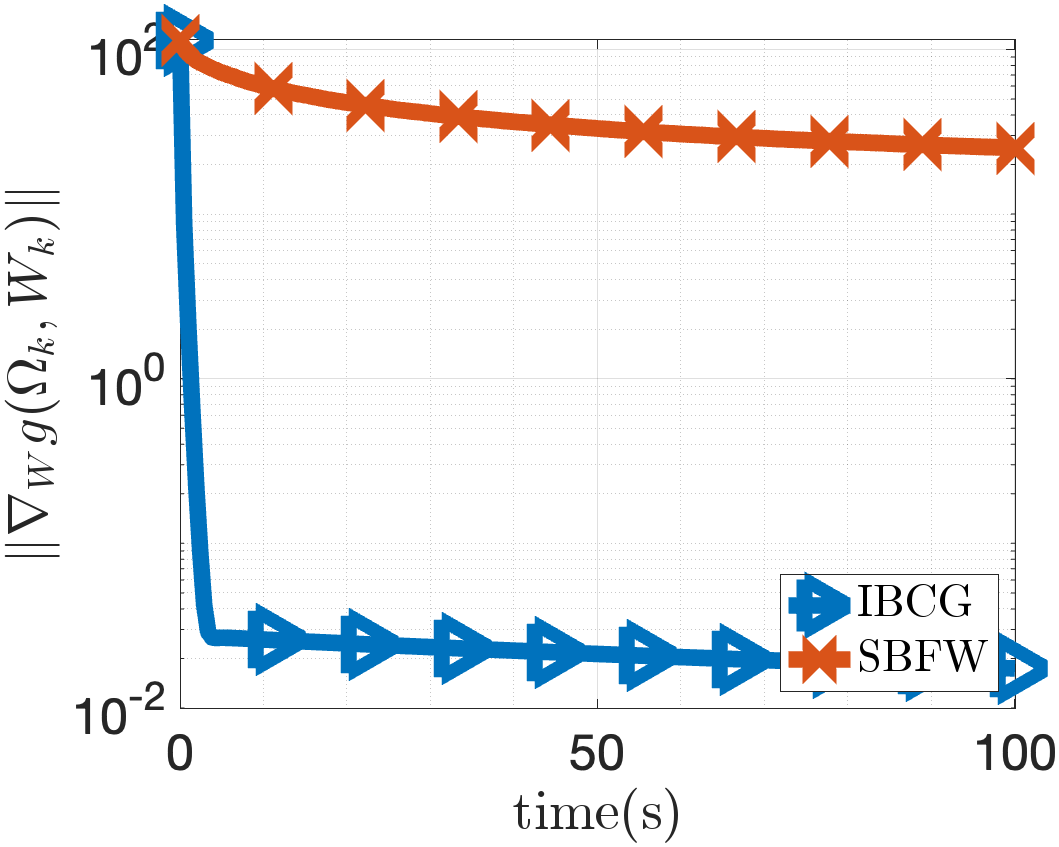}
     \includegraphics[scale=0.3]{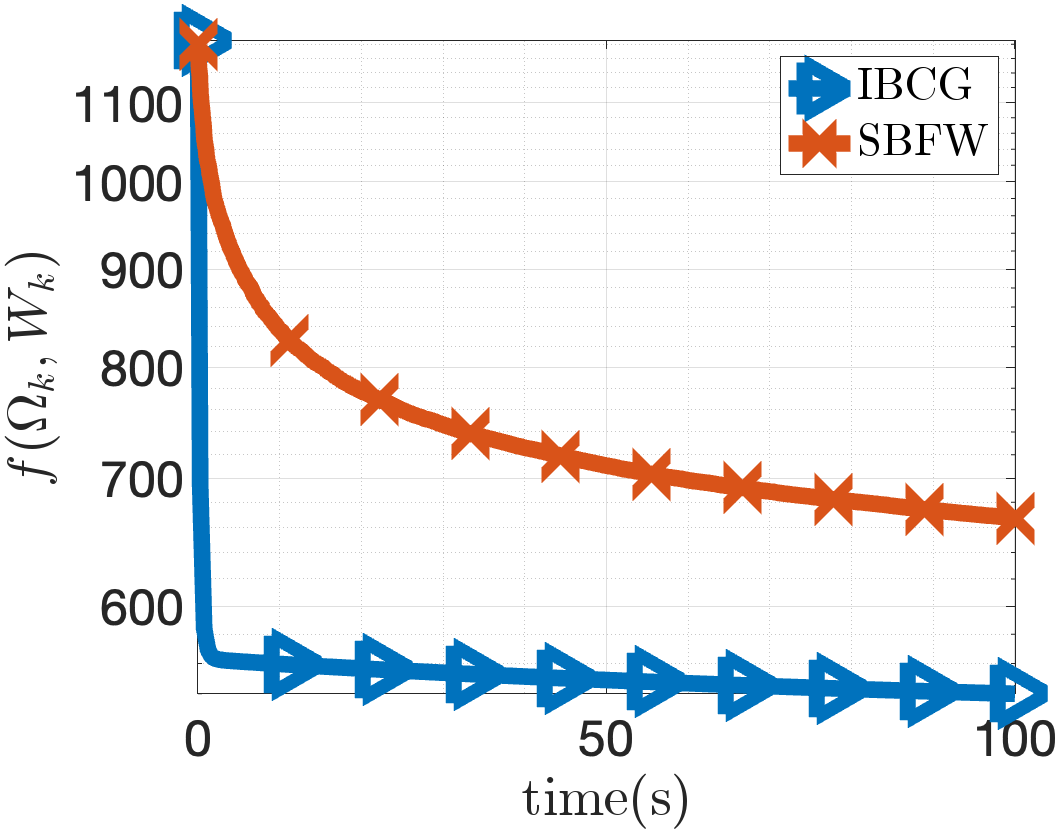}
     \vspace{-2mm}
     \caption{Performance of IBCG  vs SBFW  on problem \eqref{ex:MTL_bi} for real dataset1. Plots from left to right:
      $\|\nabla_W g(\Omega_k,W_k)\|$, 
     and $f(\Omega_k,W_k)$ in terms of running time.}
     \label{fig:MTL_time}
 \end{figure}
% \subsubsection*{Acknowledgements}
% The research of N. Abolfazli and E. Yazdandoost Hamedani is supported by NSF Grant 2127696. The research of R. Jiang and A. Mokhtari is supported in part by NSF Grants 2127697, 2019844, and  2112471,  ARO  Grant  W911NF2110226,  the  Machine  Learning  Lab  (MLL)  at  UT  Austin, and the Wireless Networking and Communications Group (WNCG) Industrial Affiliates Program. 
% \section*{Societal Implications}
% This paper presents work whose goal is to advance the field of Machine Learning. There are many potential societal consequences of our work, none which we feel must be specifically highlighted here.

\section*{Acknowledgements}
The research of N. Abolfazli and E. Yazdandoost Hamedani is supported by NSF Grant 2127696. The research of R. Jiang and A. Mokhtari is supported in part by NSF Grants 2127697, 2019844, and  2112471,  ARO  Grant  W911NF2110226,  the  Machine  Learning  Lab  (MLL)  at  UT  Austin, and the Wireless Networking and Communications Group (WNCG) Industrial Affiliates Program.

\printbibliography
%%%%%%%%%%%%%%%%%%%%%%%%%%%%%%%%%%%%%%%%%%%%%%%%%%%%%%%%%%%%%%%%%%%%%%%%%%%%%%%
%%%%%%%%%%%%%%%%%%%%%%%%%%%%%%%%%%%%%%%%%%%%%%%%%%%%%%%%%%%%%%%%%%%%%%%%%%%%%%%
% APPENDIX
%%%%%%%%%%%%%%%%%%%%%%%%%%%%%%%%%%%%%%%%%%%%%%%%%%%%%%%%%%%%%%%%%%%%%%%%%%%%%%%
%%%%%%%%%%%%%%%%%%%%%%%%%%%%%%%%%%%%%%%%%%%%%%%%%%%%%%%%%%%%%%%%%%%%%%%%%%%%%%%
\newpage
\appendix
\onecolumn
%\section{You \emph{can} have an appendix here.}

\section*{Appendix}

In section \ref{sec:supporting-lemmas}, we establish technical lemmas based on the assumptions considered in the paper. These lemmas characterize important properties of problem \ref{eq:bilevel}. Notably, Lemma \ref{lem:v_b} is instrumental in understanding the properties associated with problem \ref{eq:bilevel}.
Moving on to section \ref{sec:required-lemma}, we present a series of lemmas essential for deriving the rate results of the proposed algorithm. Among them, Lemma \ref{lem:ell-es} quantifies the error between the approximated direction $F_k$ and $\grad \ell(\bx_k)$. This quantification plays a crucial role in establishing the one-step improvement lemma (see Lemma \ref{lem:one-step}).
Next, we provide the proofs of Theorem \ref{thm:convex-upper-bound} and Corollary \ref{cr:convex_upper-bound} in sections \ref{appen:convex} and \ref{appen:convex_coro}, respectively, that support the results presented in the paper for the convex scenario. Finally, in sections \ref{appen: nonconvex} and \ref{appen:nonconvex_coro} we provide the proofs for Theorem \ref{thm:nonconvex-upper-bound} along with Corollary \ref{cr:nonconvex_upper-bound} for the nonconvex scenario.

\section{Supporting Lemmas}\label{sec:supporting-lemmas}
In this section, we provide detailed explanations and proofs for the lemmas supporting the main results of the paper. 
%which is crucial for establishing the validity and rigor of our arguments and conclusions.
\subsection{Proof of Lemma~\ref{lem:v_b}}
%We start by showing part (I). 
\textbf{(I)} Recall that $\by^*(\bx)$ is the minimizer of the lower-level problem whose objective function is strongly convex, therefore, 
    \begin{align*}
     \mu_g\norm{\by^*(\bx)-\by^*(\bar\bx)}^2&\leq \langle{\grad_y g(\bx,\by^*(\bx))-\grad_y g(\bx,\by^*(\bar\bx)),\by^*(\bx)-\by^*(\bar\bx)} \rangle \\
     &=\langle{\grad_y g(\bar\bx,\by^*(\bar\bx))-\grad_y g(\bx,\by^*(\bar\bx)),\by^*(\bx)-\by^*(\bar\bx)} \rangle
    \end{align*}
Note that $\grad_y g(\bx,\by^*(\bx))= \grad_y g(\bar\bx,\by^*(\bar\bx))=0$. Using the Cauchy-Schwartz inequality we have
\begin{align*}
    \mu_g\norm{\by^*(\bx)-\by^*(\bar\bx)}^2&\leq\| \grad_y g(\bar\bx,\by^*(\bar\bx))-\grad_y g(\bx,\by^*(\bar\bx))\| \|\by^*(\bx)-\by^*(\bar\bx)\| \\
    & \leq \na{C_{yx}^{g}} 
    %\red{L_g}
    \| \bx - \bar{\bx}\| \|\by^*(\bx)-\by^*(\bar\bx)\| %\label{pr:L_y}
\end{align*}
where the \ey{last} inequality is obtained by using the \ey{Assumption}~\ref{assum:lower}. 
\ey{Therefore, we conclude that $ \mu_g\norm{\by^*(\bx)-\by^*(\bar\bx)} \leq C_{yx}^{g} \| \bx - \bar{\bx}\| $  which leads to the desired result in part (I).} 
%which is equivalent to the desired result that  $\|\by^*(\bx) - \by^*(\Bar{\bx})\| \leq \frac{ C_{yx}^{g}}{\mu_g}\| \bx - \Bar{\bx}\|$.\vspace{1mm} \\

\textbf{(II)} %Next, \ey{we will prove the result of part (II).} 
We first show that the function $\bx \mapsto \grad_y f(\bx,\by^*(\bx))$ is Lipschitz continuous. To see this, note that for any $\bx,\bar\bx \in \cX$, we have 
\begin{align*}
    \|\grad_y f(\bx,\by^*(\bx))-\grad_y f(\bar\bx,\by^*(\bar\bx))\| &\leq L_{yx}^f\|\bx-\bar\bx\|+L_{yy}^f\|\by^*(\bx)-\by^*(\bar\bx)\| \\
    &\leq \Bigl(L_{yx}^f+\frac{L_{yy}^f \na{C_{yx}^{g}}}{\mu_g}\Bigr)\|\bx-\bar\bx\|,
\end{align*}
where in the last inequality we used Lemma~\ref{lem:v_b}-(I). Since $\cX$ is bounded, we also have $\|\bx-\bar\bx\|\leq D_\cX$. Therefore, letting $\bar \bx = \bx^*$ in the above inequality and using the triangle inequality, we have 
\begin{equation*}
    \|\grad_y f(\bx,\by^*(\bx))\| \leq \Bigl(L_{yx}^f+\frac{L_{yy}^f \na{C_{yx}^{g}}}{\mu_g}\Bigr)D_\cX+ \|\grad_y f(\na{\bx^*},\by^*(\na{\bx^*}))\|.
\end{equation*}
Thus, we complete the proof by letting $C_y^f=\Bigl(L_{yx}^f+\frac{L_{yy}^f \na{C_{yx}^{g}}}{\mu_g}\Bigr)D_\cX+ \|\grad_y f(\na{\bx^*},\by^*(\na{\bx^*}))\|$. 

Before proceeding to show the result of part (III) of Lemma~\ref{lem:v_b}, we first establish an auxiliary lemma stated next.
\begin{lemma}\label{lem:v(x)}
Under the premises of Lemma \ref{lem:v_b}, we have that for any $\bx,\bar\bx\in\cX$, $\norm{\bv(\bx)-\bv(\bar\bx)}\leq \bC_\bv\norm{\bx-\bar \bx}$ for some $\bC_\bv\geq 0$.
\end{lemma}
\begin{proof}
We start the proof by recalling that $\bv(\bx)=\nabla^2_{yy}g(\bx,\by^*(\bx))^{-1} \nabla_{y}f(\bx,\by^*(\bx))$. Next, adding and subtracting $\nabla^2_{yy} g(\bx, \by^*(\bx))\nabla_y f(\Bar{\bx},\by^*(\bar{\bx}))$ followed by a triangle inequality leads to, 
%start proofing the second part:
    \begin{align}\label{eq:bound-v}
        &\| \bv(\bx) - \bv(\Bar{\bx})\| \nonumber
        \\ 
        &= \|[\nabla^2_{yy}g(\bx,\by^*(\bx))]^{-1} \nabla_{y}f(\bx,\by^*(\bx)) - [\nabla^2_{yy}g(\Bar{\bx},\by^*(\Bar{\bx}))]^{-1} \nabla_{y}f(\Bar{\bx},\by^*(\Bar{\bx}))\| \nonumber \\
        %&= \| [\nabla_{yy}g(\bx,\by^*(\bx))]^{-1}\nabla_{y}f(\bx,\by^*(\bx)) - [\nabla_{yy}g(\bx,\by^*(\bx))]^{-1}\nabla_{y}f(\Bar{\bx},\by^*(\Bar{\bx})) \nonumber\\
        %&+ [\nabla_{yy}g(\bx,\by^*(\bx))]^{-1}\nabla_{y}f(\Bar{\bx},\by^*(\Bar{\bx})) - [\nabla_{yy}g(\Bar{\bx},\by^*(\Bar{\bx}))]^{-1}\nabla_{y}f(\Bar{\bx},\by^*(\Bar{\bx}))\| \nonumber \\
        &\leq \|[\nabla^2_{yy}g(\bx,\by^*(\bx))]^{-1}\big( \nabla_{y}f(\bx,\by^*(\bx)) - \nabla_{y}f(\Bar{\bx},\by^*(\Bar{\bx})) \big)\| + \|\big([\nabla^2_{yy}g(\bx,\by^*(\bx))]^{-1} \nonumber \\
        &- [\nabla^2_{yy}g(\Bar{\bx},\by^*(\Bar{\bx}))]^{-1} \big)\nabla_{y}f(\Bar{\bx},\by^*(\Bar{\bx}))\| \nonumber \\
        & \leq \frac{1}{\mu_g}\big( {L_{yx}^f}\| \bx - \Bar{\bx}\| + {L_{yy}^f}\|\by^*(\bx) - \by^*(\Bar{\bx})\|\big) + C_{y}^{f}\| [\nabla^2_{yy}g(\bx,\by^*(\bx))]^{-1} - [\nabla^2_{yy}g(\Bar{\bx},\by^*(\Bar{\bx}))]^{-1} \|,
        % &\leq \frac{1}{\mu_g} L_{yy}^{f} \|\bx -\Bar{\bx}\| + C_{y}^{f}\| [\nabla_{yy}g(\bx,\by^*(\bx))]^{-1} - [\nabla_{yy}g(\Bar{\bx},\by^*(\Bar{\bx}))]^{-1} \| \nonumber \\
    \end{align}
    % {\color{red}{aryan: the assumption that the gradient of $f$ wrt $y$ is bounded might be a bit strong}} 
    where in the last inequality we used Assumptions~\ref{assum:upper} and \ref{assum:lower}-(iii) \na{along with the premises of Lemma~\ref{lem:v_b}-(II)}. %The  inequality~\eqref{pr:C_v1} is obtained by adding and subtracting the term $\nabla_{yy} g(\bx, \by^*(\bx))\nabla_y f(\Bar{\bx},\by^*(\bar{\bx}))$ along with applying triangle inequality.  
   %The inequality~\eqref{pr:C_v2} follows from the assumption~\ref{assum:upper} and assumption\ref{assum:lower}-(iii). 
   Moreover, for any invertible matrices $H_1$ and $H_2$, we have that 
   \begin{equation}\label{eq:invert}
        \| H_2^{-1} -H_1^{-1}\| = \|H_1^{-1} \big( H_1 - H_2\big) H_2^{-1}\| \leq \| H_1^{-1}\| \|H_2^{-1}\| \|H_1 -H_2\|.   
        \end{equation}
   Therefore, using the result of Lemma~\ref{lem:v_b}-(I) and \eqref{eq:invert} we can further bound inequality \eqref{eq:bound-v} as follows,
   \begin{align*}
    &\norm{\bv(\bx)-\bv(\bar\bx)} \nonumber\\
    &\leq \frac{1}{\mu_g}\big(  {L_{yx}^f} \| \bx - \Bar{\bx}\| + {L_{yy}^f} \bL_{\by}\| \bx - \Bar{\bx}\|\big) + C_{y}^{f}\| [\nabla^2_{yy}g(\bx,\by^*(\bx))]^{-1} - [\nabla^2_{yy}g(\Bar{\bx},\by^*(\Bar{\bx}))]^{-1} \| \\
        &\leq \frac{1}{\mu_g} \big( {L_{yx}^{f}} + {L_{yy}^{f}}\bL_{\by} \big)\|\bx -\Bar{\bx}\| + % C_{y}^{f}\big( \frac{L_{yy}^{g}}{\mu_g^2}\|\bx -\Bar{\bx}\|\big)\\
        {\frac{C_y^f}{\mu_g^2} L_{yy}^g\bigl( \|\bx-\bar{\bx}\|+\|\by^*(\bx)-\by^*(\bar{\bx})\|\bigr)} \\
        % &= \bigl( \red{\frac{L_{xx}^{f} + L_{xy}^{f}\bL_{\by}}{\mu_g} + C_{y}^{f}\frac{L_{yy}^{g}}{\mu_g^2}}\bigr) \|\bx -\Bar{\bx}\|.
        &= \bigl( {\frac{L_{yx}^{f} + L_{yy}^{f}\bL_{\by}}{\mu_g} + \frac{C_{y}^{f}L_{yy}^{g}}{\mu_g^2}(1+\bL_{\by})}\bigr) \|\bx -\Bar{\bx}\|. % = \bC_{\bv}\|\bx -\Bar{\bx}\|, \nonumber
   \end{align*}
The result follows by letting $\bC_\bv=\frac{L_{yx}^{f} + L_{yy}^{f}\bL_{\by}}{\mu_g} + \frac{C_{y}^{f}L_{yy}^{g}}{\mu_g^2}(1+\bL_{\by})$.
\end{proof}
% where the inequality~\eqref{pr:C_v4} follows from the fact that 
    %The third inequality we refer the reader to \cite[Lemma 2.2.]{ghadimi2018approximation}.
     % {\color{red}{aryan: clarify that the second inequality is obtained by adding and subtracting the term $\nabla_y f(\Bar{\bx},\by^*(x))$ } } \na{Sure,  I will add more explanation and details.}
\textbf{(III)} We start proving this part using the definition of $\nabla \ell (\bx)$ stated in \eqref{eq:ell-v}. %By adding and subtracting the term $\nabla_{yx}g(\Bar{\bx},\by^*(\Bar{\bx}))\bv(\bx)$ to the RHS followed by 
Utilizing the triangle inequality we obtain
     \begin{align}
         &\|\nabla \ell (\bx) - \nabla \ell(\Bar{\bx})\| \nonumber\\
         &\quad = \| \nabla_{x}f(\bx,\by^*(\bx)) - \nabla^2_{yx}g({\bx},\by^*({\bx})) \bv({\bx}) - \big( \nabla_{x}f(\Bar{\bx},\by^*(\Bar{\bx})) - \nabla^2_{yx}g(\Bar{\bx},\by^*(\Bar{\bx})) \bv(\Bar{\bx})\big)\| \nonumber \\
         & \quad \leq \| \nabla_{x}f(\bx,\by^*(\bx)) - \nabla_{x}f(\Bar{\bx},\by^*(\Bar{\bx}))\| + \| \big[\nabla^2_{yx}g(\Bar{\bx},\by^*(\Bar{\bx})) \bv(\Bar{\bx})  - \nabla^2_{yx}g(\Bar{\bx},\by^*(\Bar{\bx})) \bv(\bx)\big] \nonumber \\
         & \qquad + \big[\nabla^2_{yx}g(\Bar{\bx},\by^*(\Bar{\bx})) \bv(\bx)-\nabla^2_{yx}g({\bx},\by^*({\bx})) \bv({\bx})\big]\|
     \end{align}
\na{where the second term of the RHS follows from adding and subtracting the term $\nabla^2_{yx}g(\Bar{\bx}, \by^*(\Bar{\bx}))\bv(\bx)$}. Next, from Assumptions~\ref{assum:upper}-(i) and \ref{assum:lower}-(v) together with the triangle inequality application we conclude that
      \begin{align}
         \|\nabla \ell (\bx) - \nabla \ell(\Bar{\bx})\| &\leq L_{xx}^f \| \bx - \Bar{\bx}\| + L_{xy}^f \| \by^*(\bx) - \by^*(\Bar{\bx})\|\  + C_{yx}^g \|\bv(\Bar{\bx}) - \bv(\bx)\| \nonumber \\
         & \quad + \frac{C_{y}^{f}}{\mu_g} \|\nabla^2_{yx}g(\Bar{\bx},\by^*(\Bar{\bx})) -\nabla^2_{yx}g({\bx},\by^*({\bx}))\| 
         \end{align}
    It should be  that in the last inequality, we use the fact that $\|\bv(\bx)\| = \| [\nabla^2_{yy}g(\bx,\by^*(\bx))]^{-1} \nabla_{y}f(\bx,\by^*(\bx))\| \leq \frac{ C_{y}^{f}}{\mu_g}$. Combining the result of Lemma~\ref{lem:v_b} part (I) and (II) with the Assumption~\ref{assum:lower}-(iv) leads to 
    \begin{align}
         \|\nabla \ell (\bx) - \nabla \ell(\Bar{\bx})\| &\leq L_{xx}^f \| \bx - \Bar{\bx}\| + L_{xy}^f \bL_{\by}\|\bx - \Bar{\bx}\| + C_{yx}^g \bC_{\bv}\|\bx - \Bar{\bx}\| \nonumber\\
         &\quad + \frac{C_{y}^{f}}{\mu_g} L_{yx}^g \bigl( \|\bx-\bar{\bx}\|+\|\by^*(\bx)-\by^*(\bar{\bx})\|\bigr)\nonumber \\
         &\leq L_{xx}^f \| \bx - \Bar{\bx}\| + L_{xy}^f \bL_{\by}\|\bx - \Bar{\bx}\| + C_{yx}^g \bC_{\bv}\|\bx - \Bar{\bx}\| \nonumber\\
         &\quad + \frac{C_{y}^{f}}{\mu_g} L_{yx}^g \bigl( \|\bx-\bar{\bx}\|+\bL_{\by}\|\bx-\bar{\bx}\| \bigr) \nonumber \\
          &\leq \big(L_{xx}^f +{L_{xy}^f} \bL_{\by} +C_{yx}^g \bC_{\bv} +\frac{C_{y}^{f}}{\mu_g} L_{yx}^g(1+\bL_{\by})\big) \|\bx-\bar{\bx}\|
         \end{align}
The desired result can be obtained by letting $\bL_{\ell} = L_{xx}^f +{L_{xy}^f} \bL_{\by} +C_{yx}^g \bC_{\bv} +\frac{C_{y}^{f}}{\mu_g} L_{yx}^g(1+\bL_{\by})$. \qed % which completes the proof.    

\section{Required Lemmas for Theorems \ref{thm:convex-upper-bound} and \ref{thm:nonconvex-upper-bound}}\label{sec:required-lemma}
 \na{Before we proceed to the proofs of Theorems~\ref{thm:convex-upper-bound} and \ref{thm:nonconvex-upper-bound}, we present the following technical lemmas which quantify the error between the approximated solution $\by_k$ and $\by^*(\bx_k)$, as well as between $\bw_{k+1}$  and $\bv(\bx_{k})$.}
\begin{lemma}\label{lem:lower_GD_C}
    \ey{Suppose Assumption \ref{assum:lower} holds. Let $\{(\bx_k,\by_k)\}_{k\geq 0}$ be the sequence generated by Algorithm \ref{alg:In-BiCoG}, such that %using 
    %the Lemma~\ref{lem:v_b} and \ref{lem:lower_GD}, we can show that for $\alpha\leq 2/(\mu_g+L_g)$
    $\alpha=2/(\mu_g+L_g)$. Then, for any $k\geq 0$
    \begin{equation}\label{eq:lower_GD}
     \|\by_{k}- \by^*(\bx_{k})\|\leq \beta^k \|\by_0 - \by^*(\bx_0)\| + \bL_{\by}D_\cX\sum_{i=0}^{k-1}  \gamma_i \beta^{k-i}, 
\end{equation}}%
where $\beta\triangleq (L_g-\mu_g)/(L_g+\mu_g)$.
\end{lemma}
\begin{proof}
%\subsection{Proof of Lemma~\ref{lem:lower_GD_C}}
\ey{We begin the proof by characterizing the one-step progress of the lower-level iterate sequence $\{\by_{k}\}_k$. Indeed, at iteration $k$ we aim to approximate $\by^*(\bx_{k+1})=\argmin_{\by} g(\bx_{k+1},\by)$. According to the update of $\by_{k+1}$ we observe that}
\begin{align}\label{eq:lower_iter}
    \| \by_{k+1}- \by^*(\bx_{k+1}) \|^2 &=  \| \by_k- \by^*(\bx_{k+1}) - \alpha\grad_y g(\bx_{k+1},\by_k) \|^2 \nonumber \\
    &= \| \by_k- \by^*(\bx_{k+1})\|^2 - 2\alpha \langle \grad_y g(\bx_{k+1},\by_k),  \by_k - \by^*(\bx_{k+1}) \rangle \nonumber\\
    &\quad + \alpha^2 \|\grad_y g(\bx_{k+1},\by_k)\|^2.
\end{align}
Moreover, from Assumption \ref{assum:lower} {and following Theorem 2.1.12 in \citet{nesterov2018lectures}}, we have that
\begin{align}\label{eq:g-sc-lip}
    \langle \grad_y g(\bx_{k+1},\by_k), \by_k- \by^*(\bx_{k+1})\rangle \geq \frac{\mu_g L_g}{\mu_g + L_g}\|\by_k- \by^*(\bx_{k+1})\|^2 + \frac{1}{\mu_g + L_g}\|\grad_y g(\bx_{k+1},\by_k)\|^2
\end{align}
The inequality in \eqref{eq:lower_iter} together with \eqref{eq:g-sc-lip} imply that
\begin{align}\label{eq:lower_convrg}
    \| \by_{k+1}- \by^*(\bx_{k+1}) \|^2 
    &\leq  \| \by_k- \by^*(\bx_{k+1})\|^2  - \frac{2\alpha \mu_g L_g}{\mu_g + L_g}\| \by_k- \by^*(\bx_{k+1})\|^2 \nonumber\\
    &\quad +\Big ( \alpha^2 - \frac{2\alpha}{\mu_g + L_g} \Big ) \|\grad_y g(\bx_{k+1},\by_k)\|^2.
\end{align}
Setting the step-size $\alpha = \frac{2}{\mu_g + L_g}$ in \eqref{eq:lower_convrg} leads to
\begin{align}\label{eq:lower_convrg2}
    \|\by_{k+1}- \by^*(\bx_{k+1})\|^2 &\leq \Big(\frac{\mu_g - L_g}{\mu_g + L_g}\Big )^2\|\by_{k}- \by^*(\bx_{k+1})\|^2
\end{align}

\ey{Next, recall that $\beta=(L_g-\mu_g)/(L_g+\mu_g)$. Using the triangle inequality and Part (I) of Lemma \ref{lem:v_b} we conclude that } 
    %Let $\beta\triangleq (L_g-\mu_g)/(L_g+\mu_g)$, then we have:
    \begin{align}\label{eq:one-step-triangle}
    \|\by_{k+1}- \by^*(\bx_{k+1})\|&\leq \beta\|\by_{k}- \by^*(\bx_{k+1})\| \nonumber \\
    &\leq \beta\Big [ \|\by_{k}- \by^*(\bx_{k})\| +\|\by^*({\bx_k})- \by^*(\bx_{k+1})\|\Big ] \nonumber \\
    &\leq \beta\Big [ \|\by_{k}- \by^*(\bx_{k})\| + \bL_\by\|\bx_k - \bx_{k+1}\|\Big ] .
    %\\
    %&\leq \beta\Big [ \|\by_{k}- \by^*(\bx_{k})\| +\gamma_k \bL_\by D\Big ], \nonumber
\end{align}
\ey{Moreover, from the update of $\bx_{k+1}$ in Algorithm \ref{alg:In-BiCoG}} and boundedness of $\cX$ we have that $\norm{\bx_{k+1}-\bx_k}\leq \gamma_k \ey{D_\cX}$. \ey{Therefore, using this inequality within \eqref{eq:one-step-triangle} leads to 
\begin{equation*}
\norm{\by_{k+1}-\by^*(\bx_{k+1})}\leq \beta\norm{\by_k-\by^*(\bx_k)}+\beta\gamma_k\bL_\by D_\cX.
\end{equation*}
Finally, the desired result can be deduced from the above inequality recursively.}
%where in the last inequality we used boundedness of $\cX$.
%Using the above recursive relation one can obtain the desired result.
\end{proof}

%\subsection{Proof of Lemma~\ref{lem:nu_est}}
Previously, in Lemma~\ref{lem:lower_GD_C} \ey{we quantified} how close the approximation $\by_k$ is from the optimal solution $\by^*(\bx_k)$ of the inner problem. Now, in the following Lemma, we will find an upper bound for the error of approximating $\bv(\bx_k)$ via $\bw_{k+1}$.

\begin{lemma}\label{lem:nu_est}
Let $\{(\bx_k,\bw_k)\}_{k\geq 0}$ be the sequence generated by Algorithm \ref{alg:In-BiCoG}, such that $\gamma_k = \gamma$.
    Define $ \rho_k \triangleq (1- \eta_k \mu_g)$ and $\bC_1 \triangleq L_{yy}^{g} \frac{ C_{y}^{f}}{\mu_g} +  L_{yy}^{f}$. \ey{Under Assumptions \ref{assum:upper} and \ref{assum:lower} %along with Lemma~ \ref{lem:v_b} and \ref{lem:lower_GD_C} we have:
    we have that for any $k\geq 0$,}
    \begin{equation}
          \| {\bw_{k+1}} - \bv (\bx_{k})\| \leq \rho_k \|{\bw_{k}} - \bv (\bx_{k-1})\| + \rho_k \bC_{\bv}\gamma D_{\cX} + \eta_k \bC_1\big( \beta^k D_0 ^y + \bL_{\by} \gamma  \frac{\beta}{1-\beta} D_{\cX}\big).
    \end{equation}
\end{lemma}
\begin{proof}
\ey{From the optimality condition of \eqref{eq:obj-v} one can easily verify that} $\bv(\bx_k) = \bv(\bx_k) - \eta_k \big(\nabla^2_{yy}g(\bx_k,\by^*(\bx_k))\bv(\bx_k) - \nabla_{y}f(\bx_k,\by^*(\bx_k))\big)$. \ey{Now using definition of $\bw_{k+1}$} we can write
    \begin{align}
    \| {\bw_{k+1}} - \bv (\bx_{k})\|  &= \Big\| \Big(\bw_k -\eta_k (\nabla^2_{yy}g(\bx_k,\by_k)\bw_{k} - \nabla_{y}f(\bx_k,\by_k))\Big) - \Big(\bv(\bx_k)   \nonumber \\ 
    &\quad - \eta_k \big(\nabla^2_{yy}g(\bx_k,\by^*(\bx_k))\bv(\bx_k)- \nabla_{y}f(\bx_k,\by^*(\bx_k))\big)\Big)\Big\| \nonumber \\
    &= \Big\|  \Big(I - \eta_k \nabla^2_{yy}g(\bx_k,\by_k)\Big)(\bw_{k} - \bv (\bx_{k})) - \eta_k \Big(\nabla^2_{yy}g(\bx_k,\by_k)  \nonumber \\
    &\quad - \nabla^2_{yy}g(\bx_k,\by^*(\bx_k))\Big)\bv (\bx_{k})+ \eta_k \Big( \nabla_{y}f(\bx_k,\by^*(\bx_k)) -  \nabla_{y}f(\bx_k,\by_k)\Big)\Big\|, \label{pr:w-v-1} 
    \end{align}
    where the last equality is obtained by adding and subtracting the term $(I - \eta_k \nabla^2_{yy}g(\bx_k, \by_k)) \bv(\bx_k)$. Next, using Assumptions~\ref{assum:upper} and \ref{assum:lower} along with the application of the triangle inequality we obtain
    \begin{align}
    \| {\bw_{k+1}} - \bv (\bx_{k})\| &\leq (1 - \eta_k \mu_g)\|{\bw_{k}} - \bv (\bx_{k})\| + \eta_k L_{yy}^{g} \| \by_k - \by^*(\bx_k)\|\|\bv(\bx_k)\| \nonumber\\
    &\quad + \eta_k L_{yy}^{f} \| \by_k - \by^*(\bx_k)\|.
    \end{align}
    Note that $\|\bv(\bx_k)\| = \| [\nabla^2_{yy}g(\bx,\by^*(\bx))]^{-1} \nabla_{y}f(\bx,\by^*(\bx))\| \leq \frac{ C_{y}^{f}}{\mu_g}$. Now, by adding and subtracting $\bv(\bx_{k-1})$ to the term $\| {\bw_{k}} - \bv (\bx_{k})\|$  followed by triangle inequality application we can conclude that
    \begin{align}
      \| {\bw_{k+1}} - \bv (\bx_{k})\| &\leq (1 - \eta_k \mu_g)\|{\bw_{k}} - \bv (\bx_{k-1})\|+(1 - \eta_k \mu_g)\|\bv (\bx_{k-1}) - \bv (\bx_{k})\| \nonumber\\
      & \quad+ \eta_k \Big( L_{yy}^{g} \frac{ C_{y}^{f}}{\mu_g} +  L_{yy}^{f}\Big) \| \by_k - \by^*(\bx_k)\| .\label{pr:w-v-3}
    \end{align}
   Therefore, using the result of Lemma~\ref{lem:lower_GD_C},  we can further bound inequality~\eqref{pr:w-v-3} as follows
   \begin{align}
     \| {\bw_{k+1}} - \bv (\bx_{k})\| &\leq (1 - \eta_k \mu_g)\|{\bw_{k}} - \bv (\bx_{k-1})\|+(1 - \eta_k \mu_g)\bC_{\bv}\| \bx_{k-1} - \bx_k\|  \nonumber \\ 
     &\quad +  \eta_k \bC_1\| \by_k - \by^*(\bx_k)\| \nonumber\\
     &\leq \rho_k \|{\bw_{k}} - \bv (\bx_{k-1})\| + \rho_k \bC_{\bv}\gamma D_{\cX} + \eta_k \bC_1\big( \beta^k D_0 ^y + \bL_{\by}   \gamma  \frac{\beta}{1-\beta}D_{\cX}\big)
     \end{align}
     where the last inequality follows from the boundedness assumption of set $\cX$, recalling that  $D_0^y=\| \by_0 - \by^*(\bx_0)\| $, and the fact that $\sum_{i=0}^{k-1} \beta^{k-i}  \leq \frac{\beta}{1-\beta}$.
    % The inequality~\eqref{pr:w-v-1} is obtained by adding and subtracting the term $(I - \eta_k \nabla_{yy}g(\bx_k, \by_k)) v(\bx_k)$. Combining assumptions~\ref{assum:upper} and \ref{assum:lower} along with the triangle inequality will obtain \eqref{pr:w-v-2}. Note that $\|\bv(\bx_k)\| = \| [\nabla_{yy}g(\bx,\by^*(\bx))]^{-1} \nabla_{y}f(\bx,\by^*(\bx))\| \leq \frac{ C_{y}^{f}}{\mu_g}$, then dding and subtracting $v(\bx_{k-1})$ clearly implies inequality~\eqref{pr:w-v-3}. Combining the result of Lemma~\ref{lem:lower_GD_C} along with $\| \by_0 - \by^*(\bx_0)\| \leq D_0^y$ and $\sum_{i=0}^{k-1} \beta^{k-i}  \leq \frac{\beta}{1-\beta}$ will obtain the desired result. 
\end{proof}

%\subsection{Proof of Lemma~\ref{lem:nu_est_K}}
\begin{lemma}\label{lem:nu_est_K}
Let $\{(\bx_k,\bw_k)\}_{k\geq 0}$ be the sequence generated by Algorithm \ref{alg:In-BiCoG} with step-size $\eta_k=\eta< \frac{1-\beta}{\mu_g}$ where $\beta$ is defined in Lemma \ref{lem:lower_GD_C}. 
Suppose that Assumption~\ref{assum:lower} holds and  $\bv(\bx_{-1}) = \bv(\bx_0)$, then for any $K\geq 1$, 
%For the sake of simplicity, we assume that $\eta_k = \eta$ which implies that $\rho_k = \rho$. By using the recursive relation obtained in Lemma~\ref{lem:nu_est} we have:
    \begin{align}
     \| \bw_{K} - \bv(\bx_{K-1})\| \leq \rho^K \|{\bw_0} - \bv(\bx_0) \| + \frac{\gamma \rho  \bC_{\bv} D_{\cX}}{1-\rho} + \frac{\eta \bC_1 D_0^y \rho^{K+1}}{\rho - \beta} +\frac{\gamma \eta \beta \bC_1 \bL_{\by} D_{\cX} }{(1-\rho) (1-\beta)} ,
    \end{align}
    where $\rho\triangleq 1-\eta\mu_g$. 
\end{lemma}
\begin{proof}
Applying the result of Lemma~\ref{lem:nu_est} recursively for $k=0$ to $K-1$, one can conclude that 
\begin{align}
     \| \bw_{K} - \bv(\bx_{K-1})\| &\leq \rho^K \|{\bw_0} - \bv(\bx_0) \| + \bC_{\bv} \gamma D_{\cX} \sum _ {i=1} ^ K \rho^i+ \eta \bC_1 \sum _{i=0}^K \big( \beta^i D_0^y + \gamma \bL_{\by} D_{\cX} \frac{\beta}{1-\beta}\big) \rho^{K-i}  \nonumber \\
     &\leq   \rho^K \|{\bw_0} - \bv(\bx_0) \| + \frac{\rho}{1- \rho}\bC_{\bv} \gamma D_{\cX} + \eta \bC_1  D_0^y \big(\sum _{i=0}^K  \beta^i\rho^{K-i}\big) \nonumber\\
     &\quad + \frac{\gamma \eta \beta \bC_1 \bL_{\by} D_{\cX} }{1-\beta}\sum _{i=0}^K\rho^{K-i} \label{pr:nu-es-K-1},
    \end{align}
    where the last inequality is obtained by noting that $\sum_{i=1}^{K} \rho^{i}  \leq \frac{\rho}{1-\rho}$. Finally, the choice $\eta<\frac{1-\beta}{\mu_g}$ implies that $\beta < \rho$, hence,  $\sum_{i=0}^{K} (\frac{\beta}{\rho})^i  \leq \frac{\rho}{\rho-\beta}$ which leads to the desired result.
    % \begin{align}
    %    \| \bw_{K} - \bv(\bx_{K-1})\| & \leq \rho^K \|{\bw_0} - \bv(\bx_0) \| + \frac{\rho}{1- \rho}\bC_{\bv} \gamma D_{\cX} + \eta \bC_1  D_0^y \rho ^K \big(\sum _{i=0}^K  (\frac{\beta}{\rho})^i \big) + \frac{\gamma \eta \beta \bC_1 \bL_{\by} D_{\cX} }{1-\beta}\big(\sum _{i=0}^K\rho^{K-i}\big) \nonumber \\
    %  & \leq \rho^K \|{\bw_0} - \bv(\bx_0) \| + \frac{\gamma \rho  \bC_{\bv} D_{\cX}}{1-\rho} + \frac{\eta \bC_1 D_0^y \rho^{K+1}}{\rho - \beta} +\frac{\gamma \eta \beta \bC_1 \bL_{\by} D_{\cX} }{(1-\rho)(1-\beta)}  .
    % \end{align} 
    % The inequality~\eqref{pr:nu-es-K-1} follows from the fact that $\sum_{i=1}^{K} \rho^{i}  \leq \frac{\rho}{1-\rho}$. The last inequality is obtained by assuming that $\beta \leq \rho$, and $\sum_{i=0}^{K} (\frac{\beta}{\rho})^i  \leq \frac{\rho}{\rho-\beta}$.
\end{proof}
\subsection{Proof of Lemma~\ref{lem:ell-es}}\label{proof:lemma-grad-est}
We begin the proof by considering the definition of $\nabla \ell (\bx_k)$ and $F_k$ followed by a triangle inequality to obtain
    \begin{align}
        \| \nabla \ell (\bx_k) - F_k\|
        % &= \|\nabla_{x}f(\bx_k,\by^*(\bx_k)) - \nabla_{yx}g(\bx_k,\by^*(\bx_k)) \bv(\bx_k) - \big( \nabla_{x}f(\bx_k, \by_k)- \nabla_{yx}g(\bx_k, \by_k)\bw_{k+1}\big)\| \nonumber \\
        & \leq \|\nabla_{x}f(\bx_k,\by^*(\bx_k)) - \nabla_{x}f(\bx_k, \by_k)\| \nonumber\\
        &\quad + \| \nabla^2_{yx}g(\bx_k, \by_k)\bw_{k+1} - \nabla^2_{yx}g(\bx_k,\by^*(\bx_k)) \bv(\bx_k)\| \label{pr:ell_est_1} 
        \end{align}
Combining Assumption~\ref{assum:upper}-(i) together with adding and subtracting  $\nabla^2_{yx}g(\bx_k,\by_k)\bv(\bx_k)$ to the second term of RHS
% $\| \nabla^2_{yx}g(\bx_k, \by_k)\bw_{k+1} - \nabla^2_{yx}g(\bx_k,\by^*(\bx_k)) \bv(\bx_k)\| $  
lead to 
    \begin{align}\label{eq:ell-F-first-bound}
        \| \nabla \ell (\bx_k) - F_k\| &\leq \na{L_{xy}^{f}} \| \by_k - \by^*(\bx_k)\| + \| \nabla^2_{yx}g(\bx_k, \by_k)\big( \bw_{k+1} - \bv(\bx_k)\big) + \big( \nabla^2_{yx}g(\bx_k, \by_k) \nonumber \\
        &\quad - \nabla^2_{yx}g(\bx_k,\by^*(\bx_k)) \big)  \bv(\bx_k)\| \nonumber \\
        &\leq \na{L_{xy}^{f}} \| \by_k - \by^*(\bx_k)\| + C_{yx}^{g}\|\bw_{k+1} - \bv(\bx_k)\| +   L_{yx}^{g} \frac{C_{y}^{f}}{\mu_g} \| \by_k - \by^*(\bx_k)\| 
        \end{align}
where the last inequality is obtained using Assumption~\ref{assum:lower} and the triangle inequality. Next, utilizing Lemma~\ref{lem:lower_GD_C} and \ref{lem:nu_est_K} we can further provide upper-bounds for the term in RHS of \eqref{eq:ell-F-first-bound} as follows
    \begin{align*}
       \| \nabla \ell (\bx_k) - F_k\| 
        % &= \bC_2 \| \by_k - \by^*(\bx_k)\| +C_{yx}^{g}\|\bw_{k+1} - \bv(\bx_k)\| \nonumber\\
        &\leq \bC_2 \big( \beta^k D_0^y  + \frac{\gamma \beta \bL_{\by} D_{\cX} }{1-\beta} \big)+ C_{yx}^{g} \Big(\rho^{k+1} \|{\bw_0} - \bv(\bx_0) \| + \frac{\gamma \rho  \bC_{\bv} D_{\cX}}{1-\rho} \nonumber \\
         &\quad + \frac{\eta \bC_1 D_0^y \rho^{k+2}}{\rho - \beta}+\frac{\gamma \eta \beta \bC_1 \bL_{\by} D_{\cX} }{(1-\rho)(1-\beta)}\Big) . 
    \end{align*}%
    \qed
    % The inequality~\eqref{pr:ell_est_1} follows from the application of the triangle inequality. The inequality~\eqref{pr:ell_est_2} is obtained by utilizing assumptions~\ref{assum:upper}-(i) and adding and subtracting the term $\nabla_{yx}g(\bx_k,\by_k)\bv(\bx_k)$.
    % Combining assumptions~\ref{assum:lower} with triangle inequality, we obtain inequality~\eqref{pr:ell_est_3}. Utilizing Lemma~\ref{lem:lower_GD_C} and \ref{lem:nu_est_K} clearly implies the last inequality.
\subsection{Improvement in one step} 
In the following, we characterize the improvement of the objective function $\ell(\bx)$ after taking one step of Algorithm~\ref{alg:In-BiCoG}.
\begin{lemma}\label{lem:one-step}
	Let $\{\bx_k\}_{k=0}^{K}$ be the sequence generated by Algorithm~\ref{alg:In-BiCoG}. Suppose Assumptions \ref{assum:upper} and \ref{assum:lower} hold and $\gamma_k = \gamma$, then for any $k\geq 0$ we have 
	\begin{align}\label{eq:one-step-f}
	\ell(\bx_{k+1}) &\leq \ell(\bx_{k}) - \gamma \mathcal{G}          (\bx_k) + \gamma \bC_2 \beta^k D_0^y D_{\cX} + \frac{\gamma^2 \bC_2 D_{\cX}^2       \bL_{\by} \beta}{1-\beta} +   C_{yx}^{g}\Big[ \gamma D_{\cX}            \rho^{k+1} \|{\bw_0} - \bv(\bx_0) \|  \nonumber \\
        &\quad + \frac{\gamma^2 D_{\cX}^2      \rho \bC_{\bv}}{1-\rho}+ \frac{\gamma D_{\cX} D_0^y \bC_1\eta \rho^{k+2}}{\rho - \beta} + \frac{\gamma^2 D_{\cX}^2 \bL_{\by} \bC_1 \beta \eta}{(1-\beta)(1-\rho)}\Big] + \frac{1}{2} \bL_{\ell} \gamma^2 D_{\cX}^2.
	\end{align}
\end{lemma}
\begin{proof}
Note that according to Lemma \ref{lem:v_b}-(III), $\ell(\cdot)$ has a Lipschitz continuous gradient which implies that 
  %  Since the gradient of the bilevel objective function is Lipschitz continuous according to Lemma \ref{lem:v_b}-(III), and set $\mathcal{X}$ is bounded, we have 
\begin{align}%\label{eq:step_fw}
    \ell(\bx_{k+1}) &\leq 
    % \ell(\bx_{k}) +\langle{\nabla \ell(\bx_{k}), \bx_{k+1} - \bx_k}\rangle +\frac{1}{2}\b\bL_{\ell}\|\bx_{k+1} - \bx_k\|^2 \nonumber \\
    % &=
    \ell(\bx_{k}) + \gamma \langle{\nabla \ell(\bx_{k}), \bs_{k} - \bx_k}\rangle +\frac{1}{2}\bL_{\ell}\gamma^2\|\bs_{k} - \bx_k\|^2\nonumber \\
    &= \ell(\bx_{k}) + \gamma \langle{F_k, \bs_{k} - \bx_k}\rangle + \gamma \langle{\nabla \ell (\bx_k) - F_k, \bs_{k} - \bx_k}\rangle +\frac{1}{2}\bL_{\ell}\gamma^2\|\bs_{k} - \bx_k\|^2 ,\label{pr:step-fw-1} 
 \end{align}
 where the last inequality follows from adding and subtracting the term $\gamma \langle{F_k, \bs_{k} - \bx_k}\rangle$ to the RHS. Define $\bs'_k = \argmax_{\bs\in \cX} \{\langle \grad \ell(\bx_k),\bx_k-\bs\rangle \}$ and observe that $\mathcal{G}(\bx_k) = \langle \grad \ell(\bx_k),\bx_k-\bs'_k\rangle$ by Definition~\ref{def:FW_gap}. Using the definition of $\bs_k$,  
 % and  $\mathcal{G}(x)$ together with assuming that ${\bs_k^{\prime}}$ is the maximizer of  $\mathcal{G}({\bx}_k)$, 
 we can immediately observe that
 \begin{align}
    \langle{F_k, {\bs}_{k} - {\bx}_k}\rangle 
     &= \min_{{\bs} \in \mathcal{X}} \langle{F_k, {\bs} - {\bx}_k}\rangle \nonumber \\ 
     &\leq \langle{F_k, {\bs_{k}^{\prime}}- {\bx_k}}\rangle \nonumber \\ 
     &= \langle{\nabla \ell({\bx_k}), {\bs_{k}^{\prime}}- {\bx_k}}\rangle +\langle{ F_k -\nabla \ell({\bx_k}), {\bs_{k}^{\prime}}- {\bx_k}}\rangle \nonumber \\
     &= - \cG(\bx_k) + \langle{ F_k -\nabla \ell({\bx}_k), {\bs_{k}^{\prime}}- {\bx_k}}\rangle.
     \label{pr:fw_upper_bound}
\end{align}

Next, combining \eqref{pr:step-fw-1} with \eqref{pr:fw_upper_bound} followed by the Cauchy-Schwartz inequality leads to 
    \begin{align}
         \ell({\bx}_{k+1}) \leq \ell({\bx}_{k}) - \gamma \mathcal{G}({\bx}_k)+\gamma  \|{\nabla \ell({\bx}_{k}) -F_k\| \|{\bs}_{k} - {\bs_{k}^{\prime}}}\|+\frac{1}{2} {\bL}_{\ell}\gamma^2\|{\bs}_{k} - {\bx}_k\|^2.\label{pr:step-fw-3} 
     \end{align}     
Finally, using the result of the Lemma~\ref{lem:ell-es}  together with the boundedness assumption of set $\cX$ we conclude the desired result.
\end{proof}

\section{Proof of Theorem~\ref{thm:convex-upper-bound}}\label{appen:convex}
Since $\ell $ is convex, from the definition of $\mathcal{G}(\bx_k)$ in \eqref{eq:FW_gap} we have
\begin{align}\label{eq:conv-G}
    \mathcal{G}(\bx_k) = \max_{\bs\in \cX}\{\langle {\grad \ell(\bx_k),\bx_k-\bs}\}\rangle \geq \langle{\grad \ell(\bx_k),\bx_k-\bx^*}\rangle \geq \ell(\bx_k)- \ell(\bx^*).
\end{align}
% Using the smoothness assumption of $\ell(\bx)$ we can write
% \begin{align}
%     \ell(\bx_{k+1}) &\leq \ell(\bx_k) + \langle \nabla \ell(\bx_k) , \bx_{k+1} - \bx_k \rangle + \frac{1}{2} \b\bL_{\ell} \|\bx_{k+1} - \bx_k \|^2 \nonumber \\
%     &= \ell(\bx_{k}) + \gamma_k \langle{\nabla \ell(\bx_{k}), \bs_{k} - \bx_k}\rangle +\frac{1}{2}\b\bL_{\ell}\gamma_k^2\|\bs_{k} - \bx_k\|^2 \label{pr:con_upp}
% \end{align}
% Here, in the last expression we have replaced
% the term $\bx_{k+1} - \bx_{k} = \gamma_k (s_k  - \bx_k)$. Now adding and subtracting $\gamma_k \langle F_k, s_k - \bx_k \rangle $ in \eqref{pr:con_upp} we get
We assume a fixed step-size in Theorem~\ref{thm:convex-upper-bound} and we set $\gamma_k=\gamma$. Combining the result of Lemma~\ref{lem:one-step} with \eqref{eq:conv-G} leads to
    \begin{align}\label{eq:convex-one-step1}
	\ell(\bx_{k+1}) &\leq \ell(\bx_{k}) - \gamma (\ell(\bx_k ) - \ell(\bx^*)) + \gamma \bC_2 \beta^k D_0^y D_{\cX} + \frac{\gamma^2 \bC_2 D_{\cX}^2       \bL_{\by} \beta}{1-\beta} +   C_{yx}^{g}\Big[ \gamma D_{\cX}            \rho^{k+1} \|{\bw_0} - \bv(\bx_0) \| \nonumber \\
    &\quad  + \frac{\gamma^2 D_{\cX}^2      \rho \bC_{\bv}}{1-\rho} 
    + \frac{\gamma D_{\cX} D_0^y \bC_1\eta \rho^{k+2}}{\rho - \beta} + \frac{\gamma^2 D_{\cX}^2 \bL_{\by} \bC_1 \beta \eta}{(1-\beta)(1-\rho)}\Big] + \frac{1}{2} \bL_{\ell} \gamma^2 D_{\cX}^2.
	\end{align}
Subtracting $\ell(\bx^*)$ from both sides, we get 
     \begin{align}\label{eq:convex-one-step2}
    \ell(\bx_{k+1}) - \ell(\bx^*) &\leq (1- \gamma )( \ell(\bx_k) - \ell(\bx^* )) + \cR_k(\gamma),
	\end{align}
 where
 \begin{align}\label{eq:R_def}
     \cR_k(\gamma)&\triangleq \gamma \bC_2 \beta^k D_0^y D_{\cX} + \frac{\gamma^2 \bC_2 D_{\cX}^2       \bL_{\by} \beta}{1-\beta} +   C_{yx}^{g}\Big[ \gamma D_{\cX}            \rho^{k+1} \|{\bw_0} - \bv(\bx_0) \| \nonumber \\
    & \quad + \frac{\gamma^2 D_{\cX}^2      \rho \bC_{\bv}}{1-\rho} 
    + \frac{\gamma D_{\cX} D_0^y \bC_1\eta \rho^{k+2}}{\rho - \beta} + \frac{\gamma^2 D_{\cX}^2 \bL_{\by} \bC_1 \beta \eta}{(1-\beta)(1-\rho)}\Big] + \frac{1}{2} \bL_{\ell} \gamma^2 D_{\cX}^2.
 \end{align}
Continuing \eqref{eq:convex-one-step2} recursively leads to the desired result. \qed
%implies that
% \begin{align}
%     \ell(\bx_K) - \ell(\bx^*) &\leq (1- \gamma )^K( \ell(\bx_0) - \ell(\bx^* )) + \sum_{k=0}^{K-1}(1-\gamma)^{K-k}\cR_k(\gamma).
% \end{align}

\section{Proof of Corollary~\ref{cr:convex_upper-bound}}\label{appen:convex_coro}
We start the proof by using the result of the Theorem~\ref{cr:convex_upper-bound}, i.e., 
\begin{align}\label{eq:thm-conv-result}
    \ell(\bx_K) - \ell(\bx^*) &\leq (1- \gamma )^K( \ell(\bx_0) - \ell(\bx^* )) + \sum_{k=0}^{K-1}(1-\gamma)^{K-k}\cR_k(\gamma).
\end{align}
Note that 
\begin{align*}
     &\sum_{k=0}^{K-1}(1-\gamma)^{K-k}\cR_k(\gamma)\nonumber\\
     &=  \bC_2  D_0^y D_{\cX} \Big[\sum_{k=0}^{K-1}(1-\gamma)^{K-k} \gamma \beta^k \Big] + \frac{\bC_2 D_{\cX}^2 \bL_{\by} \beta}{1-\beta} \Big[\sum_{k=0}^{K-1}(1-\gamma)^{K-k} \gamma^2 \Big]  \nonumber \\
     & \quad +   C_{yx}^{g}\Big(  \rho D_{\cX}             \|{\bw_0} - \bv(\bx_0) \| \Big[\sum_{k=0}^{K-1}(1-\gamma)^{K-k} \gamma \rho^{k} \Big]  + \frac{ D_{\cX}^2      \rho \bC_{\bv}}{1-\rho}\Big[\sum_{k=0}^{K-1}(1-\gamma)^{K-k} \gamma^2 \Big] \nonumber \\ 
    & \quad + \frac{ D_{\cX} D_0^y \bC_1\eta \rho^{2}}{\rho - \beta} \Big[\sum_{k=0}^{K-1}(1-\gamma)^{K-k} \gamma \rho^k \Big]+ \frac{ D_{\cX}^2 \bL_{\by} \bC_1 \beta \eta}{(1-\beta)(1-\rho)} \Big[\sum_{k=0}^{K-1}(1-\gamma)^{K-k} \gamma^2 \Big]\Big)  \nonumber \\
    & \quad + \frac{1}{2} \bL_{\ell} D_{\cX}^2 \Big[\sum_{k=0}^{K-1}(1-\gamma)^{K-k} \gamma^2 \Big].
 \end{align*}

%\naz{Recall the definition of $\mathcal{R}_k(\gamma)$ in \eqref{eq:R_def}, we observe that the major components contributing to this term are dominated by $\gamma$ since it is the parameter that is going to vanish. Therefore, to simplify the analysis we focus on the terms containing $\gamma$ and rewrite $\mathcal{R}_k(\gamma)$ as an order of these terms by ignoring constants. Hence, by considering the fact that $ \beta < \rho$, we can express $\mathcal{R}_k(\gamma) = \mathcal{O}(\gamma\rho^k + \gamma^2)$, capturing the dominant order of these terms. The other terms in $\mathcal{R}_k(\gamma)$ are considered minor in comparison to these leading terms when $\gamma$ goes to zero.}
Moreover, one can easily verify that $\sum_{k=0}^{K-1}(1-\gamma)^{K-k}\gamma^2\leq \gamma(1-\gamma)$ and $\sum_{k=0}^{K-1}(1-\gamma)^{K-k}\gamma\rho^k\leq \frac{\gamma(1-\gamma)}{\abs{1-\gamma-\rho}}$ from which together with the above inequality we conclude that 
% $\sum_{k=0}^{K-1}(1-\gamma)^{K-k}\cR_k(\gamma)=\cO(\gamma)$.
\begin{align}
     &\sum_{k=0}^{K-1}(1-\gamma)^{K-k}\cR_k(\gamma)\nonumber\\
     &\leq   \frac{ \bC_2  D_0^y D_{\cX}\gamma(1-\gamma)}{\abs{1-\gamma-\beta}} + \frac{ \bC_2 D_{\cX}^2 \bL_{\by} \beta \gamma(1-\gamma)}{1-\beta}  +   C_{yx}^{g}\Big(  \frac{   D_{\cX} \rho \gamma(1-\gamma)}{\abs{1-\gamma-\rho}} \|{\bw_0} - \bv(\bx_0) \|  \nonumber \\ 
     & \quad + \frac{ D_{\cX}^2 \bC_{\bv} \rho \gamma(1-\gamma) }{1-\rho}  + \frac{ D_{\cX} D_0^y \bC_1\eta \rho^{2} \gamma(1-\gamma)}{(\rho - \beta) \abs{1-\gamma-\rho}} + \frac{ D_{\cX}^2 \bL_{\by} \bC_1 \eta \beta \gamma(1-\gamma)}{(1-\beta)(1-\rho)} \Big) + \frac{1}{2} \bL_{\ell} D_{\cX}^2\gamma(1-\gamma) \nonumber \\
     &= \cO\left(\frac{\bC_{\bv} \rho }{1-\rho}\gamma+\frac{\bL_{\by} \bC_1 \beta}{(1-\beta)(1-\rho)}\gamma\right).
 \end{align}
%\naz{It should be noted that the obtained bound for$\sum_{k=0}^{K-1}(1-\gamma)^{K-k}\mathcal{R}_k(\gamma)$ is dominated by the term $\frac{\gamma D_{\mathcal{X}}^2 \rho \mathbf{C}_{\mathbf{v}}}{1-\rho}$. Moreover, taking into account that that $\sum_{k=0}^{K-1}(1-\gamma)^{K-k} \leq \frac{1}{\gamma}$ and $\mathcal{R}_k (\gamma) =\mathcal{O}(\gamma\rho^k+\gamma^2)$, the above result will be obtained}.
Using the above inequality within \eqref{eq:thm-conv-result} we conclude that $\ell(\bx_K) - \ell(\bx^*) \leq (1- \gamma )^K( \ell(\bx_0) - \ell(\bx^* )) + \cO(\frac{\bC_{\bv} \rho }{1-\rho}\gamma+\frac{\bL_{\by} \bC_1 \beta}{(1-\beta)(1-\rho)}\gamma)$ where $\mathbf{C}_{\mathbf{v}} = \cO(\kappa_g^3)$, $\bC_1=\cO(\kappa_g^2)$, $\bL_\by=\cO(\kappa_g)$ as shown in Lemma~\ref{lem:v(x)} and $\min\{1-\rho,1-\beta\} = \Omega(\frac{1}{\kappa_g})$ as shown in Lemma~\ref{lem:ell-es}. 
Next, we show that by selecting $\gamma=\log(K)/K$ we have that $(1-\gamma)^K\leq 1/K$. In fact, for any $x>0$, $\log(x)\geq 1-\frac{1}{x}$ which implies that $\log(\frac{1}{1-\gamma})\geq \gamma=\log(K)/K$, hence, $(\frac{1}{1-\gamma})^K\geq K$. Putting the pieces together we conclude that $\ell(\bx_K)-\ell(\bx^*) = \cO((1-\gamma)^K (\ell(\bx_0)-\ell(\bx^*)) + \gamma \kappa_g^5) = \tilde{\cO}(\kappa_g^5 / K)$, which leads to an iteration complexity of $\tilde{\cO}(\kappa_g^5\epsilon^{-1})$.

Furthermore, assuming that $\grad_y f(\bx,\cdot)$ is uniformly bounded for any $\bx\in\cX$, we conclude that $C_y^f=\cO(1)$, hence, $\bC_1=\cO(\kappa_g)$ from which we have that $\ell(\bx_K)-\ell(\bx^*)=\cO((1-\gamma)^K(\ell(\bx_0)-\ell(\bx^*))+\gamma\kappa_g^4)$. Therefore, selecting $\gamma=\log(K)/K$ implies that $\ell(\bx_K)-\ell(\bx^*)=\cO(\kappa_g^4/K)$ which leads to an iteration complexity of $\cO(\kappa_g^4\epsilon^{-1})$.
% $\ell(\bx_K) - \ell(\bx^*) \leq \cO(\log(K)/K)$. Therefore, to achieve $\ell(\bx_K) - \ell(\bx^*) \leq \epsilon$, Algorithm \ref{alg:In-BiCoG} requires at most $\cO(\epsilon^{-1}\log(\epsilon^{-1}))$ iterations.  
\qed     
 
\section{Proof of Theorem~\ref{thm:nonconvex-upper-bound}}\label{appen: nonconvex}
    %We assume a fixed step-size in Theorem~\ref{thm:nonconvex-upper-bound}, in the following, we will write $\gamma_k=\gamma$. 
    Recall that from Lemma~\ref{lem:one-step} we have
    \begin{align*}
    %     \gamma \mathcal{G}_{k} &\leq \ell(\bx_{k}) - \ell(\bx_{k+1}) + \gamma \bC_2 \beta^k D_0^y D_{\cX} + \frac{\gamma^2 D_{\cX}^2       \bL_{\by} \beta}{1-\beta} +   C_{yx}^{g}\Big[ \gamma D_{\cX}            \rho^{k+1} \|{\bw_0} - \bv(\bx_0) \| + \frac{\gamma^2 D_{\cX}^2      \rho \bC_{\bv}}{1-\rho} \nonumber \\
    %     &+ \frac{\gamma D_{\cX} D_0^y \bC_1\eta \rho^{k+2}}{\rho - \beta} + \frac{\gamma^2 D_{\cX}^2 \bL_{\by} \bC_1 \beta \eta}{(1-\beta)(1-\rho)}\Big] + \frac{1}{2} \bL_{\ell} \gamma^2 D_{\cX}^2 \nonumber \\
    % \end{align}
    % dividing both sides with $\gamma$ we obtain:
        \mathcal{G}(\bx_{k}) &\leq \frac{\ell(\bx_{k}) - \ell(\bx_{k+1})}{\gamma}+\bC_2  \beta^k D_0^y D_{\cX}+ \frac{ \gamma \bC_2 D_{\cX}^2       \bL_{\by} \beta}{1-\beta} + C_{yx}^{g}\Big[ D_{\cX}            \rho^{k+1} \|{\bw_0} - \bv(\bx_0) \|  \nonumber \\
        &\quad + \frac{\gamma D_{\cX}^2    \rho \bC_{\bv}}{1-\rho}+ \frac{ D_{\cX} D_0^y \bC_1\eta \rho^{k+2}}{\rho - \beta} + \frac{\gamma D_{\cX}^2 \bL_{\by} \bC_1 \beta \eta}{(1-\beta)(1-\rho)}\Big] + \frac{1}{2} \bL_{\ell} \gamma D_{\cX}^2. 
    \end{align*}
    Summing both sides of the above inequality from $k=0$ to $K-1$, we get
\begin{align*}
	\sum_{k=0}^{K-1}\mathcal{G}(\bx_k) &\leq \frac{\ell(\bx_0) -      \ell(\bx_K)}{\gamma}+ \frac{\bC_2 D_0^y D_{\cX}}{1- \beta}+       
        K \frac{\gamma \bC_2 D_{\cX}^2\bL_{\by} \beta}{1-\beta} + 
        C_{yx}^{g}\Big[ \frac{\rho D_{\cX} \|{\bw_0} - \bv(\bx_0) \|}{ 1-\rho}   \nonumber \\
        &\quad + K \frac{\gamma D_{\cX}^2    \rho \bC_{\bv}}{1-\rho} + \frac{ D_{\cX} D_0^y \bC_1\eta \rho^{2}}{(1-\rho)(\rho - \beta)} + K \frac{\gamma D_{\cX}^2 \bL_{\by} \bC_1 \beta \eta}{(1-\beta)(1-\rho)}\Big] + \frac{K}{2} \bL_{\ell} \gamma D_{\cX}^2,
\end{align*}
where in the above inequality we use the fact that $\sum_{i=0}^{K} \beta^{i}  \leq \frac{1}{1-\beta}$. Next, dividing both sides of the above inequality by $K$ and denoting the smallest gap function over the iterations from  $k=0$ to $K-1$, i.e., 
\begin{equation*}
	\mathcal{G}_{k^*} \triangleq  \min_{0\leq k \leq K-1}~\mathcal{G}(\bx_k) \leq \frac{1}{K} \sum_{k=0}^{K-1} \mathcal{G}(\bx_k), 
\end{equation*}
imply that
\begin{align}\label{eq:gap-bound-general}
    \mathcal{G}_{k^*} &\leq  \frac{\ell(\bx_0) -\ell(\bx_K)}{K \gamma}+ \frac{\gamma \bC_2 D_{\cX} \bL_{\by} \beta}{1-\beta}+ \frac{\gamma D_{\cX}^2 \rho \bC_{\bv} C_{yx}^{g} \rho}{1-\rho} +\frac{\gamma D_{\cX}^2 C_{yx}^{g} \bL_{\by} \bC_1 \beta \eta}{(1-\beta)(1-\rho)} + \frac{1}{2} \bL_{\ell}\gamma D_{\cX}^2 \nonumber \\
    &+ \frac{\bC_2  D_0^y D_{\cX} \beta}{K(1- \beta)}+ \frac{D_{\cX}  C_{yx}^{g} \rho \|{\bw_0} - \bv(\bx_0) \|}{K(1- \rho)}+ \frac{ D_{\cX} D_0^y C_{yx}^{g} \bC_1 \eta \rho^2}{K(\blue{\rho}-\beta)(1-\rho)}.
\end{align}
\qed 

\section{Proof of Corollary~\ref{cr:nonconvex_upper-bound}}\label{appen:nonconvex_coro}
    We begin the proof by using the result of the  Theorem~\ref{thm:nonconvex-upper-bound}. %and  substituting $\gamma = \frac{1}{\sqrt{K}}$.
\begin{align*}%\label{eq:gap-bound-general}
    \mathcal{G}_{k^*} &\leq  \frac{\ell(\bx_0) -\ell(\bx_K)}{K \gamma}+ \frac{\gamma \bC_2 D_{\cX} \bL_{\by} \beta}{1-\beta}+ \frac{\gamma D_{\cX}^2 \rho \bC_{\bv} C_{yx}^{g} \rho}{1-\rho} +\frac{\gamma D_{\cX}^2 C_{yx}^{g} \bL_{\by} \bC_1 \beta \eta}{(1-\beta)(1-\rho)} + \frac{1}{2} \bL_{\ell}\gamma D_{\cX}^2 \nonumber \\
    &+ \frac{\bC_2  D_0^y D_{\cX} \beta}{K(1- \beta)}+ \frac{D_{\cX}  C_{yx}^{g} \rho \|{\bw_0} - \bv(\bx_0) \|}{K(1- \rho)}+ \frac{ D_{\cX} D_0^y C_{yx}^{g} \bC_1 \eta \rho^2}{K(1-\beta)(1-\rho)}\\
    &=\cO\left(\frac{1}{K \gamma}+\frac{\gamma \bC_2 \bL_{\by} \beta}{1-\beta}+\frac{\gamma \bL_{\by} \bC_1 \beta}{(1-\beta)(1-\rho)}\right)
\end{align*}
% \begin{align*}%\label{eq:gap-bound-conv}
%     \mathcal{G}_{k^*} &\leq  \frac{\ell(\bx_0) -\ell(\bx^*)}{\sqrt{K}}+ \frac{ D_{\cX} \bL_{\by} \beta}{\sqrt{K}(1-\beta)}+ \frac{ \naz{\bC_2}D_{\cX}^2 \rho \bC_{\bv} C_{yx}^{g} \rho}{\sqrt{K}(1-\rho)} +\frac{ D_{\cX}^2 C_{yx}^{g} \bL_{\by} \bC_1 \beta \eta}{\sqrt{K}(1-\beta)(1-\rho)} + \frac{1}{2\sqrt{K}} \bL_{\ell}  D_{\cX}^2 \nonumber \\
%     &\quad + \frac{\bC_2  D_0^y D_{\cX} \beta}{K(1- \beta)}+ \frac{D_{\cX}  C_{yx}^{g} \rho \|{\bw_0} - \bv(\bx_0) \|}{K(1- \rho)}+ \frac{ D_{\cX} D_0^y C_{yx}^{g} \bC_1 \eta \rho^2}{K(1-\beta)(1-\rho)}\nonumber\\
%     &=\cO(1/\sqrt{K}). 
% \end{align*}
The desired result follows immediately from \eqref{eq:gap-bound-general} and the fact that $\ell(\bx^*)\leq \ell(\bx_K)$. Moreover, similar to the proof of Corollary \ref{cr:convex_upper-bound} we have that $\mathbf{C}_{\mathbf{v}} = \cO(\kappa_g^3)$, $\bC_1=\cO(\kappa_g^2)$, $\bL_\by=\cO(\kappa_g)$, and $\min\{1-\rho,1-\beta\} = \Omega(\frac{1}{\kappa_g})$. 
%the bound on $\mathcal{G}_k^*$ is dominated by $\frac{\ell(x_0)-\ell(x^*)}{K\gamma}$ and $\frac{\gamma D_{\mathcal{X}}^2 \rho \mathbf{C}_{\mathbf{v}} C_{yx}^g}{1-\rho}$. 
Hence, by choosing $\gamma = 1/(\kappa_g^{2.5} \sqrt{K})$, we obtain that $\mathcal{G}_k^* = \cO(\frac{1}{{K}\gamma} + \gamma \kappa_g^5) = \cO(\kappa_g^{2.5} / \sqrt{K})$, which leads to an iteration complexity of $\cO(\kappa_g^5  \epsilon^{-2})$. 

Furthermore, assuming that $\grad_y f(x,y)$ is uniformly bounded, we conclude that $C_y^f=\cO(1)$, hence, $\bC_1=\cO(\kappa_g)$ from which we have that $\cG_{k^*}=\cO(\frac{1}{K\gamma}+\gamma\kappa_g^4)$. Therefore, selecting $\gamma=1/(\kappa_g^2\sqrt{K})$ implies that $\cG_{k^*}=\cO(\kappa_g^2/\sqrt{K})$ which leads to an iteration complexity of $\cO(\kappa_g^4\epsilon^{-2})$. \qed

\section{Additional Experiments}\label{sec:additional_example}
In this section, we provide more details about the experiments conducted in section \ref{sec:numeric} as well as some additional experiments. 

\subsection{Experiment Details}\label{appen:ex_detail}
 In this section, we include more details of the numerical experiments in Section~\ref{sec:numeric}. 
 The MATLAB code is also included in the supplementary material.   

For completeness, we briefly review the update rules of SBFW~\citet{akhtar2022projection} and TTSA~\citet{hong2020two} for the setting considered in problem~\eqref{eq:bilevel}. In the following, we use $\cP_{\cX}(\cdot)$ to denote the Euclidean projection onto the set $\cX$.

Each iteration of SBFW has the following updates:
\begin{align*}
    & \by_k = \by_{k-1} - \delta_k \grad_y g (\bx_{k-1}, \by_{k-1}),\\
    &\bd_k = (1- \rho_k)(\bd_{k-1} - h(\bx_{k-1}, \by_{k-1})) + h(\bx_k, \by_k),\\
    &\bs_k = \argmin_{\bs \in \cX} \langle \bs, \bd_k \rangle,\\
    &\bx_{k+1} = (1 - \eta_k) \bx_k + \eta_k \bs_k
\end{align*}
Based on the theoretical analysis in \citet{akhtar2022projection},  $\rho_k = \frac{2}{k^{1/2}}$, $\eta_k = \frac{2}{{(k+1)}^{3/4}}$, and $\delta_k = \frac{a_0}{k^{1/2}}$ where $a_0 = \min \Big\{ \frac{2}{3 \mu_g} , \frac{\mu_g}{2 L_g^2}\Big\}$. 
Moreover, $h(\bx_k, \by_k)$ is a biased estimator of the surrogate $\ell(\bx_k)$ %gradient $\nabla \ell (\bx_k, \by_k) = \nabla_{x}f(\bx_k,\by_k) - \nabla_{yx}g(\bx_k,\by_k) [\nabla_{yy}g(\bx_k,\by_k)]^{-1} \nabla_{y}f(\bx_k,\by_k)$ 
which can be computed as follows
\begin{align*}
    h(\bx_k, \by_k) = \nabla_{x}f(\bx_k,\by_k) - M(\bx_k, \by_k) \nabla_{y}f(\bx_k,\by_k),
\end{align*}
where the term $ M(\bx_k, \by_k)$ is a biased estimation of $[\nabla^2_{yy}g(\bx_k,\by_k)]^{-1} $ with bounded variance whose explicit form is 
\begin{align*}
    M(\bx_k, \by_k) = \nabla^2_{yx}g(\bx_k,\by_k) \times \Big[ \frac{k}{L
    _g} \Pi_{i=1}^{l}\Big( I - \frac{1}{L_g}\grad^2_{yy} g(\bx_k, \by_k)\Big) \Big],
\end{align*}
and $l\in \{1,\hdots, k\}$ is an integer selected uniformly at random.  

The steps of TTSA algorithm are given by 
\begin{align*}
    &\by_{k+1} = \by_k - \beta h_k ^ g,\\
    &\bx_{k+1} = \cP_{\cX} ( \bx_k - \alpha h_k ^f),\\
    &h_k^g =\grad_y g (\bx_k, \by_k),\\
    &h_k^f =\nabla_{x}f(\bx_k,\by_k) - \nabla^2_{yx}g(\bx_k,\by_k) \times \Big[ \frac{t_{max}(k) c_h}{L_g} \Pi_{i=1}^{p}\Big( I - \frac{c_h}{L_g}\grad^2_{yy} g (\bx_k, \by_k)\Big) \Big] \nabla_{y}f(\bx_k,\by_k),
\end{align*}
where based on the theory we define  $L = L_x^f + \frac{L_y^f C_{yx}^g}{\mu_g} + C_y^f \Big( \frac{L_{yx}^g}{\mu_g} + \frac{L_{yy}^g C_{yx}^g}{\mu_g^2}\Big)$, and $L_y = \frac{C_{yx}^{g}}{\mu_g}$, then set  $\alpha = \min\Big\{ \frac{\mu_g^2}{8 L_y L L_g^2}, \frac{1}{4 L_y L} K^{-3/5} \Big\}$, $\beta = \min\Big\{ \frac{\mu_g}{ L_g^2}, \frac{2}{\mu_g} K^{-2/5} \Big\}$, $t_{max}(k) = \frac{L_g}{\mu_g}\log(k+1)$, $p \in \{0, \hdots, t_{max}(k)-1\}$, and $c_h \in (0,1]$.

\subsection{Toy Example}\label{app:toy}
  
Here we consider a variation of coreset problem in a two-dimensional space to illustrate the numerical stability of our proposed method.  
Given a point $x_0\in\reals^2$, the goal is to find the closest point to $x_0$ such that under a linear map it lies within the convex hull of given points $\{x_1,x_2,x_3,x_4\}\subset \reals^2$. Let $A\in\reals^{2\times 2}$ represents the linear map, $X \triangleq [x_1, x_2, x_3, x_4] \in \reals^{2 \times 4}$, and $\Delta_{4} \triangleq \{ \lambda \in \reals^4 | \langle \lambda, 1\rangle = 1 , \lambda \geq 0 \} $ be the standard simplex set. 
This problem can be formulated as the following bilevel optimization problem
\begin{align}\label{ex:toy}
    \min_{\lambda \in \Delta_{4}} \frac{1}{2}\| \theta(\lambda) - x_0\|^2 \quad  \hbox{s.t.}\quad  \theta(\lambda) \in \argmin_{\theta \in \reals^{2}}  \frac{1}{2}\| A \theta - X\lambda \|^2. 
\end{align}

We set the target $x_0 = (2,2)$ and 
 choose starting points as $\theta_0 = (0,0)$ and $\lambda_0 = \mathbf{1}_4/4$. 
 %then the optimal solution of the problem~\eqref{ex:toy} is  $\theta^* = (0.735, 1.005) ^T$ and $\lambda^* = (0.4655, 0.5345, 0, 0)^T$. 
 We implemented our proposed method and compared it with SBFW \cite{akhtar2022projection}. It should be noted that in the SBFW method, they used a biased estimation for $[\nabla^2_{yy}g(\lambda,\theta)]^{-1}=(A^\top A)^{-1}$ whose bias is upper bounded by $\frac{2}{\mu_g}$ (see Lemma 3.2 in \cite{ghadimi2018approximation}). Figure~\ref{fig:toy1} illustrates the
iteration trajectories of both methods for $\mu_g = 1$ and $K=10^2$. The step-sizes for both methods are selected as suggested by their theoretical analysis. We observe that our method converges to the optimal solution while SBFW fails to converge. 
This situation for SBFW exacerbates for smaller values of $\mu_g$. 
% while our method shows a consistent and robust behavior.
% (see Appendix~\ref{sec:additional_example}  for more details and experiment). 

 \begin{figure}
     \centering
     \includegraphics[scale=0.35]{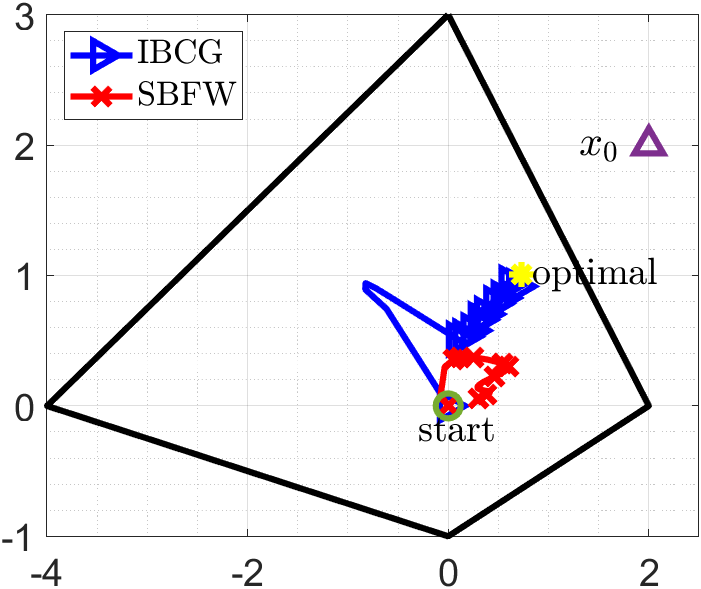}
     % \quad
     % \includegraphics[scale=0.18]{norm_g.png}
     \quad
     \includegraphics[scale=0.35]{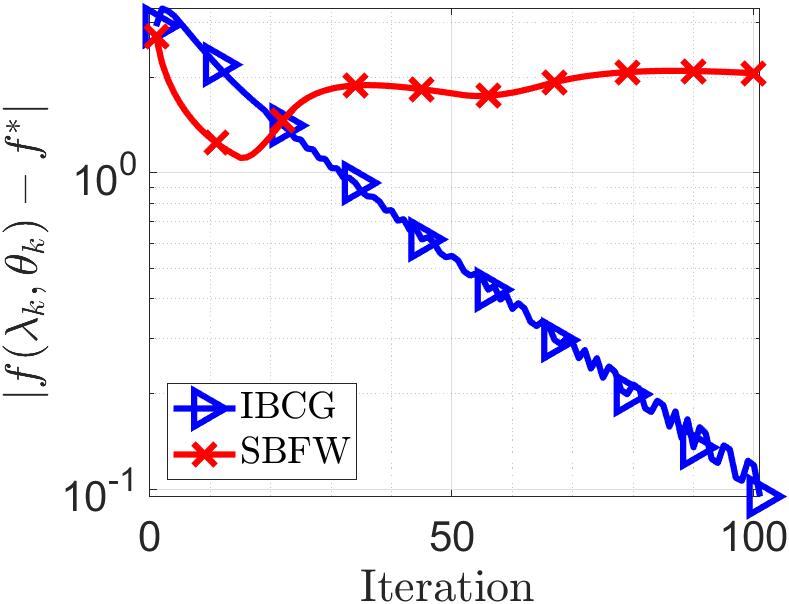}
     \caption{The performance of IBCG (blue) vs SBFW (red) on Problem \eqref{ex:toy} when $\mu_g =1$. Plots from left to right are trajectories of $\theta_k$ 
     % $\|\nabla g_{\theta}(\lambda_k,\theta_k)\|$, 
     and $f(\lambda_k,\theta_k)- f^*$.}
     \label{fig:toy1}
 \end{figure}

 % We implemented our proposed method and compared it with SBFW \citep{akhtar2022projection}. It should be noted that in the SBFW method, they used a biased estimation for $(\nabla^2_{yy}g(\lambda,\theta))^{-1}=(A^\top A)^{-1}$ whose bias is upper bounded by $\frac{2}{\mu_g}$ (see \citep{ghadimi2018approximation}[Lemma 3.2]).
Fig.~\ref{fig:toy2} illustrates the
iteration trajectories of both methods for $\mu_g = 0.1$ and $K=10^3$ in which we also included SBFW method whose Hessian inverse matrix is explicitly provided in the algorithm. The step-sizes for both methods are selected as suggested by their theoretical analysis. 
% In Figure~\ref{fig:toy1}, we observed that our method converges to the optimal solution while SBFW fails to converge.
% This situation for SBFW exacerbates with smaller values of $\mu_g$. Moreover, it can be observed in Figure~\ref{fig:toy2}  
% In fact, using the Hessian inverse as well as tuning the step-sizes their method converges to the optimal solution while our method shows a consistent and robust behavior. This result can be seen in Figure~\ref{fig:toy3}. 
Despite incorporating the Hessian inverse matrix in the SBFW method, the algorithm's effectiveness is compromised by excessively conservative step-sizes, as dictated by the theoretical result. Consequently, the algorithm fails to converge to the optimal point effectively. Regarding this issue, we tune their step-sizes, i.e., scale the parameter $\delta$ and $\eta$ in their method by a factor of 5 and 0.1, respectively. By tuning the parameters we can see in Fig.~\ref{fig:toy3} that the SBFW with Hessian inverse matrix algorithm has a better performance and converges to the optimal solution. In fact, using the Hessian inverse as well as tuning the step-sizes their method converges to the optimal solution while our method always shows a consistent and robust behavior.

 \begin{figure}
     \centering
     \includegraphics[scale=0.35]{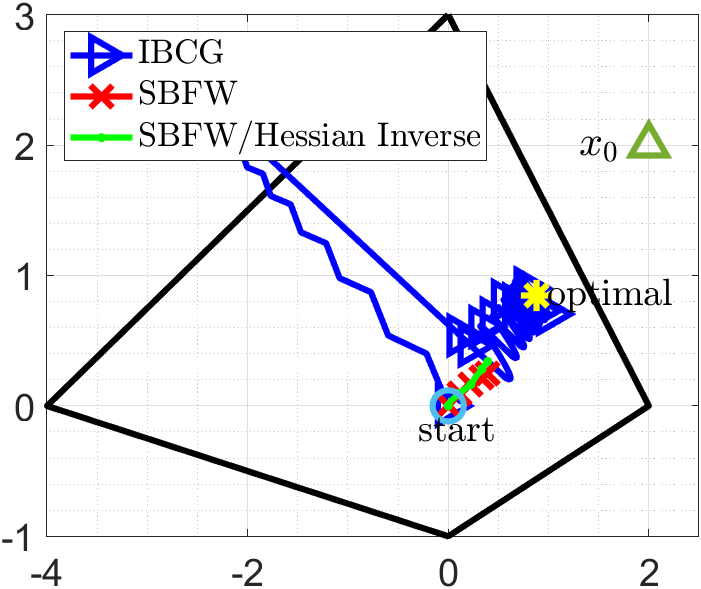}
     % \quad
     % \includegraphics[scale=0.18]{norm_g.png}
     \quad
     \includegraphics[scale=0.35]{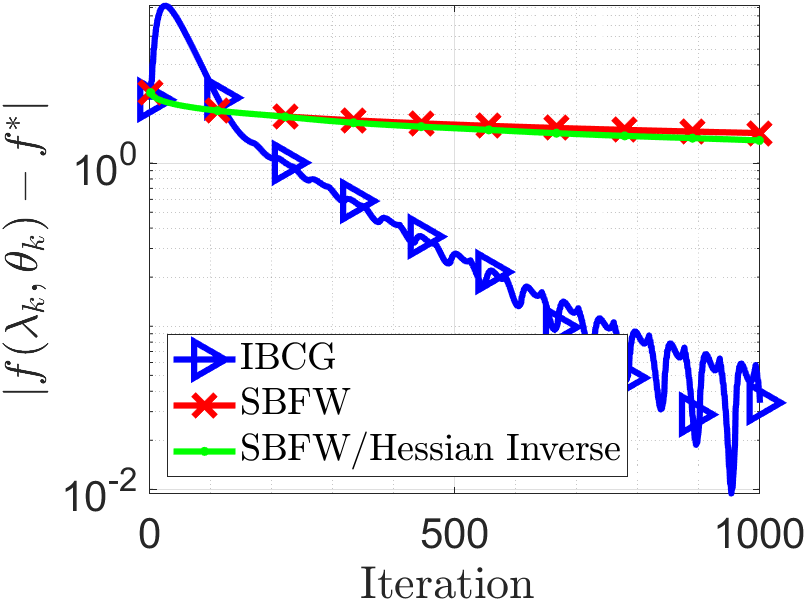}
     \caption{The performance of IBCG (blue) vs SBFW (red) and  SBFW with Hessian inverse (green) on Problem \eqref{ex:toy} when $\mu_g = 0.1$. Plots from left to right are trajectories of $\theta_k$ 
     % $\|\nabla g_{\theta}(\lambda_k,\theta_k)\|$, 
     and $f(\lambda_k,\theta_k)- f^*$.}
     \label{fig:toy2}
 \end{figure}

 \begin{figure}
     \centering
     \includegraphics[scale=0.35]{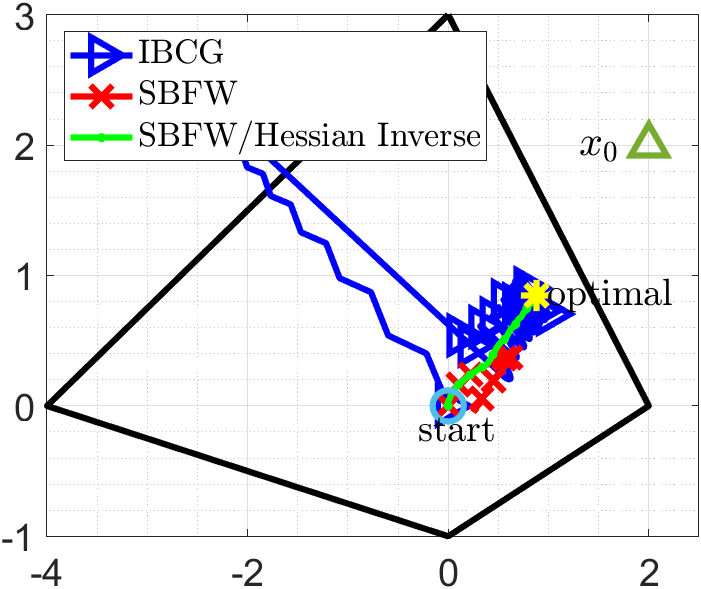}
     % \quad
     % \includegraphics[scale=0.18]{norm_g.png}
     \quad
     \includegraphics[scale=0.35]{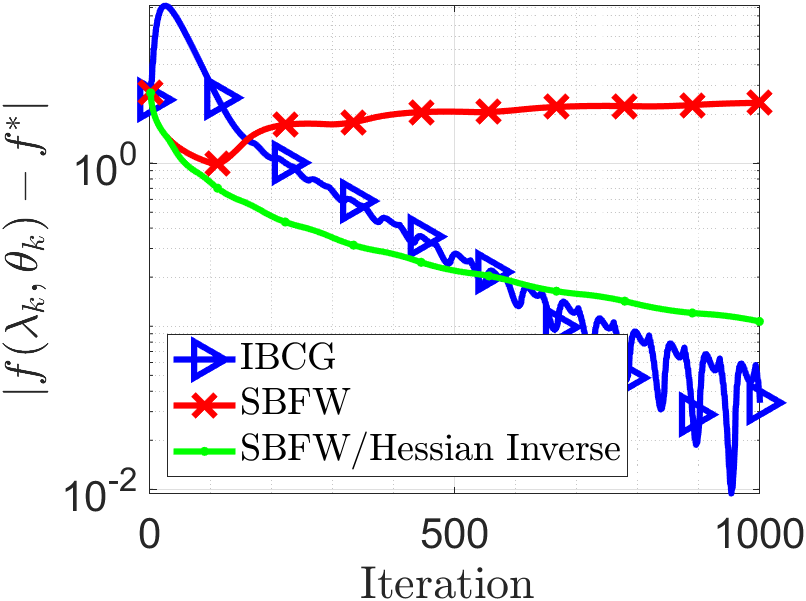}
     \caption{The performance of IBCG (blue) vs SBFW (red) and  SBFW with Hessian inverse (green) on Problem \eqref{ex:toy} when $\mu_g = 0.1$ and the SBFW parameters are tuned. Plots from left to right are trajectories of $\theta_k$ 
     % $\|\nabla g_{\theta}(\lambda_k,\theta_k)\|$, 
     and $f(\lambda_k,\theta_k)- f^*$.}
     \label{fig:toy3}
 \end{figure}

\subsection{Matrix Completion with Denoising}\label{appen: matrix_comp}
\subsubsection{Synthetic dataset}

\noindent\textbf{Dataset Generation.} 
We create an observation matrix $M = \hat{X}+E$. In this setting $\hat{X} = WW^{T}$  where $W \in \reals^{n \times r}$ containing normally distributed independent entries, and $E = \hat{n}(L+L^{T})$ is a noise matrix where $L \in \reals^{n \times n}$ containing normally distributed independent entries and $\hat{n} \in (0,1)$ is the noise factor. During the simulation process, we set $n= 250$, $r = 10$, and $\alpha = \| \hat{X}\|_{*}$. 
%The way we generate matrix $\bM$ is throughly explained in section~\ref{num: matrix_comp}.

\noindent\textbf{Initialization.} All the methods start from the same initial point $\bx_0$ and $\by_0$ which are generated randomly. We terminate the algorithms either when the maximum number of iterations $K_\mathrm{max}=10^4$ or the maximum time limit $T_\mathrm{max} = 2 \times 10^2$ seconds are achieved.  

\noindent\textbf{Implementation Details.} For our method IBCG, we choose the step-sizes as $\gamma = \frac{1}{4\sqrt{K}}$ to avoid instability due to large initial step-sizes, and set $\alpha=2/(\mu_g+L_g)$ and $\eta=0.9\times\frac{1-\beta}{\mu_g}$.  We tuned the step-size $\eta_k$ in the SBFW method by multiplying it by a factor of 0.8, and for the TTSA method, we tuned their step-size $\beta$ by multiplying it by a factor of 0.25.

 \vspace{2mm}
\subsubsection{Real dataset}
In order to emphasize the importance of projection-free bilevel algorithms in practical applications, we conducted further experiments using a larger dataset known as MovieLens 1M. This dataset consists of 1 million ratings provided by 6000 individuals for a total of 4000 movies. 
In Figure~\ref{fig:matrixcom_1M} the inferior performance of TTSA algorithm in actual computation time, especially when dealing with large datasets becomes more evident. The observed difference can be attributed to the utilization of the projection operation in contrast to the projection-free algorithms. %like IBCG and SBFW.
TTSA requires performing projections over nuclear norm at each iteration which is computationally expensive due to the
computation of full singular value decomposition. In contrast, projection-free algorithms IBCG and SBFW solve a linear minimization at each
iteration, which only requires the computation of singular vectors corresponding to the largest singular value.
On the other hand, considering the slow convergence rate of SBFW, when the size of the dataset increases, the improved performance of our proposed method becomes more evident compared to SBFW.
\begin{figure}
     \centering
     \includegraphics[scale=0.32]{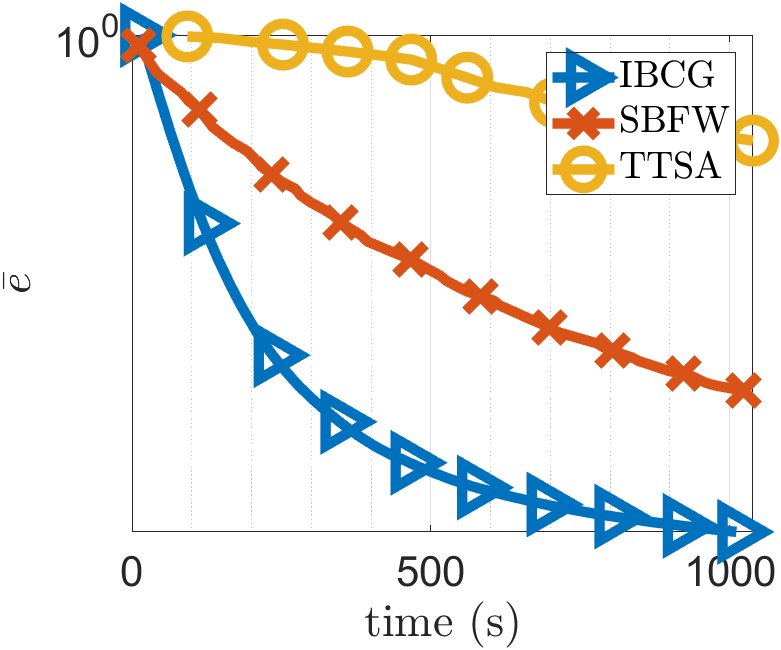}
     \quad
     \includegraphics[scale=0.32]{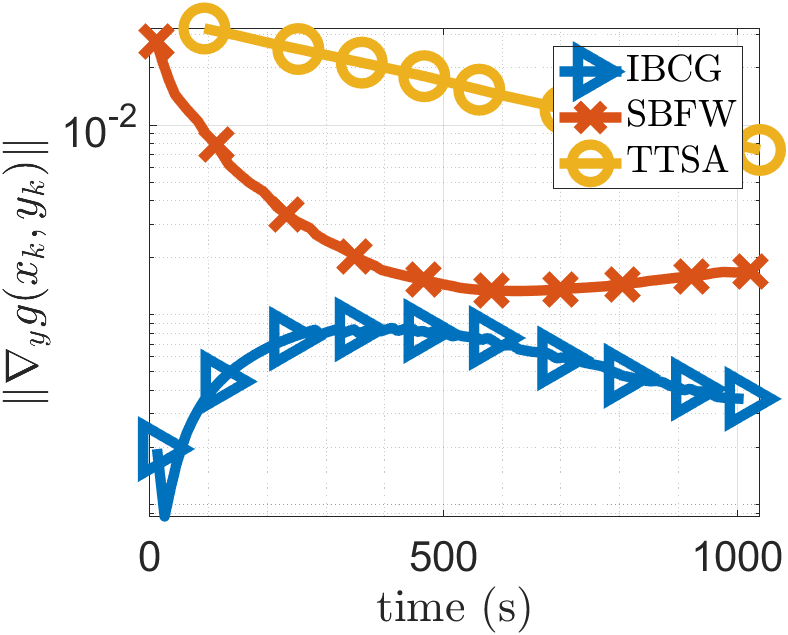}
     \quad
     \includegraphics[scale=0.32]{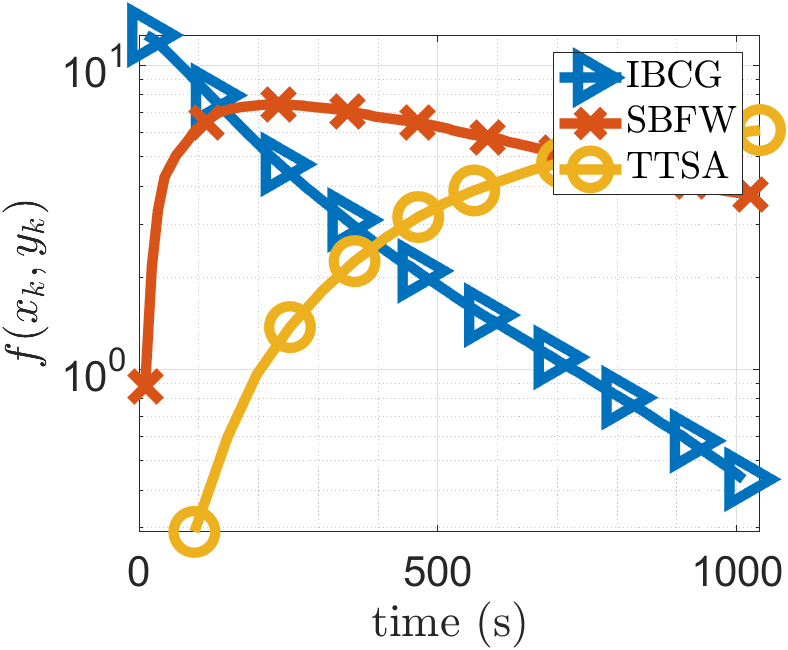}
     \caption{The performance of IBCG (blue) vs SBFW (red) and  TTSA (yellow) on Problem \eqref{ex:Matrix_comp} for real dataset. Plots from left to right are trajectories of normalized error $(\Bar{e})$, 
      $\|\nabla_y g(\bx_k,\by_k)\|$, 
     and $f(\bx_k,\by_k)$ over time.}
     \label{fig:matrixcom_1M}
 \end{figure}

 Moreover, in the following, we utilized the MovieLens 100k dataset to implement matrix completion with denoising examples for different step-sizes. These experiments will be designed to explore how different step-size selections impact the performance of the IBCG algorithm. 
 % We aim to cover a range of scenarios, from small to large step sizes, to thoroughly assess the algorithm's robustness and efficiency across diverse settings. 
 We fix the step-sizes $\alpha=2/(\mu_g+L_g)$ and $\eta=0.5\times\frac{1-\beta}{\mu_g}$ and systematically alter $\gamma = c_1 \times \frac{1}{\sqrt{K}}$  with constants $c_1 \in \{ 0.75, 0.5, 0.25, 0.1\}$ as depicted in Fig.~\ref{fig:matrixcom_etafixed}. 
 We observe that larger values of $\gamma$ directly affect the performance of the algorithm. 
 This observation matches with our theoretical result as demonstrated in Lemma \ref{lem:nu_est_K}. In particular, the error of approximating the lower-level solution and its Jacobian is directly related to the step-size $\gamma$ and larger values of $\gamma$ contributing to larger errors affecting the upper-level objective value.
 %the experiments demonstrate that the algorithm can maintain its effectiveness even with different scaling factors. Specifically, the algorithm's performance did not significantly deteriorate even when the step-size $\gamma$ was scaled down in relation to the square root of the iteration count $K$. This indicates an intrinsic robustness to step-size selection within the tested range, which is highly beneficial for real-world applications where step-size tuning can be challenging. Although for specific values of $\gamma$ the performance of the IBCG is better, the overarching trend in Figure~\ref{fig:matrixcom_etafixed} confirms that the IBCG algorithm can be effectively applied across different operational scales. 
 % This approach will also allow us to fine-tune the step-size strategy for optimal performance, which is essential for practical implementations of the IBCG algorithm.

\begin{figure}
     \centering
     \includegraphics[scale=0.32]{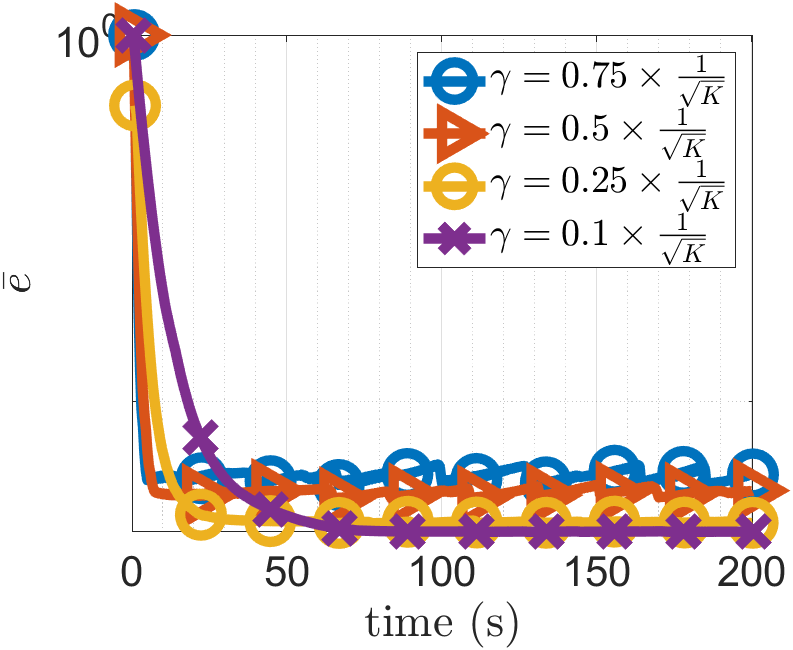}
     \quad
     \includegraphics[scale=0.32]{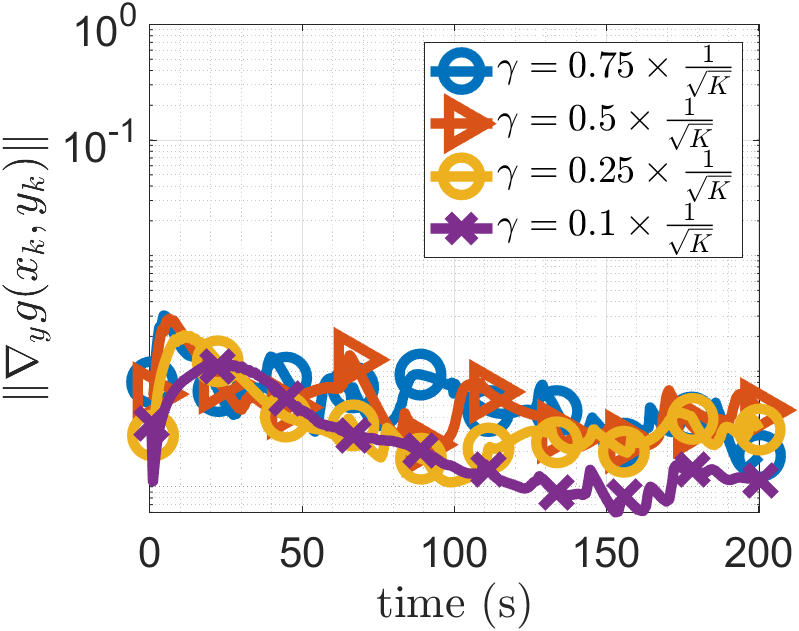}
     \quad
     \includegraphics[scale=0.32]{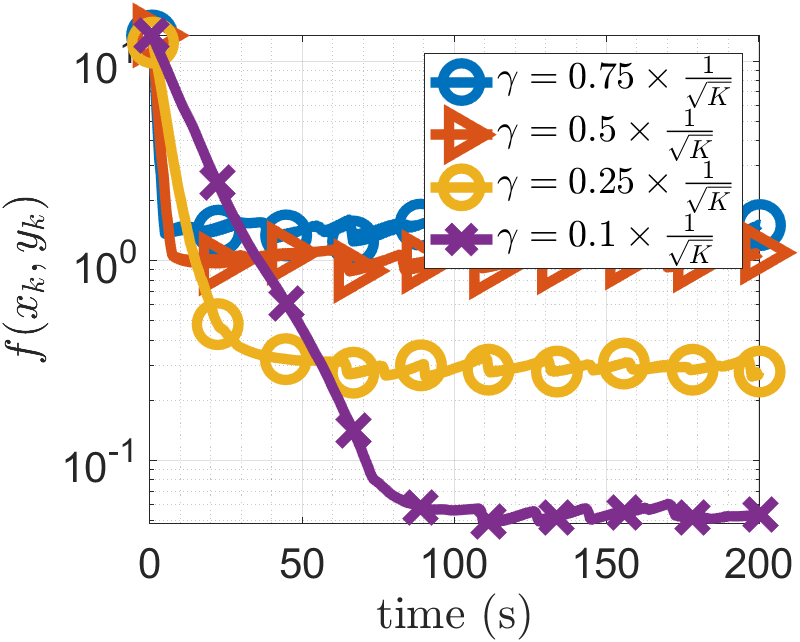}
     \caption{The performance of IBCG on Problem \eqref{ex:Matrix_comp} for fixed step-sizes $\eta$ and $\alpha$.  Plots from left to right are trajectories of normalized error $(\Bar{e})$, 
      $\|\nabla_y g(\bx_k,\by_k)\|$, 
     and $f(\bx_k,\by_k)$ over time.}
     \label{fig:matrixcom_etafixed}
 \end{figure}

\begin{figure}
     \centering
     \includegraphics[scale=0.32]{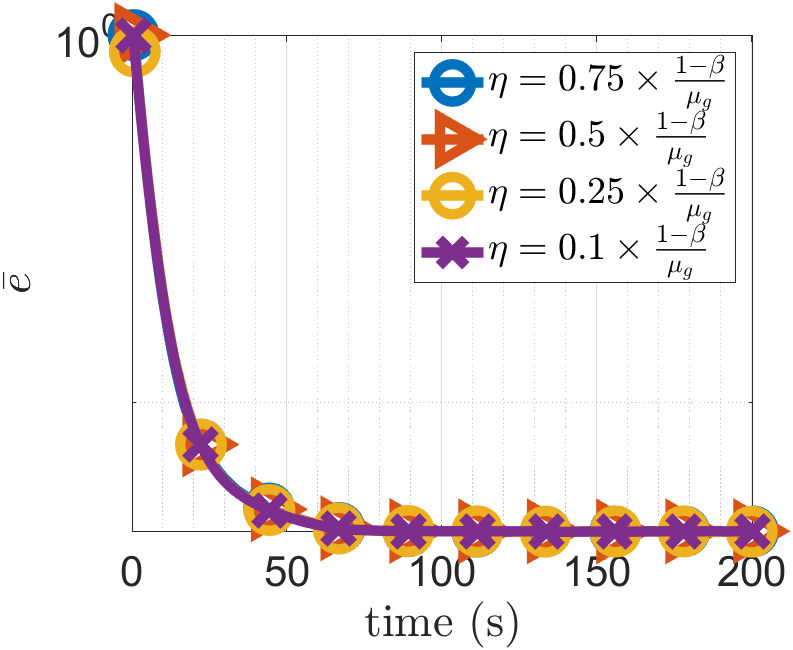}
     \quad
     \includegraphics[scale=0.32]{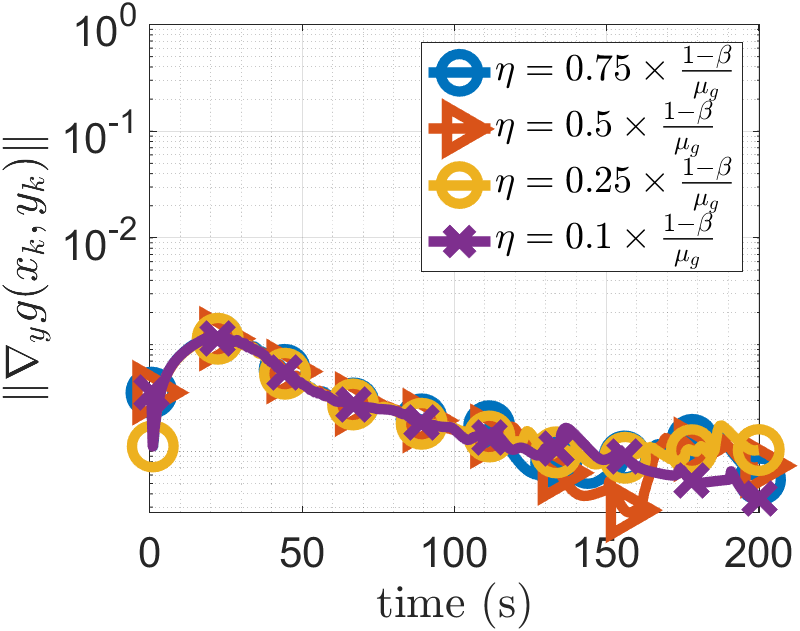}
     \quad
     \includegraphics[scale=0.32]{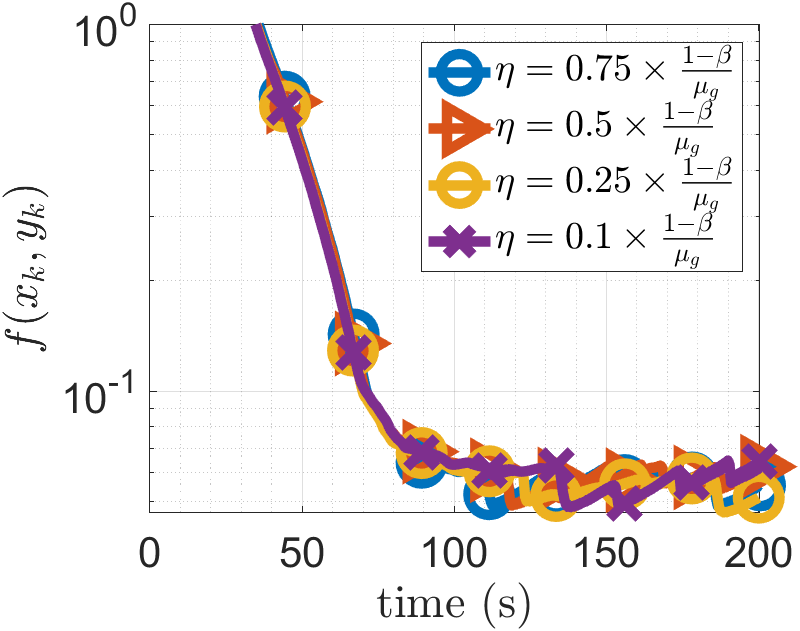}
     \caption{The performance of IBCG on Problem \eqref{ex:Matrix_comp} for fixed step-sizes $\gamma$ and $\alpha$. Plots from left to right are trajectories of normalized error $(\Bar{e})$, 
      $\|\nabla_y g(\bx_k,\by_k)\|$, 
     and $f(\bx_k,\by_k)$ over time.}
     \label{fig:matrixcom_gammafixed}
 \end{figure}
 
In Fig.~\ref{fig:matrixcom_gammafixed}, we fixed the step-sizes $\alpha=2/(\mu_g+L_g)$ and $\gamma = 0.1 \times \frac{1}{\sqrt{K}}$ and changed the value of step-size $\eta= c_2 \times\frac{1-\beta}{\mu_g}$ with constants $c_2 \in \{ 0.75, 0.5, 0.25, 0.1\}$. The performance of the IBCG is robust due to the various values of step-size $\eta$. This indicates that the choice of $\eta$ does not significantly affect the convergence rate, suggesting that the IBCG method is not overly sensitive to this parameter within the tested range. The algorithm achieves comparable accuracy levels in the end, regardless of the initial choice of $c_2$, signifying a level of stability that can be beneficial in practical applications where the optimal step-size may not be known at first.

\subsection{Multi-Task Learning}\label{appen:MTL}

\noindent\textbf{Dataset 1.} In this experiment, we consider a collection of datasets: \texttt{bodyfat}, \texttt{housing}, \texttt{mg}, \texttt{mpg}, and \texttt{space-ga} from LibSVM library \cite{CC01a}. We consider each dataset as distinct data points for each task. To standardize the feature space across these datasets, we extend the number of features in each by adding columns of zeros to the datasets that originally had fewer features. Furthermore, all datasets are scaled to fit within the range of $[-1, 1]$.  
% We assign 60\% of the data points of each task $\cT_i$ as the training set $(\bA_i^{\mathrm{tr}},\bb_i^{\mathrm{tr}})$, 20\% as the validation set $(\bA_i^{\mathrm{val}},\bb_i^{\mathrm{val}})$ and the rest as the test set $(\bA_i^{\mathrm{test}},\bb_i^{\mathrm{test}})$.

\noindent\textbf{Initialization.} For both methods we choose $\lambda_1 = \lambda_2 = 0.1$ and start from the same initial point $\Omega_0$ and $W_0$. We terminate the algorithms either when the maximum number of iterations $K_\mathrm{max}=10^4$ or the maximum time limit $T_\mathrm{max} = 100$ seconds are achieved.

\noindent\textbf{Implementation Details.} For our method IBCG, we choose the step-sizes as $\gamma = \frac{1}{10^{4} \times \sqrt{K}}$, and set $\eta=0.9\times\frac{1-\beta}{\mu_g}$.  We tuned the step-sizes $\eta_k,\gamma_k, \rho_k$ in the SBFW method by multiplying all by a factor of $10^{3}$. Moreover, we scale the parameter $a_0$ by multiplying it by a factor of 100. Empirically, we observe that these modifications lead to faster convergence in SBFW.

\begin{table}[]\scriptsize
\centering
\begin{tabular}{ccccccc}
\hline
\multirow{2}{*}{\textbf{Model}} & \multicolumn{5}{c}{\textbf{Dataset1}}           & \multirow{2}{*}{\textbf{Average}} \\ \cline{2-6}
                                & bodyfat & housing & mg     & mpg    & space-ga &                                   \\ \hline
\textbf{Individual regression}  & 0.1771  & 7.6547  & 0.6333 & 1.7412 & 0.5232   & 2.1459                            \\ \hline
\textbf{MTL\_IBCG}              & 0.0599  & 1.1630  & 0.0480 & 3.2205 & 0.0245   & 0.9032                           \\ \hline
\textbf{MTL\_SBFW}              & 0.0783  & 0.9675  & 0.0540 & 3.4654 & 0.0253   & 0.9181                            \\ \hline
\end{tabular}
\caption{Test error of applying Ridge regression on each dataset vs test error of IBCG and SBFW methods for multi-task learning problem presented in \eqref{ex:MTL_bi}}
\label{tab:MTL_err1} 
\end{table}

Furthermore, we compare the performance of the MTL model considered in \eqref{ex:MTL_bi} over the test dataset after training with IBCG and SBFW algorithms and compared with the performance of each dataset individually trained by a ridge regression (RR) model. The results of the test error are depicted in Table~\ref{tab:MTL_err1}. The considered bilevel MTL model improves the test accuracy for 4 out of 5 datasets considered. In fact, the average test error of RR models is 2.1459 while our method achieves 0.9032. Moreover, comparing the performance of our proposed method with SBFW we observe that the test error for IBCG is lower in most of the datasets with an average of 0.9032 while SBFW leads to an average test error of 0.9181. 
%Table~\ref{tab:MTL_err1} compares the performance of regression models on several datasets evaluating the test errors across three different modeling approaches. The first approach, individual regression, treats each dataset as a separate task and applies ridge regression independently to each. The resulting test errors for individual regression show considerable variation, with values ranging from 0.1771 to 7.6547 and an overall average test error of 2.1459. The second approach, MTL$\_$IBCG using the IBCG optimization algorithm to train the model, leverages potential correlations among the datasets to improve prediction accuracy. This method achieves significantly lower test errors, indicating that the shared information across tasks effectively enhances the model's generalization capability, with test errors substantially reduced to a range between 0.0245 and 1.1630, leading to an impressive average test error of 0.9032. The third approach employs MTL with the SBFW optimization algorithm, which also utilizes the shared structure of the tasks but yields slightly higher test errors than IBCG and an average test error of 0.9181. Overall, the results underscore the efficacy of MTL approaches, particularly our method using IBCG, in outperforming traditional individual regression models by effectively harnessing the commonalities between different but related tasks.

\noindent\textbf{Dataset 2.} We consider the 1979 National Longitudinal Survey of Youth (NLSY79)\footnote{\href{https://www.bls.gov/nls/nlsy79.htm}{NLSY79 Data Overview}} dataset which consists of a data matrix $\bA\in \reals^{{n \times d}}$ {with $n = 6213$ instances and $d=19$ attributes}
and an outcome vector $\bb\in \reals^{n}$. We assign 60\% of the dataset as the training set $(\bA^{\mathrm{tr}},\bb^{\mathrm{tr}})$, 20\% as the validation set $(\bA^{\mathrm{val}},\bb^{\mathrm{val}})$ and the rest as the test set $(\bA^{\mathrm{test}},\bb^{\mathrm{test}})$. To align the dataset with the structure of our problem, we partition the data matrix \( \mathbf{A} \in \mathbb{R}^{n \times d} \) and the associated response vector \( \mathbf{b} \in \mathbb{R}^{n} \) into \( T =10\) distinct segments. This ensures that each task is allocated \( n_i \) data points i.e $\bA_i \in \reals^{{n_i \times d}}, \mathbf{b_i} \in \mathbb{R}^{n_i}$, with the cumulative total across all tasks \( \sum_{i=1}^T n_i \) equating to \( n \).

\noindent\textbf{Initialization.} For both methods we choose $\lambda_1 = \lambda_2 = 0.1$ and start from the same initial point $\Omega_0$ and $W_0$. We terminate the algorithms either when the maximum number of iterations $K_\mathrm{max}=10^4$ or the maximum time limit $T_\mathrm{max} = 100$ seconds are achieved.

\noindent\textbf{Implementation Details.} For our method IBCG, we choose the step-sizes as $\gamma = \frac{1}{10^{8} \times \sqrt{K}}$ to avoid instability due to large initial step-sizes, and set $\alpha= \frac{2}{(\mu_g+L_g)}$ and $\eta=0.9\times\frac{1-\beta}{\mu_g}$.  We tuned the step-sizes $\eta_k,\gamma_k, \rho_k$ in the SBFW method by multiplying all by a factor of 100. Moreover, we scale the parameter $a_0$ by multiplying it by a factor of 10. Empirically, we observe that these modifications lead to a better performance of SBFW in terms of both gradient norm of the lower-level problem and upper-level objective function value.

\begin{figure*}[ht!]
     \centering
     \includegraphics[scale=0.2]{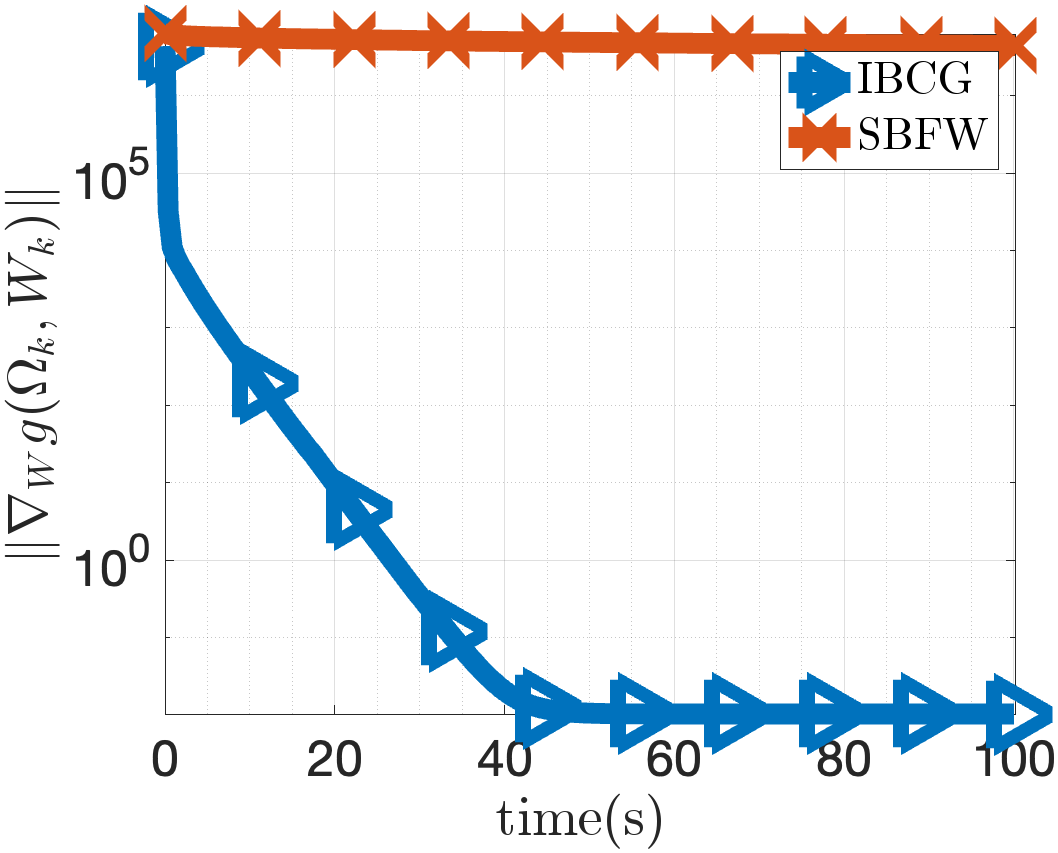}
     \includegraphics[scale=0.2]{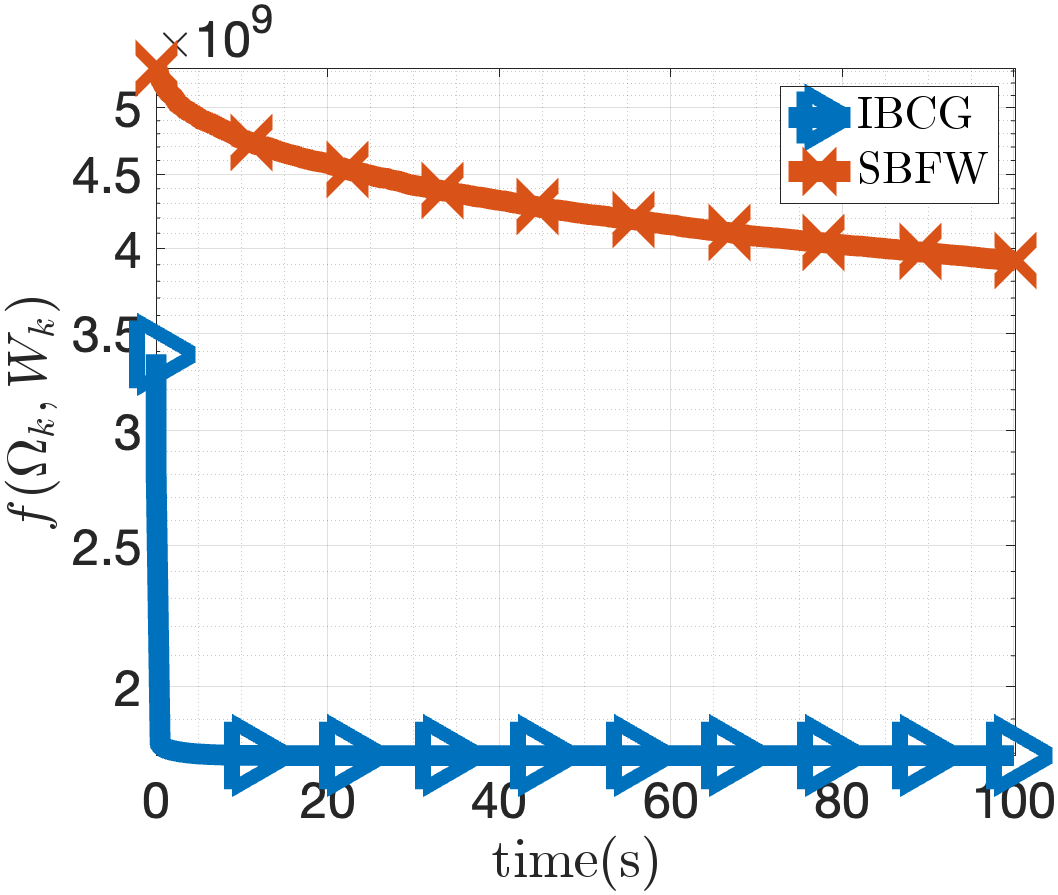}
     \qquad
     \includegraphics[scale=0.2]{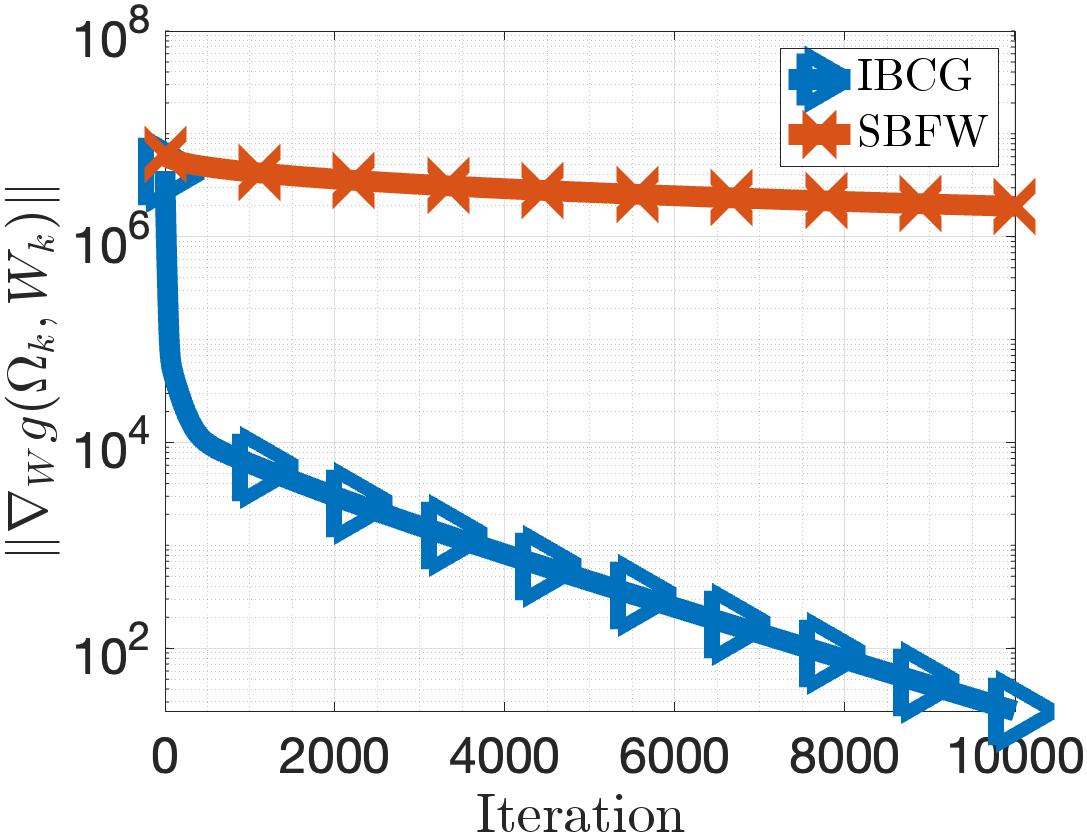}
     \includegraphics[scale=0.2]{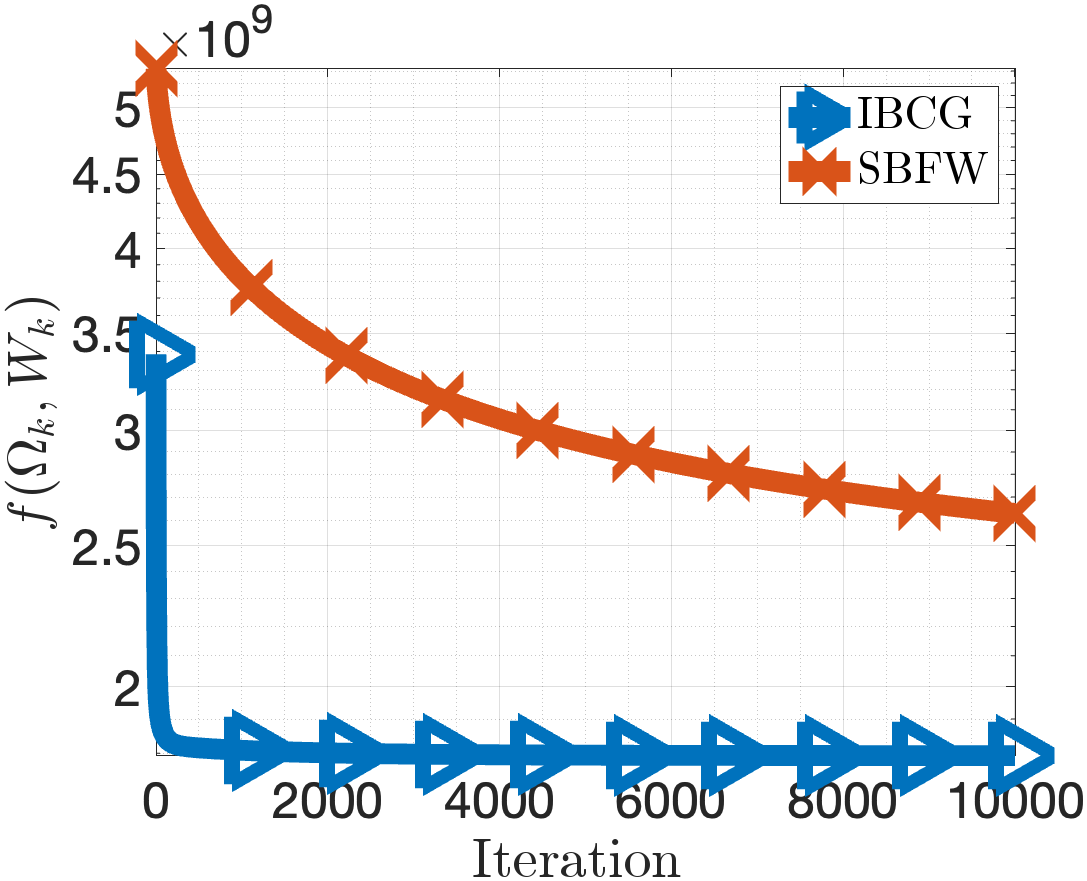}
     \vspace{-2mm}
     \caption{Performance of IBCG  vs SBFW on problem \eqref{ex:MTL_bi} for real dataset2. Plots from left to right: 
      $\|\nabla_W g(\Omega_k,W_k)\|$, 
     and $f(\Omega_k,W_k)$ over time and fixed number of iterations, respectively.}
     \label{fig:MTL_1}
     \vspace{-3mm}
 \end{figure*}

The performance comparison between our IBCG method and the SBFW method for problem~\eqref{ex:MTL_bi} on real dataset 2 is depicted in Fig.~\ref{fig:MTL_1}. In terms of lower-level gradient norm, IBCG shows a rapid initial decrease and then, indicates a fast convergence rate, however, SBFW exhibits a slower descent. The upper-level function value plots further corroborate these findings, with IBCG achieving a substantially lower function value quickly, demonstrating its superior performance. %SBFW, while also improving the function value, does so at a slower rate and does not achieve values as low as IBCG. 
%The iteration plots illustrate the progression over a set number of iterations, and the time plots show the performance against the fixed running time.
%, both reinforcing IBCG's advantage in terms of both convergence speed and solution quality.

Similar to the experiment for Dataset1, we compare the performance of the MTL model in \eqref{ex:MTL_bi} over the test dataset after training with IBCG and SBFW algorithms and compared with the performance of each partition individually trained by an RR model. The results of the test error are depicted in Table~\ref{tab:MTL_err2}. The average test error of RR models is 7.3027 while our method achieves 1.1830. Moreover, comparing the performance of our proposed method with SBFW we observe that the test error for IBCG is lower in all of the datasets with an average of 1.1830 while SBFW leads to an average test error of 1.7759.

% compares the performance of regression models on several datasets evaluating the test errors across three different modeling approaches. The individual regression approach yields higher test errors, indicating a less accurate model on average across all datasets, with an average test error of 7.3027. Conversely, the IBCG method significantly improves upon this, showcasing its efficacy with a substantially lower average test error of 1.1830. This demonstrates the advantage of leveraging shared information across tasks in multi-task learning settings. The SBFW method, while outperforming the individual regression, does not achieve errors as low as those attained by IBCG, with an average test error of 1.7759. These results underscore the effectiveness of the IBCG approach in this context, suggesting that it provides a more robust model for MTL problem as formulated in the study.

 \begin{table}[]\scriptsize
\centering
\begin{tabular}{cccccccccccc}
\hline
\multirow{2}{*}{\textbf{Model}} & \multicolumn{10}{c}{\textbf{Dataset2}}                                                   & \multirow{2}{*}{\textbf{Average}} \\ \cline{2-11}
                                & 1      & 2      & 3      & 4      & 5      & 6      & 7      & 8      & 9      & 10     &                                   \\ \hline
\textbf{Individual regression}  & 5.620 & 5.852 & 14.358 & 2.993 & 6.910 & 13.590 & 14.463 & 2.444 & 3.804 & 2.993 & 7.3027                            \\ \hline
\textbf{MTL\_IBCG}              & 1.2853 & 1.0021 & 1.1868 & 1.1131 & 1.4108 & 0.9580 & 1.1280 & 1.1849 & 1.3763 & 1.1845 & 1.1830                            \\ \hline
\textbf{MTL\_SBFW}              & 1.7794 & 1.6111 & 1.8025 & 1.5937 & 1.9606 & 1.4405 & 1.7451 & 1.8579 & 2.0640 & 1.9039 & 1.7759                            \\ \hline
\end{tabular}
\caption{Test error of applying Ridge regression on each dataset vs test error of IBCG and SBFW methods for multi-task learning problem presented in \eqref{ex:MTL_bi}}
\label{tab:MTL_err2}
\end{table}
\end{document}